\documentclass[11pt, a4paper, reqno]{amsart}%
\usepackage{epsfig,subfigure}
\usepackage[utf8]{inputenc}
\usepackage{amssymb}
\usepackage{amsfonts}
\usepackage{amsthm}
\usepackage{amsmath,multirow,rotating}
\usepackage[a4paper, margin=3.5cm]{geometry}

\usepackage[round]{natbib}

\usepackage{bm}
\usepackage[colorlinks=true,linkcolor=blue,citecolor=blue,pdfborder={0 0 0}]{hyperref}
\usepackage{color}
\usepackage{paralist}
\usepackage{booktabs}

\newtheorem{theorem}{Theorem}[section]
\newtheorem{proposition}[theorem]{Proposition}
\newtheorem{lemma}[theorem]{Lemma}
\newtheorem{corollary}[theorem]{Corollary}

\theoremstyle{definition}

\theoremstyle{remark}
\newtheorem{remark}[theorem]{Remark}
\newtheorem{example}[theorem]{Example}

\numberwithin{equation}{section}

\newcommand{\Exp}{\operatorname{E}}
\newcommand{\Cov}{\operatorname{Cov}}

\newcommand{\Var}{\operatorname{Var}}

\renewcommand{\mathbf}{\boldsymbol}

\newcommand{\eps}{\varepsilon}

\newcommand{\N}{\mathbb N}
\newcommand{\R}{\mathbb R}
\newcommand{\Z}{\mathbb Z}
\newcommand{\scs}{\scriptscriptstyle}
\newcommand{\dto}{\rightsquigarrow}
\newcommand{\Sb}{\mathbb S}
\newcommand{\Zb}{\mathbb Z}
\newcommand{\ip}[1]{\lfloor#1\rfloor}
\newcommand{\Prob}{{\operatorname{P}}}

\newcommand{\Fb}{\mathbb F}

\newcommand{\I}{\mathbf{1}}
\newcommand{\wto}{\rightsquigarrow}
\newcommand{\pto}{\stackrel{\mathbb{P}}{\longrightarrow}}
\newcommand{\Nor}{\mathcal{N}}
\newcommand{\op}{o_\Prob(1)}

\newcommand{\C}{\mathbb{C}_n}
\newcommand{\Cxi}{\mathbb{C}_{n,\xi}}
\newcommand{\bs}{{q}}

\newcommand\blfootnote[1]{%
  \begingroup
  \renewcommand\thefootnote{}\footnote{#1}%
  \addtocounter{footnote}{-1}%
  \endgroup
}

\allowdisplaybreaks

\begin{document}
\title{Statistics for Heteroscedastic Time Series Extremes}
\author{Axel B\"ucher, Tobias Jennessen\\}

\date{\today}

\begin{abstract} 
Einmahl, de Haan and Zhou (2016, Journal of the Royal Statistical Society: Series B, 78(1), 31– 51) recently introduced a stochastic model that allows for heteroscedasticity of extremes. The model is extended to the situation where the observations are serially dependent, which is crucial for many practical applications. We prove a local limit theorem for a kernel estimator for the scedasis function, and a functional limit theorem for an estimator for  the integrated scedasis function. We further prove consistency of a bootstrap scheme that allows to test for the null hypothesis that the extremes are homoscedastic. Finally, we propose an estimator for the extremal index governing the dynamics of the extremes and prove its consistency. All results are illustrated by Monte Carlo simulations. An important intermediate result concerns the sequential tail empirical process under serial dependence.
\bigskip

\noindent
\emph{Key words.} 
Extremal Index; 
Kernel Estimator; 
Multiplier Bootstrap;
Non-Stationary Extremes;
Regular Varying Time Series.
\end{abstract}

\maketitle

\tableofcontents

%%%%%%%%%%%%%%%%%
%%%%%%%%%%%%%%%%%
%\input{text-main}

\section{Introduction} \label{sec:intro}

Classical extreme value statistics is concerned with analyzing the extremal behavior of a series of independent and identically distributed (i.i.d.) random variables. However, in many practical applications, the latter assumption is not justifiable, since the data typically consist of observations collected on one or more variables over time. The observations may then both exhibit serial dependence and they may be drawn from a distribution that changes smoothly  (or even abruptly) as time progresses. The latter is particularly the case in many applications from environmental statistics (e.g., due to climate change), while the former is also omnipresent in typical applications from finance. 

While an abundance of methods has been proposed for tackling the resulting challenges concerning the bulk of the data (see, e.g., \citealp{BroDav91} for a classical account on time series analysis; \citealp{Dah12} for an overview on locally stationary processes that allow for nonparametric smooth changes over time; or \citealp{AueHor13} for an overview on results for change point analysis involving abrupt changes), respective results concerning extreme value analysis are much less developed, in particular for the situation exhibiting both serial dependence and non-stationarity. 

Theoretical results on extreme value analysis for stationary time series build on corresponding probabilistic theory summarized in \cite{LeaLinRoo83}, see also Chapter~10 in \cite{BeiGoeSegTeu04} for an overview or \cite{KulSou20} for a modern account in the heavy tailed case. Respective asymptotic results on a large class of estimators for the tail index can be found in \cite{Dre00}, with some substantial extensions on important intermediate results in \cite{DreRoo10}. Results regarding the time series dynamics for the heavy tailed case can be found in \cite{KulSou20} and the references therein. Smooth non-stationarity has often been approached by parametric regression models, see, e.g., \cite{DavSmi90,Col01}, where, however, no asymptotic theory is provided. Nonparametric approaches that were supported by asymptotic results can be found in \cite{HalTaj00_trend}; these authors also explicitly allow for serial dependence. \cite{DehTanNev15} consider a situation in which the smooth non-stationarity was formulated in a parametric way on the level of the domain of attraction condition rather than the limit situation. A nonparametric version of that model was investigated in \cite{EinDehZho16} (see below for details). Under the assumption of serial independence, these authors also provide asymptotic theory, which was recently extended in \cite{DehZho21_trends} to trends in the tail index and in \cite{Ein_etal22} to multivariate, spatial applications. Finally, change point tests for the tail index and the extremal dependence (i.e., abrupt changes in the tail behavior) can be found in \cite{KojNav17, BucKinKoj17, Hog17,Hog18}, with the latter two references  explicitly allowing for serially dependent observations.

The above literature review reveals a crucial gap which motivates the present paper: the models initiated by \cite{EinDehZho16} (subsequently referred to as EdHZ) have never been investigated under the assumption that the observations are serially dependent. Throughout the paper,  we therefore work under the following model adapted from EdHZ: for sample size $n$ and at time points $i\in\{1, \dots, n\}$, we observe possibly dependent random variables $X_1^{\scs (n)}, \ldots, X_n^{\scs (n)}$ with continuous cumulative distribution functions (c.d.f.s) $F_{n,1},\ldots,F_{n,n}$, i.e., $X_{i}^{\scs (n)}\sim F_{n,i}$.  We assume that all these distribution functions share a common right endpoint $x^*=\sup\{x \in \R:F_{n,i}(x)<1\}$, and that there exists some continuous reference c.d.f.\ $F$ with the same right endpoint $x^*$ that is strictly increasing  on its support and some positive function $c$ on $[0,1]$ such that
\begin{equation} \label{eq:common F}
\lim_{x\uparrow x^*}\frac{1-F_{n,i}(x)}{1-F(x)}=c\Big(\frac{i}{n}\Big).
\end{equation}
As in EdHZ, we refer to $c$ as the \textit{scedasis function}, which we additionally assume to be a bounded and continuous probability density function. The case where $c\equiv 1$ corresponds to \textit{homogeneous extremes}, while the opposite is referred to as  \textit{heteroscedastic extremes}. The integrated scedasis function shall be denoted by
\[ 
C(s) := \int_{0}^{s} c(x) \, \mathrm{d}x, \quad s \in [0,1]. 
\]
We allow for serial dependence in the following sense: for each $n \in \N$, the unobservable sample $U_1^{\scs (n)},\ldots,U_n^{\scs (n)}$ with $U_i^{\scs (n)} = F_{n,i}(X_i^{\scs (n)})$ is assumed to be an excerpt from a strictly stationary time series $(U_t^{\scs (n)})_{t \in \Z}$ whose distribution does not depend on $n$. The dynamics of the extremes of the latter series will later be captured by the concept of regular variation \citep{BasSeg09}, see  Condition \ref{cond:reg} below for details, and by the extremal index $\theta$ \citep{Lea83}, see Condition \ref{cond:ei}. 
Recall that the reciprocal of the extremal index may be interpreted as the mean cluster size of subsequent extreme observations. 

Our contributions within the above model are as follows: first, we provide a (pointwise) central limit theorem on the kernel estimator for the scedasis function that was studied in EdHZ for the independent case. Notably, the serial dependence will only show up in the asymptotic estimation variance. Second, we study an empirical version of the integrated scedasis function from EdHZ and provide a respective functional central limit theorem; again, the asymptotic covariance functional will be different from that in the serially independent case. The latter is a major nuisance for testing the null hypothesis of homoscedastic extremes, i.e., $H_0:c \equiv 1$, where standard approaches based on functionals of the law of the Brownian bridge as proposed in EdHZ do not work any more. As a circumvent, we develop a suitable multiplier bootstrap scheme and show its consistency; for this, we need to extent results from \cite{Dre15} and Section 12 in \cite{KulSou20} to the non-stationary case. The bootstrap scheme is then used to define a classical bootstrap test as well as a test based on self-normalization, the latter being computationally much more efficient but slightly less powerful. Finally, we also propose an estimator for the extremal index $\theta$ of the underlying stationary time series that governs the dynamics of the extremes and show its consistency. For that purpose, we use a suitable modification of the block-maxima estimator from \cite{Nor15,BerBuc18} to the current non-stationary setting.  On a theoretical level, a crucial tool for most of the afore-mentioned asymptotics is a functional central limit theorem for the sequential tail empirical process (STEP), which may also be of interest for other statistical problems not tackled in this paper. 

The remaining parts of this paper are organized as follows: the assumptions needed to prove the asymptotic results are summarized and discussed in Section \ref{sec:assumptions}, where we also introduce a location-scale model meeting these assumptions. Section~\ref{sec:scedasis} is concerned with the estimation of the scedasis function and the integrated scedasis function. Section~\ref{sec:testing} is about testing for the null hypothesis that the extremes are homoscedastic. The assessment of the serial dependence is dealt with in Section~\ref{sec:EI}, where we also extend the discussion on the location-scale model.  A functional central limit theorem for the sequential tail empirical process under serial dependence is presented in Section~\ref{sec:sstep}. The finite-sample behavior of the introduced methods is investigated in a Monte Carlo simulation study in Section~\ref{sec:sim}. The proofs for Section~\ref{sec:sstep} are given in Section~\ref{sec:proofs}, with some auxiliary lemmas collected in Section~\ref{sec:auxiliary}. Finally,  all other proofs are deferred to a supplementary material.

Throughout, all convergences are for $n\to\infty$ if not mentioned otherwise. Weak convergence is denoted by $\wto$. The left-continuous generalized inverse of some increasing function $H$ is denoted by $H^{-1}(p) = \inf \{ x \in \R: H(x) > p \}$.  The sup-norm of some real-valued function $f$ defined on some domain $T$ is denoted by $\|f\|_\infty$.

\section{Mathematical Preliminaries} \label{sec:assumptions}

Let $k=k_n$ be an increasing integer sequence satisfying $k\to\infty$ and $k=o(n)$ as $n \to \infty$; the STEP and our estimators for the scedasis function will be defined in terms of $k$, which essentially determines the threshold for declaring an observation as extreme. Let $L \in \N$ be some arbitrary but fixed constant (later determining, on which set the STEP will be defined; most often, we need $L=1$ or $L=2$). We impose the following set of assumptions:
\begin{compactenum}[(B8)]
\renewcommand{\theenumi}{(B0)}
\renewcommand{\labelenumi}{\theenumi}
\item \label{cond:basic} \textbf{Basic assumptions.} 
The conditions on the model in Section \ref{sec:intro} are met. 
\renewcommand{\theenumi}{(B1)}
\renewcommand{\labelenumi}{\theenumi}
\item \label{cond:reg} \textbf{Multivariate regular variation.} For each $n\in\N$, $U_{1}^{\scs (n)}, \dots, U_{n}^{\scs (n)}$ is an excerpt from a strictly stationary time series $(U_t^{\scs (n)})_{t\in\Z}$ whose marginal stationary distribution is necessarily standard uniform on $(0,1)$. The processes $(U_t^{\scs (n)})_{t\in\Z}$ are all equal in law; denote a generic version by $(U_t)_{t\in\Z}$.  The process $Z_t = 1/(1-U_{t})$ (note that $Z_t$ is standard Pareto) is  stationary and regularly varying, necessarily  with index $\alpha=1$ \citep{BasSeg09}.
\renewcommand{\theenumi}{(B2)}
\renewcommand{\labelenumi}{\theenumi}
\item \label{cond:sbias} \textbf{Regularity of $c$.} 
The function $c$ is Hölder-continuous of order $1/2$, that is, there exists $K_c>0$ such that
\[
|c(s) - c(s')| \le K_c |s-s'|^{1/2} \qquad \forall\ s,s'\in[0,1].
\]
\renewcommand{\theenumi}{(B3)}
\renewcommand{\labelenumi}{\theenumi}
\item \label{cond:mixing} \textbf{Blocking sequences and Beta-mixing.} 
There exist integer sequences $1< \ell_n < r=r_n<n$, both converging to infinity as $n\to\infty$ and satisfying 
$
\ell_n = o(r), r = o(\sqrt{k} \vee \tfrac{n}{k}), 
$
such that the beta-mixing coefficients of $(U_t)_{t\in\Z}$ satisfy 
$
\beta(n)= o(1)$ and $\frac{n}{r} \beta(\ell_n) = o(1)$.
\renewcommand{\theenumi}{(B4)}
\renewcommand{\labelenumi}{\theenumi}
\item \label{cond:expnum} \textbf{Moment bound on the number of extreme observations.} Let $c_\infty=c_\infty(L)=1+ L\|c\|_\infty$, where $\| \cdot\|_\infty$ denotes the sup norm of a real-valued function. There exists $\delta>0$ such that
\begin{align*}
\Exp\Big[ \Big\{\sum_{s=1}^{r} \bm 1(U_s > 1- \tfrac{k}n c_\infty(L)) \Big\}^{2+\delta} \Big]= O(r\tfrac{k}n).
\end{align*}
\renewcommand{\theenumi}{(B5)}
\renewcommand{\labelenumi}{\theenumi}
\item \label{cond:hbound} \textbf{Moment bound on extreme increments.} There exists a non-decreasing, continuous function $h:[0,c_\infty(L)] \to [0,\infty)$, positive on $(0,c_\infty(L)]$ and with $h(0)=0$, such that, for all sufficiently large $n$,
\begin{align*}
\Exp\Big[\Big\{\sum_{s=1}^{r} \bm 1(1-\tfrac{k}n x \ge U_s > 1- \tfrac{k}n y) \Big\}^{2} \Big] \le r \frac{k}n \times h(y-x)
\end{align*}
for all $0 \le x \le y \le c_\infty(L)$ with $c_\infty(L)$ from \ref{cond:expnum}.
\renewcommand{\theenumi}{(B6)}
\renewcommand{\labelenumi}{\theenumi}
\item \label{cond:2nd order}  \textbf{Second order condition.}
There exists a positive, eventually decreasing function $A$ with $\lim_{t \to \infty} A(t) = 0$ such that, as $x \uparrow x^*$, 
\[ 
\sup_{n \in \N} \max_{1 \leq i \leq n} \left| \frac{1-F_{n,i}(x)}{1-F(x)} - c\Big(\frac{i}{n}\Big) \right| = O\Big( A \Big(\frac{1}{1-F(x)}\Big) \Big). 
\]
\end{compactenum}
%\end{condition}

Condition~\ref{cond:reg} allows to control the serial dependence within the observed time series via \textit{tail processes} \citep{BasSeg09}. More precisely, 
by Theorem 2.1 in \cite{BasSeg09}, regular variation of $(Z_t)_{t\in\Z}$ is equivalent to the fact that there exists a process $(Y_t)_{t\in\N_0}$ (the \textit{tail process}) with $Y_0$ standard Pareto such that, for every $\ell\in\N$ and as $x\to\infty$,
\begin{align} \label{eq:tp}
\Prob(x^{-1} (Z_0, \dots, Z_\ell) \in \cdot \mid Z_0 > x) \dto \Prob((Y_0, \dots, Y_\ell) \in \cdot),
\end{align}
where, necessarily,  $Y_j\ge 0$ for $j\ge 1$. Further, by Theorem 2 and its subsequent discussion in \cite{Seg03}, $Y_j$ is absolutely continuous on $(0,\infty)$ and may have an atom at $0$.

Condition \ref{cond:sbias}  has also been imposed in \cite{EinDehZho16}. Since $k=o(n)$, it implies that
\[
\lim_{n\to\infty} \sup_{s \in [0,1]} {\sqrt k} \bigg| \frac{1}{n} \sum_{i=1}^{\ip{ns}} c(\tfrac{i}n)  - C(s)\bigg| = 0,
\] 
which will imply that there is no asymptotic bias in our main result below. The condition will however also be needed to prove \eqref{eq:rho} below.

The conditions in \ref{cond:mixing}, \ref{cond:expnum}, \ref{cond:hbound} are essentially conditions imposed in  Example~3.8  in \cite{DreRoo10} for deriving weak convergence of the standard non-sequential univariate tail empirical process under stationarity. Condition~\ref{cond:hbound} has mostly been  shown with $h(z)=K z$, for some $K>0$, see, e.g., \cite{Dre00} for solutions of stochastic recurrence equations. Condition \ref{cond:2nd order} is a second-order condition on the speed of convergence in (\ref{eq:common F}); it was also used in \cite{EinDehZho16}. It is worth noting that Conditions \ref{cond:expnum}-\ref{cond:hbound} (and only these) depend on the constant $L \in \N$. The sequence $\ell_n$ in \ref{cond:mixing} plays the role of a small-block length in a big-block-small-block technique, while $r-\ell_n$ is the length of a corresponding big block.

\begin{example} \label{ex:model}
 Let us consider the following location-scale model, for which the above conditions can be shown to hold. Let 
 \[ 
 X_i^{(n)} = \sigma \big( \tfrac{i}{n} \big) W_i + \mu \big( \tfrac{i}{n} \big), \quad i=1,\ldots,n,
 \] 
 where $(W_t)_{t \in \Z}$ is a strictly stationary time series (see below for an explicit example) with c.d.f.\ $F$ and where $\sigma : [0,1] \to (0,\infty)$ and $\mu:[0,1] \to \R$ are sufficiently smooth functions. Then, we obtain 
\[
F_{n,i}(x) = F \Big( \frac{x-\mu (\tfrac{i}{n})}{\sigma ( \tfrac{i}{n})} \Big), \quad x \in \R, 
\] 
and $U_i^{(n)} = F_{n,i}(X_i^{(n)})  = F(W_i)$, $i=1,\ldots,n$, such that $U_1^{(n)},\ldots,U_n^{(n)}$ is an excerpt from a strictly stationary time series, with marginal distribution given by the uniform distribution on $[0,1]$. 
 
Next, as a special case, consider $(W_t)_{t \in \Z}$ to be a max-autoregressive process (ARMAX), defined by the recursion 
\begin{align} 
W_t = \max \{ \lambda W_{t-1}, (1-\lambda) V_t \}, \quad t \in \Z, \label{def_ARMAX}
\end{align} 
where $\lambda \in [0,1)$ and $(V_t)_{t \in \Z}$ is an i.i.d. sequence of Fréchet(1)-distributed random variables. A stationary solution of the above recursion is given by $W_t = \max_{j \geq 0} (1-\lambda) \lambda^j V_{t-j},$ such that the stationary solution is again Fréchet(1)-distributed, i.e., $F(x) = \exp(-1/x)$. Then, the scedasis function $c$ can be easily calculated via 
\[ 
\lim_{x \to \infty} \frac{1-F_{n,i}(x)}{1-F(x)} = \lim_{x\to \infty} \frac{1-\exp\big(- \sigma(\tfrac{i}{n}) / \{ x-\mu ( \tfrac{i}{n}) \} \big)}{1-\exp(-1/x)}  = \sigma \big( \tfrac{i}{n} \big),
\] 
yielding $c=\sigma$. We show that Conditions \ref{cond:basic}-\ref{cond:2nd order} are met. Condition \ref{cond:basic} and \ref{cond:sbias} are obviously fulfilled, provided the scedasis function $c$ is sufficiently regular. Condition \ref{cond:reg} can be seen to hold as follows. Since $(W_t)_{t \in \Z}$ is a moving maximum process, its tail process exists by Theorem 13.5.5 in \cite{KulSou20}, which implies that it is regularly varying by Theorem 2.1 in \cite{BasSeg09}. Then, $Z_t=1/(1-U_t)=1/\{1-F(W_t)\}$ is regularly varying with index $\alpha=1$ according to Lemma 2.1 in \cite{DreSegWar15}. By \cite{BerBuc18}, page 2322, $(W_t)_{t \in \Z}$, and hence also $(U_t)_{t \in \Z}$, is geometrically $\beta$-mixing, whence Condition \ref{cond:mixing} holds. In that reference it is further shown that their Condition 2.1(ii) holds for $\delta =1$, implying that our Condition \ref{cond:expnum} also holds for $\delta = 1$ in view of the fact that, $rk =o(n)$ by Condition \ref{cond:mixing}.
This also yields that 
$
\Exp [|\sum_{s=1}^{r} \I ( 1- \tfrac{k}{n} x \geq U_s > 1- \tfrac{k}{n} y ) |^3]\lesssim r \frac{k}{n}(y-x)$ for all $0 \leq x \leq y \leq c_{\infty}(L), 
$
%$
%\Exp \Big[ \Big| \sum_{s=1}^{r} \I \big( 1- \tfrac{k}{n} x \geq U_s > 1- \tfrac{k}{n} y \big) \Big|^3 \Big] \lesssim r \frac{k}{n}(y-x), \quad 0 \leq x \leq y \leq c_{\infty}(L), 
%$
for $n$ large enough (such that $rk/n \leq 1$), which implies \ref{cond:hbound}. Finally, Condition \ref{cond:2nd order} can be seen to hold for $A(x)=x^{-1}$. 
\end{example}

%%%%%%%%%%%%%%%%%%%%%%%%%%%%%%%%%%%%%%%%%%%%%%%%%%%%%%%%%%%%%%%%%%%%%%%%%%%%%%%%%%%%%%%%%

\section{Estimation of the (integrated) scedasis function} \label{sec:scedasis}

In this section, we provide weak convergence  results for estimators for the scedasis function $c$ and its integrated version $C$; see also \cite{EinDehZho16} for related results in the serial independent case. Throughout, let $X_{n,1} \leq \ldots \leq X_{n,n}$ denote the order statistic of $X_1^{\scs (n)},\ldots,X_n^{\scs (n)}$.
 
First, for the estimation of the scedasis function, we apply a kernel density estimator. Let $K$ be a continuous and symmetric function on $[-1,1]$ with $K(x) = 0$ for $|x| > 1$ and $\int_{-1}^{1}K(x) \ \mathrm{d}x = 1$. Let $h= h_n > 0$ denote a bandwidth paramater. Since we are also concerned with the estimation of $c$ near the boundaries of the interval $[0,1]$, we make use of the boundary-corrected kernel $K_b$ of $K$ (see \citealp{Jones93}): 
for $s \in [0,1]$, set 
\[ 
\tilde{c}_n(s) = \frac{1}{k h} \sum_{i=1}^{n} \I(X_i^{(n)} > X_{n,n-k} ) \ K_b\Big( \frac{s-i/n}{h}, s \Big),  
\] 
where $k=k_n$ is from Condition \ref{cond:mixing} and where $K_b$ is defined as follows. First,  
for $p \in [0,1]$, let 
\[ 
a_j(p) = \int_{-1}^{p} x^j K(x) \ \mathrm{d}x, \quad b_j(p) = \int_{-p}^{1} x^j K(x) \ \mathrm{d}x. 
\]
For $s \le h$, write $s=ph$ and let 
\[ 
K_b(x,s) = \frac{a_2(p) - a_1(p)x}{a_0(p)a_2(p)-a_1^2(p)}K(x), \quad x \in [-1,1], 
\]
and for $s \ge 1- h$, write $s=1-ph$ and let
\[ 
K_b(x,s) = \frac{b_2(p) - b_1(p)x}{b_0(p)b_2(p)-b_1^2(p)}K(x), \quad x \in [-1,1], 
\]
and for $s \in (h,1-h)$, let $K_b(x,s)=K(x)$ for $x \in [-1,1]$. 
Note that $K_b$ is depending on $n$, which we have suppressed from the notation.

To obtain asymptotic normality of the introduced estimator we additionally impose the following condition.

\begin{compactenum}[(B7)]
	\renewcommand{\theenumi}{(B7)}
	\renewcommand{\labelenumi}{\theenumi}
	\item \label{cond:bandwidth}  \textbf{Bandwidth.}
%	The bandwidth sequence $h = h_n > 0$ satisfies $h \to 0$ and $kh \to \infty$. Further, $k^{1/5} h \to \lambda \geq 0$ and $r=o(\sqrt{kh})$ and $kr^2 = o(n^2 h^3)$. 
	The bandwidth sequence $h = h_n > 0$ satisfies $h \to 0$ and $kh \to \infty$. Further, $k^{1/5} h \to \lambda \geq 0$ and $r=o(\sqrt{kh})$ and $h \ge k^{-1/3}$. 
\end{compactenum}	
The first three conditions in \ref{cond:bandwidth} are standard bandwidth conditions that have also been imposed in Proposition 2 in \cite{EinDehZho16} to establish asymptotic normality of the scedasis estimator at point $s=1$. The condition $r=o(\sqrt{kh})$ is slightly stronger than $r=o(\sqrt{k})$ from Condition \ref{cond:mixing}, which is used in Theorem~\ref{thm:integr_scedasis} below to derive asymptotic normality of the estimator for the integrated scedasis function, where the rate of convergence is $\sqrt{k}$. Finally, the condition $h \ge k^{-1/3}$ is required for technical reasons in the proof (together with $r=o(n/k)$ from Condition \ref{cond:mixing}, it implies $kr^2 = o(n^2 h^3)$, which will be used throughout the proofs); note that it is satisfied for the standard MSE optimal bandwidth choice of the order $k^{-1/5}$ \citep{Tsy09}.

\begin{theorem} \label{thm:scedasis}
	Suppose that Conditions \ref{cond:basic}-\ref{cond:bandwidth} hold for $L=2$ and let $c \in C^2([0,1])$. Let the function $K$ be Lipschitz-continuous and symmetric on $[-1,1]$ with $K(x) = 0$ for $|x| > 1$ and $\int_{-1}^{1}K(x) \ \mathrm{d}x = 1$. 
	If $k$ satisfies $\sqrt{k}A(\tfrac{n}{2k}) \to 0$, then, for any $s \in [0,1]$ and as $n \to \infty$,
	\[ 
	\sqrt{kh} \big\{ \tilde{c}_n(s) - c(s) \big\} \wto \Nor(\mu_s, \sigma_s^{2}), 
	\]
	where
	\begin{align*}
	\mu_s &= \lambda^{5/2} \frac{c''(s)}{2} a(s) , \quad \sigma_s^{2} = c(s) \eta(s) \ \Big\{ d_0(1,1) + 2 \sum_{h=1}^{\infty} d_h(1,1) \Big\} 
	\end{align*}
	and where, recalling the tail process $(Y_t)_{t\in\N_0}$ associated with $(Z_t)_{t\in\Z}$ from \eqref{eq:tp},
	\begin{align} \label{eq:defdh}
	d_h(x,x') = \Prob\Big(Y_0 > \frac{1}{ x} , Y_{h} > \frac{1}{x'}\Big)
	\end{align}  
%	and 
%	$a(s) = \int_{-1}^{1} K(x) x^2 \ \mathrm{d}x, \eta(s) = \int_{-1}^{1} K^2(x) \ \mathrm{d}x$ for $s \in (0,1)$ and $a(0) = \int_{-1}^{0} K_b(x,0)x^2 \ \mathrm{d}x, a(1) = \int_{0}^{1} K_b(x,1)x^2  \ \mathrm{d}x, \eta(0) = \int_{-1}^{0} K_b^2(x,0) \ \mathrm{d}x,  \eta(1) =  \int_{0}^{1} K_b^2(x,1) \ \mathrm{d}x$.
and $a(0) = \int_{-1}^{0} K_b(x,0)x^2 \ \mathrm{d}x,  \eta(0) = \int_{-1}^{0} K_b^2(x,0) \ \mathrm{d}x, a(1) = \int_{0}^{1} K_b(x,1)x^2  \ \mathrm{d}x,  \eta(1) =  \int_{0}^{1} K_b^2(x,1) \ \mathrm{d}x$ and, for $s\in(0,1)$,
\begin{align*}
a(s) &= \int_{-1}^{1} K(x) x^2 \ \mathrm{d}x,   \qquad	\eta(s) = \int_{-1}^{1} K^2(x) \ \mathrm{d}x.
\end{align*}
%and, for $s\in(0.1)$,
%	\begin{align*}
%	a(s) &= \int_{-1}^{1} K(x) x^2 \ \mathrm{d}x,  &\quad  a(0) &= \int_{-1}^{0} K_b(x,0)x^2 \ \mathrm{d}x, &\quad  a(1) &= \int_{0}^{1} K_b(x,1)x^2 \ \mathrm{d}x, \\
%	\eta(s) &= \int_{-1}^{1} K^2(x) \ \mathrm{d}x, &\quad \eta(0) &= \int_{-1}^{0} K_b^2(x,0) \ \mathrm{d}x, &\quad \eta(1)& =  \int_{0}^{1} K_b^2(x,1) \ \mathrm{d}x.
%	\end{align*}
\end{theorem}
It is part of the assertion that the series defining $\sigma^2_s$ is convergent. The result may further be extended to cover the cases $s=s_n=ph$ and $s=s_n=1-ph$ for some $p\in (0,1]$; details are omitted for the sake of brevity.

%%%%%%%%%%%%%%%%%%%%%%%%%%%%%%%%%%%%%%%%%%%%%%%%%%%%%%%%%%%%%%%%%%%%%%%%%%%%%%%%%%%%%%%%

Next, we analyze an estimator for the integrated scedasis function $C(s) = \int_{0}^{s} c(x) \ \mathrm{d}x$, that was also investigated in \cite{EinDehZho16}. Define the estimator for $C$ as
\[ \hat{C}_n(s) = \frac{1}{k} \sum_{i=1}^{\lfloor ns \rfloor} \I\big(X_i^{(n)} > X_{n,n-k} \big), \quad s \in [0,1]. \] 

\begin{theorem} \label{thm:integr_scedasis}
	Suppose that Conditions \ref{cond:basic}-\ref{cond:2nd order} hold for $L=1$ and that $k$ satisfies $\sqrt{k}A(\tfrac{n}{k}) \to 0$. Then, as $n \to \infty$,
\[ 
\big\{ \sqrt{k} \big( \hat{C}_n(s) - C(s) \big) \big\}_{s \in [0,1]} \wto \big\{ \mathbb{S}(s,1) - C(s) \mathbb{S}(1,1) \big\}_{s \in [0,1]} 
\]
	in $(\ell^{\infty}([0,1]), \|\cdot\|_{\infty})$, where $\mathbb{S}$ denotes a tight, centered Gaussian process on $[0,1]^2$ with covariance given by
\begin{align} \label{eq:cov2}
\mathfrak c((s,x),(s',x')) = C(s\wedge s')  \Big\{ d_0(x,x') + \sum_{h =1}^\infty  \big( d_h( x,  x') + d_h(x',x) \big)  \Big\},
\end{align}
where $d_h$ is defined in \eqref{eq:defdh}. It is part of the assertion that the above series is convergent.  
\end{theorem}
%It is part of the assertion that the above series is convergent.  

%%%%%%%%%%%%%%%%%%%%%%%%%%%%%%%%%%%%%%%%%%%%%%%%%%%%%%%%%%%%%%%%%%%%%%%%%%%%%%%%%%%%%%%%%

\section{Testing for heteroscedastic extremes} \label{sec:testing}

In the following we construct tests that allow to detect whether the time series exhibits heteroscedasticity of extremes. Here, the extremes are homoscedastic (i.e., not heteroscedastic) if the scedasis function satisfies $c \equiv 1$ or if, equivalently, the integrated scedasis function satisfies $C(s)=s$ for all $s\in [0,1]$. Thus, we test 
\[ 
H_0 : C(s) = s \ \textrm{ for all } \ s \in [0,1], \qquad
H_1: C(s) \neq s \ \textrm{ for some } \ s \in [0,1].
\]
To this purpose, we pursue two approaches, where one is based on a bootstrap-procedure and the other uses a self-normalization technique. Let $\C (s) = \sqrt{k} \{ \hat C_n(s) -s \}$, $s \in [0,1]$, such that, by Theorem \ref{thm:integr_scedasis}, under $H_0$ and as $n \to \infty$,
\begin{align*}
  \big\{ \C (s) \big\}_{s \in [0,1]} \wto \big\{ \mathbb{S}(s,1) - s \mathbb{S}(1,1) \big\}_{s \in [0,1]}
\end{align*}
in $(\ell^{\infty}([0,1]), \|\cdot\|_{\infty})$.
Note that $\mathbb{S}(\cdot, 1)$ is a tight, centered Gaussian process on $[0,1]$ satisfying $\Cov(\mathbb{S}(s,1),\mathbb{S}(t,1))=(s \wedge t)\sigma^2$, $s,t \in [0,1]$, where
$ 
\sigma^2 = d_0(1,1)+2 \sum_{h=1}^{\infty} d_h(1,1)$ and $d_h$ is defined in Theorem \ref{thm:integr_scedasis}, which implies that under $H_0$, as $n \to \infty$,
\begin{align}
  \C \wto \sigma \mathbb{B} \quad \textrm{in} \ (\ell^{\infty}([0,1]), \|\cdot\|_{\infty}), \label{brownianbridge}
\end{align}
where $\mathbb{B}$ denotes a Brownian Bridge on $[0,1]$. 

For both approaches take the block length parameter $r$ from Condition \ref{cond:mixing} (which now becomes a hyperparameter of the statistical method; see \citealp{Dre15} and \citealp{KulSou20} for a similar approach), set $m = \lfloor n/r \rfloor$ and let 
\begin{align*} 
I_j = \{ (j-1)r+1, \ldots, jr\}, \quad j=1, \ldots, m, %\label{blocks} 
\end{align*}
be the $j$-th block of size $r$.

We start with the bootstrap, more precisley, we use a multiplier block bootstrap. Let $B\in\N$ denote the number of boostrap repetitions and let $(\xi_1^{\scs (b)},\ldots,\xi^{\scs (b)}_{m})_{b=1,\ldots,B}$ be i.i.d. and independent from $(X_i^{\scs (n)})_i$, with $\Exp[\xi_j^{\scs (b)}]=0, \Exp[(\xi_j^{\scs (b)})^2]=1$ and $|\xi_j^{\scs (b)}| \leq M$ for some constant $M>0$ for all $j=1,\ldots,m$ and $b=1,\ldots,B$ (for instance, $\xi_j^{\scs (b)}$ is Rademacher distributed). 
Set 
\begin{align*} 
\Cxi^{(b)} (s) &= \mathbb{D}^{(b)}_{n,\xi}(s) - \hat C_n(s) \mathbb{D}^{(b)}_{n,\xi}(1),
\end{align*}
where
\begin{align}
\mathbb{D}^{(b)}_{n,\xi}(s) &= \frac{1}{\sqrt{k}} \sum_{j=1}^{m} (\xi_j^{(b)} - \bar{\xi}^{(b)}) \sum_{i \in I_j} \I \Big( X_i^{(n)} > X_{n,n-k} \Big) \I(\tfrac{i}{n} \leq s), \quad s \in [0,1], \nonumber 
\end{align}
and $\bar{\xi}^{(b)}= m^{-1} \sum_{j=1}^{m} \xi_j^{(b)}$. Note that we may write
\[
 \mathbb{D}^{(b)}_{n,\xi}(s) 
= \frac{1}{\sqrt{k}} \sum_{j=1}^m \xi_j^{(b)} \Big\{ Y_{n,j}(s)-\frac{1}{m} \sum_{\ell=1}^m Y_{n,\ell}(s)\Big\}
\]
with $Y_{n,j}(s)=\sum_{i \in I_j} \I \big( X_i^{(n)} > X_{n,n-k} \big) \I(\tfrac{i}{n} \leq s)$, which is akin to the process considered in Formula~(2.3) in \cite{Dre15}.

\begin{theorem} \label{thm:boot_jointconv}
 Suppose that Conditions \ref{cond:basic}-\ref{cond:2nd order} hold for $L=1$ and that $k$ satisfies $\sqrt{k}A(\tfrac{n}{k}) \to 0$. Then, as $n \to \infty$,
 \begin{align*}
 \Big(\mathbb C_n, \Cxi^{(1)}, \ldots, \Cxi^{(B)} \Big)
 \wto \Big( \mathbb{C}, \mathbb{C}^{(1)}, \ldots, \mathbb{C}^{(B)} \Big) \quad \textrm{in} \ (\ell^{\infty}([0,1]), \|\cdot\|_{\infty})^{B+1},
 \end{align*}
 where $\mathbb{C}(s) = \mathbb{S}(s,1)-C(s)\mathbb{S}(1,1)$ and $\mathbb{C}^{(1)}, \ldots, \mathbb{C}^{(B)}$ are independent copies of $\mathbb{C}$.
\end{theorem}

The previous theorem may alternatively be formulated as a conditional limit theorem, see Section~3.6 in \cite{VanWel96} or Section 10 in \cite{Kos08} for details on that mode of convergence when applied to non-measurable stochastic processes. More precisely, by Lemma 3.11 in \cite{BucKoj19}, the weak convergence relation in the previous theorem is equivalent to the fact that
$
\sup_{h \in \mathrm{BL}_1(\ell^\infty([0,1])} \big| \Exp \big[ h(\Cxi^{\scs (1)}) \mid X_{n,1}, \dots, X_{n,n}\big] - \Exp[h(\mathbb{C})] \big| = \op
$
and that $\Cxi^{\scs (1)}$ is asymptotically measurable, where $\mathrm{BL}_1(\ell^\infty([0,1])$ denotes the set of real valued Lipschitz functions on $\ell^\infty([0,1])$ with Lipschitz constant 1 that are bounded by~1. We prefer to work with the unconditional statement from Theorem~\ref{thm:boot_jointconv}, as it is more intuitive.

We propose to test for $H_0:c\equiv 1$ based on the test statistics
\[ 
S_{n,1} = \|\C\|_{\infty}, \quad T_{n,1} = \int_0^1 \C(s)^2 \ \mathrm{d}s.
\]
In view of Theorem~\ref{thm:boot_jointconv}, the corresponding bootstrap quantities are given by
\[ 
S_{n,1}^{(b)} = \|\Cxi^{(b)}\|_{\infty}, \quad T_{n,1}^{(b)} = \int_0^1 \Cxi^{(b)}(s)^2 \ \mathrm{d}s, \quad b =1,\ldots,B.
\]
For $\alpha\in(0,1)$, let $\hat q_{n,B,S}(1-\alpha)$ and $\hat q_{n,B,T}(1-\alpha)$ denote the empirical $(1-\alpha)$-quantile of $S_{n,1}^{\scs (1)},\ldots,S_{n,1}^{\scs (B)}$ and $T_{n,1}^{\scs (1)},\ldots,T_{n,1}^{\scs (B)}$, respectively. The test procedures are then defined as
\begin{align*}
 \varphi_{n,B,S}(\alpha) = \I \big(S_{n,1} > \hat q_{n,B,S}(1-\alpha)\big), \quad \varphi_{n,B,T}(\alpha) = \I \big(T_{n,1} > \hat q_{n,B, T}(1-\alpha) \big).
\end{align*}

\begin{corollary} \label{cor:test_boot}
 Suppose that Conditions \ref{cond:basic}-\ref{cond:2nd order} hold for $L=1$  and that $k$ satisfies $\sqrt{k}A(\tfrac{n}{k}) \to 0$. Let $\alpha \in (0,1)$. Then, if  $H_0: c \equiv 1$ is met, 
\[ 
\lim_{n,B \to \infty} \Prob(\varphi_{n,B,S}(\alpha)=1) = \alpha, \quad \lim_{n,B \to \infty} \Prob(\varphi_{n,B,T}(\alpha)=1)=\alpha.
\]
Further, if $H_1: c \not \equiv 1$ is met, then, for any $B \in \N$,
\[ 
\lim_{n \to \infty} \Prob (\varphi_{n,B,S}(\alpha) = 1) = 1, \quad  \lim_{n \to \infty} \Prob (\varphi_{n,B,T}(\alpha) = 1) = 1.
\]
\end{corollary}

Next, we introduce tests based on the concept of self-normalization. The basic idea is to consider the quotient of two statistics, such that the unknown variance factor $\sigma$ in (\ref{brownianbridge}) cancels out. To do this, we take two of the bootstrap-quantities from Theorem \ref{thm:boot_jointconv}, and define
\[ 
S_{n,2} = \frac{\|\mathbb{C}_n\|_{\infty}}{\|\mathbb C_{n,\xi}^{(1)} - \mathbb C_{n,\xi}^{(2)}\|_{\infty}}, \quad T_{n,2} = \frac{\int_0^1 \mathbb{C}_n^2(s) \ \mathrm{d}s}{\int_0^1 \big( \mathbb C_{n,\xi}^{(1)}(s)- \mathbb C_{n,\xi}^{(2)}(s) \big)^2 \ \mathrm{d}s}. 
\]
By Theorem \ref{thm:boot_jointconv} we know that under $H_0$, as $ n \to \infty$,
\begin{align*}
 S_{n,2} \wto S_2 := \frac{\|\mathbb{B}\|_{\infty}}{\|\mathbb{B}^{(1)}-\mathbb{B}^{(2)}\|_{\infty}}, \quad T_{n,2} \wto T_2 := \frac{\int_0^1 \mathbb{B}(s)^2 \ \mathrm{d}s}{\int_0^1 \big( \mathbb{B}^{(1)}(s)-\mathbb{B}^{(2)}(s) \big)^2 \ \mathrm{d}s},
\end{align*}
where $\mathbb{B}, \ \mathbb{B}^{(1)}$ and $\mathbb{B}^{(2)}$ are independent Brownian Bridges on $[0,1]$.
For $\alpha\in(0,1)$, let $q_S(1-\alpha)$ and $q_T(1-\alpha)$ be the $(1-\alpha)$-quantile of $S_2$ and $T_2$, respectively. The corresponding test procedures are given by
\[ 
\varphi_{n,S}(\alpha) = \I \big( S_{n,2} > q_S(1-\alpha) \big), \quad \varphi_{n,T}(\alpha) = \I \big( T_{n,2} > q_T(1-\alpha) \big).
\]

\begin{corollary} \label{cor:test_norm}
 Suppose that Conditions \ref{cond:basic}-\ref{cond:2nd order} hold for $L=1$ and that $k$ satisfies $\sqrt{k}A(\tfrac{n}{k}) \to 0$. Let $\alpha \in (0,1)$. Then,  if $H_0: c \equiv 1$ is met,
 \[ 
 \lim_{n \to \infty} \Prob(\varphi_{n,S}(\alpha)=1) = \alpha, \quad \lim_{n \to \infty} \Prob(\varphi_{n,T}(\alpha)=1)=\alpha
 \]
 Further, if $H_1: c \not \equiv 1$ is met, then
 \[ \lim_{n \to \infty} \Prob_{H_1} (\varphi_{n,S}(\alpha) = 1) = 1, \quad \lim_{n \to \infty} \Prob (\varphi_{n,T}(\alpha) = 1) = 1. 
 \]
\end{corollary}

%%%%%%%%%%%%%%%%%%%%%%%%%%%%%%%%%%%%%%%%%%%%%%%%%%%%%%%%%%%%%%%%%%%%%%%%%%%%%%%%%%%%%%%%%

\section{Assessing the serial dependence} \label{sec:EI}

Within our basic model described in the introduction, the dynamics of the time series extremes are governed by the stationary time series $(Z_t)_{t \in \Z}$ from Condition \ref{cond:reg}. There are many interesting statistical problems related to those dynamics which are worth to be investigated like, e.g., estimating the distribution of the tail process (see \citealp{Dav18} for stationary observations) or estimation of general cluster functionals  (see Section~10 in \citealp{KulSou20} for stationary observations). Throughout, we restrict attention to estimating  the extremal index~$\theta$, which may be regarded as the most traditional parameter associated with the serial dependence.

Recall that the extremal index $\theta \in (0,1]$ of $(Z_t)_t$ exists iff the same is true for $(U_t)_t$ (in that case, the indices are equal), and that the latter requires that, for any $\tau > 0$, there exists a sequence $(u_n(\tau))_{n \in \N}$ such that $\lim_{n \to \infty} n\{ 1- u_n(\tau) \} = \tau$ and 
\begin{align}
	\lim_{n \to \infty} \Prob \big( \max_{1 \leq i \leq n} U_i \leq u_n(\tau) \big) = e^{-\theta \tau}. \label{extremal_index}
\end{align}
One can further show that, if the extremal index exists, then (\ref{extremal_index}) holds for any sequence $(u_n(\tau))_n$ with $\lim_{n \to \infty} n \{ 1-u_n(\tau) \} = \tau$. Subsequently, we choose $u_n(\tau) = 1-\tau/n$.

For estimating $\theta$, we divide the finite stretch of observations $X_1^{\scs (n)}, \ldots, X_n^{\scs (n)}$ into non-overlapping successive blocks of size $\bs=\bs_n$, i.e., into blocks 
\begin{align*} 
I_j' = \{ (j-1)\bs+1, \ldots, j\bs\}, \quad j=1, \ldots, k', %\label{blocks2} 
\end{align*}
where $k'=\lfloor n/\bs \rfloor$.
%A possible remainder block of less than $r$ observations will be asymptotically negligible and should be deleted; consequently, we assume $n=m_nb_n$ for simplicity. 
For $j=1,\ldots,k'$, set 
\begin{align} \label{eq:znj}
Z_{n,j} = \bs \big\{ 1-\max_{i \in I_j'} F(X_i^{(n)}) \big \}, \quad 
\hat Z_{n,j} = \bs \big\{ 1-\max_{i \in I_j'} \hat{F}_n(X_i^{(n)}) \big\}, 
\end{align}
where $\hat F_n(x) = n^{-1} \sum_{i=1}^{n} \I(X_i^{(n)} \leq x)$ denotes the empirical c.d.f. of $X_1^{(n)}, \ldots, X_n^{(n)}$.  Note that, in view of \eqref{eq:common F}, for sufficiently large $x\in\R$,
\[
\Exp[1-\hat F_n(x)] = \frac1n \sum_{i=1}^n \{1-F_{n,i}(x)\}
=
 \{1-F(x)\} \frac1n \sum_{i=1}^n \{ c(i/n) +o(1) \} \approx 1-F(x)
\]
(ignoring the possible non-uniformity in \eqref{eq:common F} for the moment),
whence $\hat Z_{n,j}$ can be regarded as an observable counterpart of $Z_{n,j}$.

In the following, we will show that $Z_{n,1+\lfloor\xi k'\rfloor}$, $\xi \in [0,1)$, asymptotically follows an exponential distribution with parameter depending on $\theta$, this result being the basis for our estimation procedure for $\theta$, see Lemma \ref{Zni_exp}. To prove this, we impose the subsequent conditions.

\begin{compactenum}[(B10)]
	\renewcommand{\theenumi}{(B8)}
	\renewcommand{\labelenumi}{\theenumi}
	\item 
	\label{cond:ei} \textbf{Extremal Index.} 
	The stationary time series $(U_t)_{t \in \Z}$ from Condition \ref{cond:reg} is assumed to have an extremal index $\theta \in (0,1]$.
	\renewcommand{\theenumi}{(B9)}
	\renewcommand{\labelenumi}{\theenumi}
	\item 
	\label{cond:blocksize} \textbf{Blocking sequences and mixing.}
	The blocksize $\bs$ is chosen in such a way that it satisfies $\bs = o(\sqrt{n})$ and $n \beta(\bs) = o(\bs)$ as $n \to \infty$.
	\renewcommand{\theenumi}{(B10)}
	\renewcommand{\labelenumi}{\theenumi}
	\item \label{cond:uniform integr}  \textbf{Uniform integrability.}
	For some $\delta_1 > 0$,
	\[ 
	\limsup_{n \to \infty} \sup_{\xi \in (0,1)} \Exp \Big[ \big| Z_{n,1+\lfloor \xi k' \rfloor} \big|^{2 + \delta_1} \Big] < \infty. 
	\]
\end{compactenum}	

Condition \ref{cond:uniform integr} is imposed to deduce uniform integrability of the $Z_{n,1+\lfloor \xi k' \rfloor}^2$; it will imply convergence of the corresponding first and second moments. 

\begin{remark}
We exemplarily show that the above conditions hold for the location-scale model from Example \ref{ex:model} with $(W_t)_{t \in \Z}$ chosen as the max-autoregressive process defined in (\ref{def_ARMAX}). 
First, the process $(W_t)_{t \in \Z}$ has an extremal index $\theta$ given by $\theta=1-\lambda$ \citep[Chapter 10]{BeiGoeSegTeu04}, such that $(U_t)_{t \in \Z}$ also has extremal index $\theta$ and Condition \ref{cond:ei} holds.
Further, by \cite{BerBuc18}, page 2322, $(W_t)_{t \in \Z}$, and hence also $(U_t)_{t \in \Z}$, is geometrically $\beta$-mixing, whence Condition \ref{cond:blocksize} is fulfilled for appropriate choice of $\bs$.
Regarding Condition \ref{cond:uniform integr}, we have, for $j \in \N$,
\begin{align}
 Z_{n,j} 
 &=  \nonumber
 \bar Z_{n,j} \frac{1-\max_{i \in I_j'}F(X_i^{(n)})}{1-\max_{i \in I_j'}F_{n,i}(X_i^{(n)})} \\
 &\leq 
 \bar Z_{n,j} \frac{1-\exp\big( - (c_{\textrm{min}} \max_{i \in I_j'} W_i +  \inf_{s \in [0,1]}\mu(s))^{-1}\big)}{1-\exp\big( -(\max_{i \in I_j'} W_i )^{-1}\big)}, \label{Z_and_bar_Z}
\end{align}
where $\bar Z_{n,j} =\bs \big\{ 1-\max_{i \in I_j'} U_i \big \}$. Note that the distribution of the right-hand side in the last display is independent of $j \in \N$. By induction, 
$
\Prob \big( \max_{i=1,\ldots,b} W_i \leq x \big) = F(x)^{1+\theta(b-1)} = \exp (- \{ 1+\theta(b-1) \}/x)$ for $x>0, b \in \N, 
$
such that $\max_{i=1,\ldots,\bs} W_i$ converges to $\infty$ in probability. Therefore, any absolute moment of the second factor of the right-hand side in (\ref{Z_and_bar_Z}) converges. Further, it is shown in Example~6.1 in \cite{BerBuc17}, see the proof of their Condition 2.1(vi) holds, that 
$
\limsup_{n \to \infty} \Exp \big[ \bar Z_{n,1}^{\delta'} \big] < \infty 
$
for any $\delta' > 0$. Along with inequality (\ref{Z_and_bar_Z}), Hölder's inequality implies that Condition \ref{cond:uniform integr} holds.
\end{remark}

\begin{lemma} \label{Zni_exp}
Fix $ \xi \in [0,1)$. Suppose that Conditions \ref{cond:basic}-\ref{cond:sbias}, \ref{cond:2nd order} and \ref{cond:ei}-\ref{cond:blocksize} hold. Then, $Z_{n,1+ \lfloor \xi k' \rfloor} \wto \mathrm{Exp}(\theta c(\xi))$ as $ n \to \infty$. 
\end{lemma}

This result motivates estimators based on the method of moments, see \cite{Nor15, BerBuc18} for the stationary case. Consider the (unobservable) random variable
\[ T_{n} = \frac{1}{k'} \sum_{j=1}^{k'} Z_{n,j}. \]
%, \quad T_{n2} = \frac{1}{m_n} \sum_{s=1}^{m_n} Z_{n,s} \ c \Big(\frac{s}{m_n}\Big).\] 
Then, for
$
\varphi_n: [0,1] \to \R, \varphi_n(\xi) = \sum_{j=1}^{k'} \Exp \big[Z_{n,j}\big] \I \big( \xi \in \big[ \tfrac{j-1}{k'}, \tfrac{j}{k'} \big) \big), 
$
we obtain 
\[ 
\Exp[T_{n}] = \frac{1}{k'} \sum_{j=1}^{k'} \Exp\big[Z_{n,j}\big] = \int_{0}^{1} \varphi_n(\xi) \ \mathrm{d}\xi. 
\]
By Condition \ref{cond:uniform integr} and Lemma \ref{Zni_exp}, for any fixed $\xi\in[0,1)$, we have 
$
\varphi_n(\xi) = \Exp \big[ Z_{n,1+\lfloor \xi k' \rfloor} \big] \to \Exp [V_\xi],  n \to \infty, 
$
where $V_{\xi}  \sim \textsf{Exp}(\theta c(\xi))$. Since $\sup_{n \in \N} \|\varphi_n\|_{\infty} < \infty$ by Condition \ref{cond:uniform integr}, the dominated convergence theorem implies
\begin{align}
\Exp \big[ T_{n} \big] = \int_{0}^{1} \varphi_n(\xi) \ \mathrm{d}\xi \to \int_{0}^{1} \Exp \big[ V_{\xi} \big] \ \mathrm{d}\xi = \frac{1}{\theta} \int_{0}^{1} \! \frac{1}{c(\xi)} \ \mathrm{d}\xi. \nonumber %\label{moment_eq1}
\end{align}
\ \\
Recall that the function $c$ is positive and continuous on $[0,1]$; thus there is a positive number $c_{\min}$ such that $c(s) > c_{\min}$ for all $s \in [0,1]$. Therefore, it is advisable to also truncate $\tilde c_n$ from below, say by considering $\hat{c}_n = \max(\tilde{c}_n, \kappa)$ with some small, positive constant $\kappa>0$. Subsequently, we assume that $0<\kappa<c_{\min}$.
Now, let us estimate $\tau = \int_{0}^{1} c(\xi)^{-1} \ \mathrm{d}\xi$ by $\hat{\tau}_n = \int_{0}^{1} \hat{c}_n(\xi)^{-1} \ \mathrm{d}\xi$. 
Since $\Exp[T_{n}] \to \theta^{-1} \tau$, a sensible, observable method of moments estimator for $\theta$ is given by 
\begin{align*}
	\hat{\theta}_{n} = \hat T_n^{-1} \hat{\tau}_n, \quad \textrm{where} \quad \hat T_n = \frac{1}{k'} \sum_{j=1}^{k'} \hat Z_{n,j}.
\end{align*}
The subsequent theorem yields concistency of this estimator; its finite-sample properties are studied in Section~\ref{sec:sim}.

\begin{theorem} \label{thm:cons_EI}
	Suppose that Conditions \ref{cond:basic}-\ref{cond:sbias} hold. Assume $c \in C^2([0,1])$ and let the function $K$ in the definition of $\hat{c}_n$ be Lipschitz-continuous.
	\begin{compactenum}
		\item[(a)] If additionally Conditions \ref{cond:mixing}-\ref{cond:bandwidth} hold for $L=2$ and if $k$ satisfies $\sqrt{k}A(\tfrac{n}{2k}) \to 0$, then $\hat \tau_n =\tau+\op$ as $n \to \infty$.
		\item[(b)] If additionally Conditions \ref{cond:mixing}-\ref{cond:2nd order} hold for $k=k'$ and for all $L \in \N$, and if $k'$ satisfies $\sqrt{k'}A(\tfrac{n}{Lk'}) \to 0$ for all $L\in\N$, and if Conditions \ref{cond:ei}-\ref{cond:uniform integr} hold, then $\hat T_n =\theta^{-1} \tau+\op$ as $n \to \infty$. 
	\end{compactenum}
	In particular, if all of the above conditions are met, then $\hat \theta_{n} \pto \theta$ as $n \to \infty$. 
	\end{theorem}

%%%%%%%%%%%%%%%%%%%%%%%%%%%%%%%%%%%%%%%%%%%%%%%%%%%%%%%%%%%%%%%%%%%%%%%%%%%%%%%%%%%%%%%%

%%%%%%%%%%%%%%%%%%%%%%%%%%%%%%%%%%%%%%%%%%%%%%%%%%%%%%%%%%%%%%%%%%%%%%%%%%%%%%%%%%%%%%%%%

\section{Weak convergence of the (simple) STEP} \label{sec:sstep}

Functional weak convergence of the subsequent processes will be essential for proving the asymptotic results in the previous sections. Precisely, we are interested in the simple sequential tail empirical process (simple STEP) $\mathbb{S}_n$ and the sequential tail empirical process (STEP) $\mathbb{F}_n$ defined as
\begin{align}
\mathbb{S}_n(s,x) 
&=\label{eq:defsstep}
\sqrt{k} \bigg\{ \frac{1}{k}\sum_{i=1}^{[ns]} \bm{1}\big\{ U_{i}^{(n)} > 1 - \tfrac{k}{n} c(\tfrac{i}{n}) x  \big\}  - xC(s)  \bigg\}, \\ 
\mathbb{F}_n(s,x)
&=  \label{eq:defstep}
\sqrt{k} \bigg\{ \frac{1}{k}\sum_{i=1}^{[ns]} \bm{1}\Big\{ X_{i}^{(n)} > V\Big(\frac{n}{kx}\Big)  \Big\}  - xC(s)  \bigg\},
\end{align}
where $(s,x) \in [0,1] \times [0,\infty)$ and where $V=(\tfrac{1}{1-F})^{-1}$.

\begin{proposition} \label{prop:sstep}
Suppose that Conditions~\ref{cond:basic}--\ref{cond:mixing} hold. Fix some constant $L\in\N$ and suppose that Conditions \ref{cond:expnum} and \ref{cond:hbound} hold for $L$. Then, as $n \to \infty$,
\[
\mathbb{S}_n \wto \mathbb{S} \quad \textrm{in} \quad (\ell^\infty([0,1] \times [0,L]), \| \cdot \|_\infty),
\]
where $\Sb$ denotes a tight, centered Gaussian process on $[0,1] \times [0,L]$ with covariance
$\mathfrak c((s,x),(s',x'))$ as defined in \eqref{eq:cov2}.
\end{proposition}

\begin{proposition}  \label{prop:step}
	Suppose that Conditions~\ref{cond:basic}--\ref{cond:mixing} and \ref{cond:2nd order} hold. Fix some constant $L\in\N$ and suppose that Conditions \ref{cond:expnum} and \ref{cond:hbound} hold for $L$. If $k$ satisfies $\sqrt{k}A(\tfrac{n}{Lk}) \to 0$ as $n \to \infty$, then
	\[ 
	\sup_{(s,x) \in [0,1] \times [0,L]} |\mathbb{F}_n(s,x) - \mathbb{S}_n(s,x)| = \op. 
	\]
	As a consequence,	
	$
	\mathbb{F}_n \wto \mathbb{S}$  in $(\ell^\infty([0,1] \times [0,L]), \| \cdot \|_\infty).
	$
\end{proposition}

\section{Finite-Sample Results}
\label{sec:sim}

A simulation study is carried out to analyze the finite-sample performance of the introduced methods. Results are presented for scaled versions of two common time series models. 
Define the following functions, later resulting in different scedasis functions.
\begin{align*}
 \textrm{(i)} & \ c_{1,\beta}(s) = \beta + 2(1-\beta)s,  \\
 \textrm{(ii)} & \ c_{2,\beta}(s) = (\beta + 4(1-\beta)s) \I(s \in [0,0.5]) + (4 -3\beta -4(1-\beta)s) \I(s \in (0.5,1]).
\end{align*}
Note that $c_{1,\beta}$ is a straight line connecting the points $(0,\beta)$ and $(1,2-\beta)$, while $c_{2,\beta}$ is a polygonal chain with vertices $(0,\beta)$  $(1/2, 2-\beta)$ and $(1, \beta)$.

We consider the following  scale models.
\begin{compactenum}
 \item[$\bullet$] The ARMAX-model:
 Let $(W_t)_t$ be an ARMAX-process as specified in \eqref{def_ARMAX}. We consider $\lambda \in\{ 0,0.25 \}$; note that $\lambda =0$ corresponds to the i.i.d.\ case. Denote the c.d.f.\ of $W_t$ by $F$, which is the c.d.f.\ of the standard Fr\'echet-distribution. For $j\in\{1, 2\}$ and $i\in\{1, \dots, n\}$, let 
 \[
 X_i^{(n)} = c_{j,\beta}(\tfrac{i}{n}) W_i.
 \]
 By Example \ref{ex:model}, the scedasis function $c$ is equal to $c_{j,\beta}$.
 Further, for $j\in\{1, 2\}$, consider 
 \begin{equation*} 
  X_i^{(n)} = \tilde c_{j,\beta}(\tfrac{i}{n},W_i)W_i :=  \big\{ \I(W_i < p) + c_{j,\beta}(\tfrac{i}{n}) \ \I(W_i \geq p) \big\} W_i, %\label{ARMAX_c3} 
 \end{equation*}
 where $p$ is the $80\%$-quantile of $F$. In this model, the scale transformation introduced by $c_{j,\beta}$ only effects the observations exceeding the large threshold $p$. One can easily see that the scedasis function $c$ is equal to $c_{j,\beta}$.
 \item[$\bullet$] The ARCH-model: 
 Let $(W_t)_t$ be an ARCH-process, i.e.,
 \[ W_t = (2 \times 10^{-5} + \lambda W_{t-1}^2)^{1/2} V_t, \quad t \in \Z,\] where $\lambda \in (0,1)$ and $(V_t)_{t \in \Z}$ is an i.i.d. sequence of $\mathcal{N}(0,1)$-distributed random variables. We consider $\lambda=0.7$. By Theorem 1.1 in \cite{DehResRooVri89} the c.d.f.\ $F$ of $W_t$ satisfies $1-F(x) \sim dx^{-\kappa'}$ as $x \to \infty$ for some constant $d>0$, with $\kappa'$ (approximately)  given by $\kappa' = \kappa'(\lambda) = 1.586$; see Table 3.2 in that reference.  For $j\in\{1, 2\}$ and $i\in\{1, \dots, n\}$, let 
 \[
 X_i^{(n)} = c_{j,\beta}(\tfrac{i}{n})^{1/\kappa'} \ W_i. 
 \]
 The scedasis function $c$ is equal to $c_{j,\beta}$.
 Further, similar as for the ARMAX-model, consider 
 \begin{equation*} 
  X_i^{(n)} = \tilde c_{j,\beta}(\tfrac{i}{n},W_i)W_i := \big\{ \I(W_i < p) + c_{j,\beta}(\tfrac{i}{n})^{1/\kappa'} \ \I(W_i \geq p) \big\} W_i %\label{ARCH_c3} 
 \end{equation*}
 for $j\in\{1,2\}$,
 where $p$ is the $80\%$-quantile of $F$. A straightforward calculation shows that the scedasis function $c$ is equal to $c_{j,\beta}$ as well. 
\end{compactenum}
Note that the ARMAX model with $\lambda =0$ corresponds to the case that the observations are independent. We call this case simply the independent model.

In the subsequent simulation study, the parameter $\beta$ of the scedasis functions, is set to $\beta=1,0.75,0.5,0.25$. In each case, the sample size is fixed to $n=2000$ and the performance of the statistical methods is assessed based on $N=1000$ simulation runs each if not mentioned otherwise.

\subsection{Estimation of the scedasis function}
We start by briefly considering the behavior of the kernel estimator for the scedasis function. For the sake of brevity, we restrict the presentation to the ARCH-model with scedasis function $c_{2,\beta}$ with $\beta=0.5$; the behavior within the other models was found to be very similar. In Figure \ref{Fig:plot_c2}, we depict the estimator $\tilde c_n$ for four exemplary time series, where we use the biweight kernel $K$ 
\begin{align} \label{eq:biw}
K(x)=\frac{15}{16}(1-x^2)^2, \quad x \in [-1,1],
\end{align}
$k=400$ and consider bandwidths $h \in \{0.03, 0.11, 0.19, 0.27\}$. We observe typical over-fitting (under-smoothing) for small values of $h$ and under-fitting (over-smoothing) for large values of $h$. Note in particular that the estimator no longer captures the peak of $c_{2,\beta}(s)$ at $s=0.5$ for $h=0.27$. Visual inspection suggests that reasonably good choices for the bandwidth lie in the interval $[0.1,0.2]$; an observation that was confirmed in simulations regarding the other models described in the previous section.

\begin{figure} [t!]
	\begin{center}
		\includegraphics[width=\textwidth]{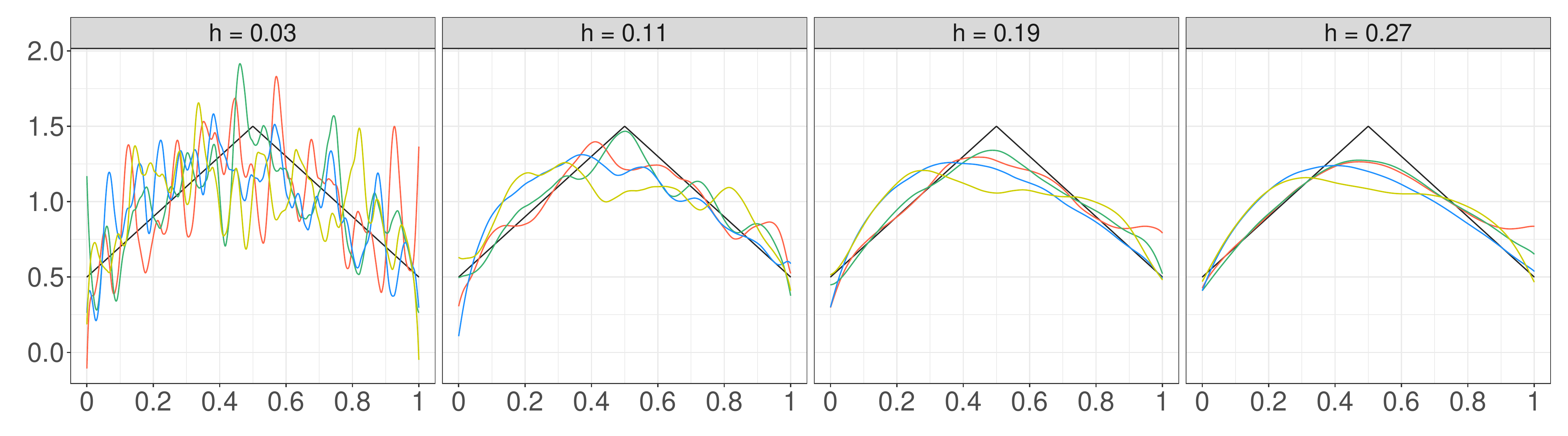} \vspace{-.5cm}
		\caption{The scedasis function $c_{2,\beta}$ with $\beta=0.5$ (black line) and the estimator $\tilde c_n$ evaluated at four exemplary time series generated from the ARCH-model.} 
		\vspace{-.3cm} 
		\label{Fig:plot_c2}
	\end{center}
\end{figure}

\subsection{Testing for heteroscedastic extremes}

We next study the performance of the introduced test procedures. Recall that both the tests based on the multiplier block bootstrap and the ones relying on the method of self-normalization depend on a multiplier sequence $(\xi_i)_i$, for which we choose an i.i.d.\ Rademacher sequence. The following results are based on $B=300$ bootstrap replicates. We consider block sizes $q \in\{ 4,8 \}$ and number of exceedances $k\in\{100,200\}$, which corresponds to $5\%$ or $10\%$ of the total observations, respectively. The test level is set to $\alpha=0.05$.

\begin{table}[t!]
\centering
{ \scriptsize
\begin{tabular}{ll|rrrr|rrrr|rr}
%  \hline 
  \hline
 &&  \multicolumn{4}{c|}{Bootstrap with $(k,r)=$ }& \multicolumn{4}{c|}{Self-Normalization with $(k,r)=$}& \multicolumn{2}{c}{EdHZ  with $k=$} \\
   Model & $\beta$ & $(100,4)$ & $(100,8)$ & $(200,4)$ & $(200,8)$ & $(100,4)$ & $(100,8)$ & $(200,4)$ & $(200,8)$ & 100 & 200 \\ 
 \hline
  \addlinespace[.1cm]
  \multicolumn{12}{l}{\hspace{.3cm}\textbf{Panel (A): Scale model $c_{1,\beta}$}} \\ \hline
 Indep. & 1.0 & 3.9 & 2.2 & 1.8 & 0.6 & 4.2 & 2.3 & 2.3 & 1.5 & 4.7 & 4.1 \\ 
   & 0.75 & 23.3 & 17.9 & 34.6 & 22.5 & 16.5 & 12.7 & 21.8 & 15.6 & 26.9 & 46.2 \\ 
    & 0.5 & 78.2 & 71.5 & 95.8 & 91.2 & 51.2 & 44.0 & 70.0 & 60.7 & 81.9 & 98.2 \\ 
    & 0.25 & 98.9 & 98.4 & 100.0 & 99.9 & 84.5 & 76.5 & 94.7 & 89.4 & 99.0 & 100.0 \\ 
    %& 0.0 & 100.0 & 100.0 & 100.0 & 100.0 & 96.6 & 92.6 & 99.5 & 98.6 & 100.0 & 100.0 \\  
%    \hline
   ARMAX & 1.0 & 6.6 & 4.5 & 4.4 & 2.7 & 5.5 & 4.0 & 3.4 & 3.1 & 14.8 & 12.1 \\ 
    & 0.75 & 21.4 & 17.3 & 30.7 & 20.8 & 14.6 & 12.2 & 18.7 & 14.3 & 35.4 & 50.7 \\ 
   & 0.5 & 65.6 & 59.7 & 88.3 & 79.2 & 42.5 & 36.7 & 59.3 & 48.3 & 77.4 & 96.0 \\ 
  & 0.25 & 94.7 & 91.8 & 99.6 & 99.4 & 71.5 & 66.5 & 90.6 & 82.4 & 97.6 & 100.0 \\ 
    %& 0.0 & 99.9 & 99.6 & 100.0 & 100.0 & 90.7 & 88.4 & 99.1 & 96.1 & 100.0 & 100.0 \\  
%    \hline
   ARCH & 1.0 & 7.9 & 5.0 & 4.4 & 1.7 & 5.7 & 4.5 & 3.6 & 3.0 & 16.1 & 11.1 \\ 
    & 0.75 & 48.0 & 38.4 & 51.0 & 37.1 & 30.7 & 24.4 & 28.8 & 25.3 & 61.6 & 66.2 \\ 
    & 0.5 & 94.7 & 92.0 & 98.5 & 96.5 & 73.6 & 67.6 & 81.2 & 69.9 & 98.0 & 99.6 \\ 
    & 0.25 & 99.9 & 99.8 & 100.0 & 100.0 & 94.3 & 91.2 & 96.8 & 93.2 & 100.0 & 100.0 \\ 
   %& 0.0 & 100.0 & 100.0 & 100.0 & 100.0 & 98.7 & 98.4 & 99.7 & 99.5 & 100.0 & 100.0 \\
 \hline
  \addlinespace[.1cm]
  \multicolumn{12}{l}{\hspace{.3cm}\textbf{Panel (B): Scale model $c_{2,\beta}$}} \\ \hline
  Indep. & 1.0 & 3.9 & 2.2 & 1.8 & 0.6 & 4.2 & 2.3 & 2.3 & 1.5 & 4.7 & 4.1 \\ 
    & 0.75 & 7.3 & 4.3 & 5.6 & 2.2 & 6.7 & 4.6 & 5.6 & 3.2 & 5.8 & 7.4 \\ 
   & 0.5 & 29.4 & 19.6 & 52.0 & 32.3 & 17.0 & 11.2 & 25.6 & 17.3 & 20.1 & 55.3 \\ 
   & 0.25 & 78.6 & 68.6 & 98.2 & 92.6 & 42.1 & 33.6 & 57.9 & 48.4 & 68.0 & 98.8 \\ 
   %& 0.0 & 99.8 & 98.4 & 100.0 & 100.0 & 66.2 & 59.0 & 82.5 & 75.7 & 98.8 & 100.0 \\  
%   \hline
   ARMAX & 1.0 & 6.6 & 4.5 & 4.4 & 2.7 & 5.5 & 4.0 & 3.4 & 3.1 & 14.8 & 12.1 \\ 
    & 0.75 & 9.6 & 6.8 & 9.4 & 5.7 & 7.7 & 6.2 & 7.4 & 5.2 & 17.6 & 21.9 \\ 
    & 0.5 & 27.5 & 18.3 & 41.2 & 26.4 & 16.3 & 12.0 & 20.2 & 16.2 & 36.9 & 64.8 \\ 
    & 0.25 & 62.8 & 53.7 & 90.6 & 79.6 & 32.9 & 26.3 & 50.7 & 40.5 & 74.3 & 97.0 \\ 
    %& 0.0 & 92.7 & 88.3 & 99.8 & 99.3 & 58.1 & 49.1 & 78.3 & 66.9 & 96.8 & 100.0 \\  
%    \hline
   ARCH & 1.0 & 7.9 & 5.0 & 4.4 & 1.7 & 5.7 & 4.5 & 3.6 & 3.0 & 16.1 & 11.1 \\ 
    & 0.75 & 19.8 & 13.0 & 14.2 & 6.4 & 13.7 & 8.1 & 10.7 & 5.6 & 27.6 & 26.6 \\ 
    & 0.5 & 66.5 & 53.5 & 73.7 & 52.7 & 35.1 & 25.6 & 37.9 & 24.7 & 73.1 & 86.6 \\ 
  & 0.25 & 96.2 & 92.7 & 99.5 & 97.5 & 65.1 & 56.4 & 69.3 & 60.0 & 98.5 & 99.9 \\ 
     %& 0.0 & 99.8 & 99.7 & 100.0 & 100.0 & 80.7 & 76.6 & 87.7 & 78.5 & 99.8 & 100.0 \\ 
    \hline
  \addlinespace[.1cm]
  \multicolumn{12}{l}{\hspace{.3cm}\textbf{Panel (C): Scale model $\tilde c_{1,\beta}$}} \\ \hline
   Indep. & 1.0 & 3.9 & 2.2 & 1.8 & 0.6 & 4.2 & 2.3 & 2.3 & 1.5 & 4.7 & 4.1 \\ 
    & 0.75 & 24.3 & 18.4 & 34.8 & 22.8 & 15.8 & 12.5 & 24.2 & 14.4 & 26.9 & 46.2 \\ 
    & 0.5 & 78.2 & 71.0 & 96.0 & 91.4 & 49.2 & 46.2 & 71.0 & 58.4 & 81.9 & 98.2 \\ 
   & 0.25 & 99.3 & 98.6 & 100.0 & 100.0 & 80.1 & 75.9 & 94.0 & 89.1 & 99.0 & 100.0 \\ 
    %& 0.0 & 100.0 & 100.0 & 100.0 & 100.0 & 96.4 & 94.6 & 99.2 & 98.4 & 100.0 & 100.0 \\ 
%    \hline
   ARMAX & 1.0 & 6.6 & 4.5 & 4.4 & 2.7 & 5.5 & 4.0 & 3.4 & 3.1 & 14.8 & 12.1 \\ 
     & 0.75 & 21.9 & 17.7 & 30.3 & 20.6 & 16.0 & 12.7 & 17.9 & 12.2 & 35.4 & 50.7 \\ 
     & 0.5 & 65.8 & 57.9 & 88.4 & 80.1 & 41.6 & 36.0 & 61.2 & 50.1 & 77.4 & 96.0 \\ 
     & 0.25 & 94.8 & 91.7 & 99.6 & 99.5 & 71.5 & 61.2 & 88.4 & 83.3 & 97.6 & 100.0 \\ 
     %& 0.0 & 99.8 & 99.7 & 100.0 & 100.0 & 92.2 & 87.9 & 98.7 & 96.7 & 100.0 & 100.0 \\ 
%     \hline
    ARCH & 1.0 & 7.9 & 5.0 & 4.4 & 1.7 & 5.7 & 4.5 & 3.6 & 3.0 & 16.1 & 11.1 \\ 
     & 0.75 & 46.5 & 37.9 & 52.0 & 36.2 & 31.4 & 24.4 & 29.4 & 21.7 & 61.6 & 66.2 \\ 
     & 0.5 & 94.6 & 92.0 & 98.7 & 96.4 & 76.9 & 66.1 & 79.7 & 68.8 & 98.0 & 99.6 \\ 
     & 0.25 & 99.9 & 100.0 & 100.0 & 100.0 & 93.9 & 89.8 & 97.6 & 94.2 & 100.0 & 100.0 \\ 
     %& 0.0 & 100.0 & 100.0 & 100.0 & 100.0 & 98.9 & 96.6 & 99.8 & 98.5 & 100.0 & 100.0 \\ 
    \hline
  \addlinespace[.1cm]
  \multicolumn{12}{l}{\hspace{.3cm}\textbf{Panel (D): Scale model $\tilde c_{2,\beta}$}} \\ \hline
   Indep. & 1.0 & 3.9 & 2.2 & 1.8 & 0.6 & 4.2 & 2.3 & 2.3 & 1.5 & 4.7 & 4.1 \\ 
     & 0.75 & 6.9 & 4.0 & 6.4 & 2.0 & 5.8 & 5.2 & 5.9 & 3.3 & 5.8 & 7.4 \\ 
     & 0.5 & 28.8 & 20.1 & 51.6 & 33.0 & 17.6 & 10.6 & 24.9 & 17.9 & 20.1 & 55.3 \\ 
     & 0.25 & 78.4 & 69.0 & 98.5 & 92.2 & 39.3 & 32.5 & 58.6 & 46.2 & 68.0 & 98.8 \\ 
     %& 0.0 & 99.7 & 98.9 & 100.0 & 100.0 & 68.7 & 59.1 & 83.8 & 75.1 & 98.8 & 100.0 \\ 
%     \hline
    ARMAX & 1.0 & 6.6 & 4.5 & 4.4 & 2.7 & 5.5 & 4.0 & 3.4 & 3.1 & 14.8 & 12.1 \\ 
     & 0.75 & 10.5 & 7.0 & 10.1 & 5.3 & 7.1 & 6.4 & 7.1 & 4.4 & 17.6 & 21.9 \\ 
     & 0.5 & 25.8 & 19.1 & 42.0 & 25.9 & 17.3 & 10.6 & 23.7 & 14.1 & 36.9 & 64.8 \\ 
     & 0.25 & 63.6 & 53.3 & 90.8 & 79.5 & 33.9 & 26.2 & 48.7 & 40.8 & 74.3 & 97.0 \\ 
     %& 0.0 & 92.8 & 88.4 & 99.9 & 99.4 & 55.0 & 48.7 & 77.3 & 64.4 & 96.8 & 100.0 \\ 
%     \hline
    ARCH & 1.0 & 7.9 & 5.0 & 4.4 & 1.7 & 5.7 & 4.5 & 3.6 & 3.0 & 16.1 & 11.1 \\ 
     & 0.75 & 19.8 & 12.3 & 14.8 & 6.8 & 13.9 & 7.2 & 11.3 & 6.9 & 27.6 & 26.6 \\ 
     & 0.5 & 65.7 & 53.4 & 73.3 & 53.8 & 36.4 & 26.5 & 36.8 & 27.4 & 73.1 & 86.6 \\ 
     & 0.25 & 96.3 & 93.1 & 99.6 & 97.8 & 64.5 & 54.9 & 70.8 & 59.7 & 98.5 & 99.9 \\ 
     %& 0.0 & 99.8 & 99.8 & 100.0 & 100.0 & 83.6 & 73.8 & 88.4 & 80.2 & 99.8 & 100.0 \\ 
   \hline 
%   \hline
\end{tabular}
}
\vspace{-.1cm}
\caption{Empirical rejection percentage of the test procedures.}
\label{Tab:Rej_perc}
\end{table}

Since the Cramér-von-Mises-type test statistics (i.e., $\varphi_{n,B,T}$ and $\varphi_{n,T}$) were found to be superior to the Kolmogorov-Smirnov-type test statistics (i.e., $\varphi_{n,B,S}$ and $\varphi_{n,S}$), we only present results for the former. Here, we refer to $\varphi_{n,B,T}$ simply as the bootstrap, and to $\varphi_{n,T}$ as the self-normalization. All rejection percentages are presented in Table \ref{Tab:Rej_perc}.

We start by discussing the behavior of the tests under $H_0:C(s)=s$ for all $s \in [0,1]$; note that $\beta =1$ represents being under $H_0$ for all data generating processes under consideration.  We also present results for the Cramér-von-Mises-type test from \cite{EinDehZho16}, which was designed for the case of independent data and is here denoted by EdHZ. One can see that our tests hold their level and, as expected, that the test from \cite{EinDehZho16} holds its level in the independent model, but fails to do so in the other dependent models. 

Next, we consider the performance under the alternatives. One can see that the power of the tests increases with decreasing $\beta$, which is to be expected since a decrease in $\beta$ results in a stronger deviation of $c_{j,\beta}$ from the null hypothesis that the scedasis function equals one. In general, the power of the bootstrap-test is uniformly higher than the power of the test based on self-normalization, but both exhibit high power for $\beta = 0.25$. Recall again that the self-normalization test only requires evaluation of $\mathbb C_{n,\xi}^{\scs (b)}$ for $b\in\{1,2\}$, while the expression must be evaluated a large number of times for the bootstrap test (we choose $B=300$).
With regard to the choice of $k$ and $r$ the highest power is usually attained for $k=200$ and $r=4$.

\subsection{Estimation of the extremal index} \label{Sim:EI}

We finally briefly evaluate the performance of the estimator for the extremal index. For comparison, we also introduce a second estimator for $\theta$ based on the method of moments, which may also be motivated by  Lemma \ref{Zni_exp}: under the notation of Section \ref{sec:EI}, consider the (unobservable) random variable
\[ 
T_{n2} = \frac{1}{k'} \sum_{j=1}^{k'} Z_{n,j} \ c \big(\tfrac{j}{k'}\big).
\]
Note that $
\Exp \big[ Z_{n,1+\lfloor \xi k' \rfloor} \big]  c \big(\tfrac{1+\lfloor \xi k' \rfloor}{k'}\big) \to \Exp[V_{\xi}]  c(\xi) = \frac{1}{\theta},
$ 
where $V_{\xi}\sim \mathrm{Exp}(\theta c(\xi))$, by continuity of $c$, Condition \ref{cond:uniform integr} and Lemma \ref{Zni_exp}.
Then, as in Section \ref{sec:EI}, it follows that
\[ 
\Exp[T_{n2}] = \frac{1}{k'} \sum_{j=1}^{k'} \Exp\Big[Z_{n,j} c \big(\tfrac{j}{k'}\big) \Big] 
\to 
\int_{0}^{1} \frac{1}{\theta} \ \mathrm{d}\xi = \frac{1}{\theta}. 
\]
Therefore, another sensible method of moments estimators for $\theta$ is given by 
\begin{align*}
\hat{\theta}_{n2} = \Big\{ \frac{1}{k'} \sum_{s=1}^{k'} \hat{Z}_{n,s} \ \hat{c}_n \Big(\frac{s}{k'}\Big) \Big\}^{-1}.
\end{align*}
%In the following, we write $\hat \theta_{n1}= \hat \theta_n$ for better distinction.

\begin{figure} [t!]
	\begin{center}
		\includegraphics[width=\textwidth]{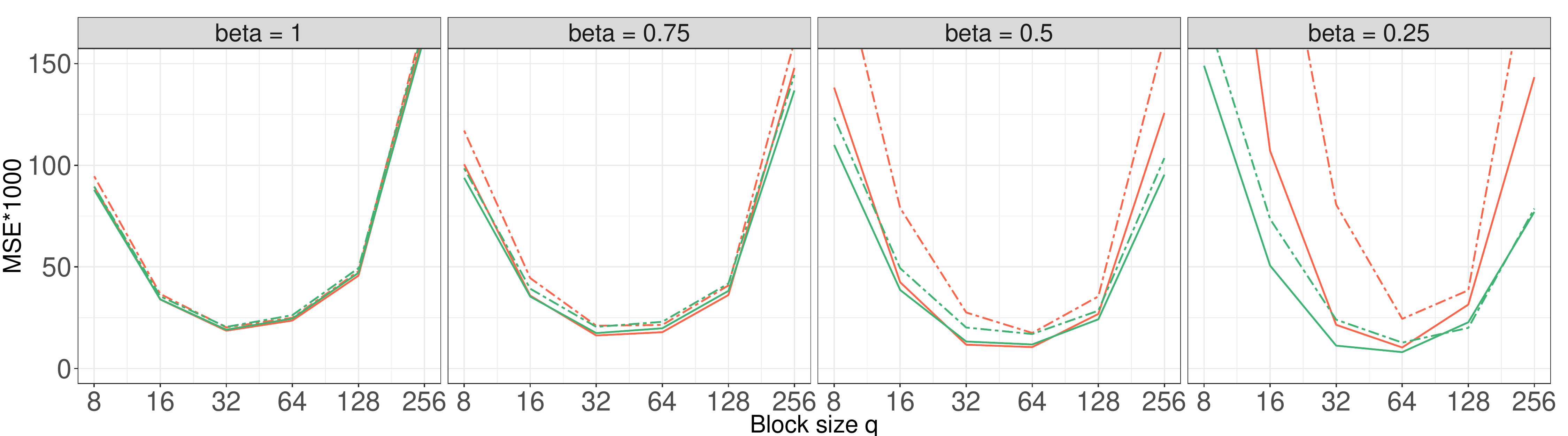} \vspace{-.5cm}
		\caption{Mean squared error, multiplied by $10^3$, for the estimation of $\theta$ in the ARCH-model with scedasis function $c_{2,\beta}$ for $\hat \theta_{n1}$ (orange lines) and $\hat \theta_{n2}$ (green lines) with $k=400$ (solid lines) and $k=300$ (dotted lines).} 
		\vspace{-.3cm} 
		\label{Fig:ei_c2_ARCH}
	\end{center}
\end{figure}

We only present results for the ARCH-model; the ARMAX- and independent model were found to yield very similar results.
Note that for $\lambda=0.7$ in the ARCH-model we have $\theta=0.721$, see Table 3.2 in \cite{DehResRooVri89}.

In what follows, the block size $\bs$ is chosen from the set $\{8,16,32,64,128,256\}$ (recall that $k'=\lfloor n/\bs \rfloor$) and the number of exceedances $k \in \{300,400\}$ are considered. (Here, slightly larger values of $k$ turned out to work better than in the context of testing for heteroscedastic extremes above.) Regarding the kernel density estimator, we set $\kappa=0.1$, set the bandwidth to $h =0.2$ and use the biweight kernel from \eqref{eq:biw}.

In Figure \ref{Fig:ei_c2_ARCH}, the mean squared error (MSE) of $\hat \theta_{n1}$ and $\hat \theta_{n2}$ is plotted as a function of the block size $q$, where the true scedasis function is given by $c_{2,\beta}$ for different values of $\beta$.  
One can see that the MSE-curves are mostly U-shaped, and that a minimum value is reached at an intermediate blocksize of $\bs \in \{32,64\}$. Further, in most scenarios the alternative estimator $\hat \theta_{n2}$ outperforms the estimator $\hat \theta_{n1}$, and the larger number of exceedances $k=400$ seems to work better than $k=300$ in terms of minimal MSE-values. The same observations were found for the other scedasis functions $c_{1,\beta}$, $\tilde c_{1,\beta}$ and $\tilde c_{2,\beta}$.

%%%%%%%%%%%%%%%%%%%%%%%%%%%%%%%%%%%%%%%%%%%%%%%%%%%%%%%%%%%%%%%%%%%%%%%%

\section{Proofs} \label{sec:proofs}

For space considerations, we only present the proofs for the theoretical results from Section~\ref{sec:sstep}, which are central to all other proofs. The remaining proofs for Sections~\ref{sec:scedasis}-\ref{sec:EI} are collected in a supplementary material.

%In this section, we collect all proofs. It is instructive to start with the central theoretical results from Section~\ref{sec:sstep}.

%\subsection{Proofs for Section \ref{sec:sstep}}

\begin{proof}[Proof of Proposition \ref{prop:sstep}]
Recall that $c_\infty(L)=1+ L\|c\|_\infty$. For $i\in\{1,\dots, n\}$ and $n\in\N$, define
\begin{align} \label{eq:xnprime}
X'_{n,i} =  \bigg( \frac{U_i^{(n)} - (1- \tfrac{k}{n}c_\infty(L)) }{\tfrac{k}{n}}\bigg)_+ = \max  \bigg( \frac{U_i^{(n)} - (1- \tfrac{k}{n}c_\infty(L)) }{\tfrac{k}{n}}, 0\bigg)
\end{align}
and let $v_n = \Pr(X'_{n,i} \ne 0) = \tfrac{k}{n} c_\infty(L)$. We may then write
\begin{multline*}
\mathbb{S}_n(s,x) 
=
\frac{1}{\sqrt{k}} \sum_{i=1}^{\ip{sn}}  \Big\{ \bm 1( X'_{n,i} > c_\infty(L)- c(\tfrac{i}{n}) x) - \Pr(X'_{n,i} > c_\infty(L)- c(\tfrac{i}{n}) x) \Big\} \\
+ {\sqrt k} \bigg\{ \frac{1}{n} \sum_{i=1}^{\ip{ns}} c(\tfrac{i}n)  - C(s)\bigg\} x \equiv \mathbb{S}_{n,1}(s,x)  + \mathbb{S}_{n,2}(s,x) .
\end{multline*}
As a consequence of \ref{cond:sbias}, the term $\Sb_{n,2}$ converges to zero, uniformly in $s$ and $x$, and we are left with investigating $\Sb_{n,1}$. We are going to identify that process with an empirical cluster process, see \cite{DreRoo10}. In the following we set $L=1$; the proof for arbitrary $L \in \N$ follows analogously. We also write $c_\infty=c_\infty(1)$.

Recall that $1< r < n$ denotes an integer sequence converging to infinity such that $r=o(n)$ as $n\to\infty$. Let $Y_{n,j}$ denote the $j$th block of consecutive values of $X'_{n,1}, \dots, X'_{n,n}$, i.e.,
\[
Y_{n,j} = (X'_{n,i})_{i \in I_j}, \quad I_j =\{ (j-1)r+1 , \dots,  jr\}, \quad j=1, \dots, m = \ip{n/r}.
\] 
We may then write
\begin{align*}
\Sb_{n,1}(s,x) 
&= 
c_\infty^{1/2} \bigg\{  \frac{1}{\sqrt{nv_n}} \sum_{i=1}^{rm}  \bm 1\big\{ X'_{n,i} > c_\infty- c(\tfrac{i}{n}) x, \tfrac{i}{n} \le s\big\} \\
&\hspace{5cm} - \Exp \bm 1\big\{X'_{n,i} > c_\infty- c(\tfrac{i}{n}) x, \tfrac{i}{n} \le s\big\} \bigg\}  + o_{\Prob}(1) \\
&=  
c_\infty^{1/2} \bigg\{ \frac{1}{\sqrt{nv_n}} \sum_{j=1}^{m}  \big\{\tilde f_{j,n,s,x}(Y_{n,j}) -  \Exp[\tilde f_{j,n,s,x}(Y_{n,j})] \big\} \bigg\} + o_\Prob(1) 
\end{align*}
where the $o_\Prob(1)$ is due to the fact that $mr\ne n$ in general, and where $\tilde f_{j,n,s,x}$ denotes the cluster functional (see \citealp{DreRoo10}, Definition 2.1)
\[
\tilde f_{j,n,s,x}(y_1, \dots, y_\ell) = \sum_{i=1}^\ell \bm 1(y_i > c_\infty - c (\tfrac{(j-1)r + i}{n}) x, \tfrac{(j-1)r + i}{n} \le s), \qquad \ell \in \N.
\]
Hence, we need to show functional weak convergence of $\{\tilde \Zb_n(s,x)  \}_{(s,x)}$, where
\begin{align}
\tilde \Zb_n(s,x) = \frac{1}{\sqrt{nv_n}} \sum_{j=1}^{m}  \big\{ \tilde f_{j,n,s,x}(Y_{n,j}) -  \Exp[ \tilde f_{j,n,s,x}((Y_{n,j})] \big\}. \label{Zn_tilde_def}
\end{align}
Unfortunately, results from  \cite{DreRoo10} are not directly applicable, as functions $f$ depending on $n$ (and, even more complicated, on $j$) are not allowed in their theory. Before proceeding, note that we may slightly redefine $\tilde f_{j,n,s,x}$. Indeed, let $\Zb_n$ be defined analogously to $\tilde \Zb_n$, but in terms of
\begin{align}
f_{j,n,s,x}(y_1, \dots, y_\ell) = \bm1(j \le \ip{sm}) g_{j,n,x}(y_1, \dots, y_\ell) \label{f_jnsx}
\end{align}
where
\[ 
g_{j,n,x}(y_1, \dots, y_\ell) =
\sum_{i=1}^\ell \bm 1(y_i > c_\infty - c (\tfrac{(j-1)r + i}{n}) x), \qquad \ell \in \N.
\]
Now, for all $s,x\in[0,1]$,
\begin{align}
|\Zb_n(s,x) - \tilde \Zb_n(s,x) | 
&\le
2 \frac{1}{\sqrt{nv_n}} \sum_{j=1}^{m} \sum_{i=1}^{r} | \bm1(\tfrac{j}{m} \le s) - \bm1(\tfrac{(j-1)r + i}{n} \le s) | \nonumber \\
& \le 
2 \frac{r}{\sqrt{nv_n}} \sum_{j=1}^{m} \bm1(\tfrac{j-1}{m} < s \le \tfrac{j+1}{m}) 
\le 
4 \frac{r}{\sqrt{nv_n}}. \label{Zn_tilde}
\end{align}
Recalling $v_n=\tfrac{k}{n} c_\infty$, we have $r= o(\sqrt{nv_n})$ by \ref{cond:mixing}. As a consequence, we have shown that
\begin{align}
  \mathbb S_n = c_\infty^{1/2} \Zb_n + \op \quad \textrm{in} \quad \ell^\infty([0,1] \times [0,L]), \label{Sn_Zn}
\end{align}
such that it is sufficient to show that the process~$\Zb_n$ converges to $c_\infty^{\scs -1/2} \Sb$.

Consider weak convergence of the fidis of $\Zb_n$ first, and for that purpose let us first assume that the blocks $Y_{n,1}, \dots, Y_{n,m}$ are independent. The general case will be reduced to the independent case by the Bernstein blocking technique below. 
Under the assumption of independent blocks, we may apply the Cram\'er-Wold device and the classical Lindeberg~CLT (\citealp{Bil95}, Theorem~27.2). We need to show that
\begin{align} \label{eq:cov1}
\lim_{n \to \infty} \mathfrak c_n((s,x), (s',x')),
=
c_\infty^{-1} \mathfrak c((s,x), (s',x')),
\end{align}
where $\mathfrak c$ is defined in \eqref{eq:cov2} and where
\begin{align}% \label{eq:cov1}
 \mathfrak c_n((s,x), (s',x')) =   \frac{1}{nv_n} \sum_{j=1}^{m} \Cov \big( f_{j,n,s,x}(Y_{n,j}), f_{j,n,s',x'}(Y_{n,j})\big), \label{cn_cov}
\end{align}
and that the Lindeberg condition is satisfied, that is, 
for any $(s,x) \in[0,1]^2$ and any $\eps>0$,
\begin{multline*} %\label{eq:lin}
\lim_{n\to\infty} \frac{1}{nv_n} \sum_{j=1}^{m} \Exp\Big[ \{ f_{j,n,s,x}(Y_{n,j})- \Exp f_{j,n,s,x}(Y_{n,j}) \}^2 \\
\bm 1( |f_{j,n,s,x}(Y_{n,j})- \Exp f_{j,n,s,x}(Y_{n,j})| > \eps \sqrt{nv_n}) \Big] = 0.
\end{multline*}

Observing that $|f_{j,n,s,x}| \le r$, the Lindeberg condition is actually  a simple consequence of the assumption $r =o(\sqrt{nv_n})$ in \ref{cond:mixing}, see also Corollary 3.6 in \cite{DreRoo10} for a similar argumentation. 

It remains to prove  \eqref{eq:cov1}, and for that purpose, we follow arguments from the proof of Remark 3.7 and Corollary 4.2 in \cite{DreRoo10}. First of all, since $| \Exp[f_{j,n,s,x}(Y_{n,j})] | \le r \Prob(X_{n,1} \ne 0) = rv_n$ for all $j=1, \dots, m$ and $s,x\in[0,1]$, we have that 
\begin{align} \label{eq:cnk}
 \mathfrak c_n((s,x), (s',x')) 
&=  \nonumber 
 \frac{1}{nv_n} \sum_{j=1}^{m} \Exp[ f_{j,n,s,x}(Y_{n,j}) f_{j,n,s',x'}(Y_{n,j})\big] + O(rv_n) \\
&=
\frac{r}{n} \sum_{j=1}^{m}  \bm 1(j \le \ip{(s \wedge s')m}) A_n(j) + O(rv_n)
\end{align}
where 
\[
A_n(j)
=
 \frac{1}{rv_n} \Exp\big[ g_{j,n,x}(Y_{n,j})g_{j,n,x'}(Y_{n,j})\big]
\]
and where the remainder is $o(1)$ by \ref{cond:mixing}.

Let us next calculate $A_{n}(j)$. For that purpose, recall the notion of the length of the core of a cluster $y$, denoted by $L(y)$, see Definition 2.1 in \cite{DreRoo10}. Let $K>0$ be a constant and decompose
\begin{align*}
A_n(j) 
& = \frac{1}{rv_n} \Exp\big[ g_{j,n,x}(Y_{n,j})g_{j,n,x'}(Y_{n,j}) \bm 1(L(Y_{n,j}) \le K) \big]  \\
& \hspace{2cm} +  \frac{1}{rv_n} \Exp\big[ g_{j,n,x}(Y_{n,j})g_{j,n,x'}(Y_{n,j}) \bm 1(L(Y_{n,j}) > K) \big]  \\
&= S_{n,K}(j) + R_{n,K}(j).
\end{align*}
By stationarity, we have
\begin{align*} 
R_{n,K}(j) 
&\le  \nonumber
\frac1{rv_n} \sum_{i,i'=1}^{r} \Prob \Big( X'_{n,i} > 0, X'_{n,i'} >0, L(Y_{n,1}) >K \Big)  \\
&\le \nonumber
\frac1{rv_n} \Exp\Big[ \Big(\sum_{i =1}^{r} \bm 1(X'_{n,i} > 0) \Big)^2 \bm 1(L(Y_{n,1}) >K) \Big]  \\
&\le  
\Big\{ \frac1{rv_n} \Exp\Big[ \Big(\sum_{i =1}^{r} \bm 1(X'_{n,i} > 0) \Big)^{2+\delta} \Big] \Big\}^{2/(2+\delta)} \Big\{ \frac1{rv_n} \Prob(L(Y_{n,1}) >K)  \Big\}^{\delta/(2+\delta)}.
\end{align*}
Thus, as a consequence of \ref{cond:expnum} and Lemma~5.2(vii) in \cite{DreRoo10}, which is applicable by \ref{cond:mixing},   we obtain that
\begin{align} \label{eq:rnk}
\lim_{K\to\infty} \limsup_{n\to\infty} \sup \{ R_{n,K}(j) : j=1, \dots, m\} = 0.
\end{align}
Further, for any $j\in\{1, \dots, m\}$,
\begin{align*}
S_{n,K}(j) = &\, 
\frac1{rv_n} \sum_{i,i'\in J_{n,j}} \Prob \Big( X'_{n,i} > c_\infty - c(\tfrac{i}n) x, X'_{n,i'} > c_\infty - c(\tfrac{i'}n) x', L(Y_{n,j}) \le K \Big)   \\
\nonumber = &\, 
\frac1{rv_n} \sum_{i,i'\in J_{n,j}, \atop  |i-i'|\le K} \Prob \Big( X'_{n,i} > c_\infty - c(\tfrac{i}n) x, X'_{n,i'} > c_\infty - c(\tfrac{i'}n) x', L(Y_{n,j}) \le K \Big)   \\
=&\, 
S_{n,K}'(j) + R_{n,K}'(j),
\end{align*}
where
\begin{align*}
S_{n,K}'(j) 
&= 
\frac1{rv_n} \sum_{i,i'\in J_{n,j},  \atop |i-i'|\le K} \Prob \Big( X'_{n,i} > c_\infty - c(\tfrac{i}n) x, X'_{n,i'} > c_\infty - c(\tfrac{i'}n) x' \Big),   \\
R_{n,K}'(j) &=
\frac1{rv_n} \sum_{i,i'\in J_{n,j}, \atop |i-i'|\le K} \Prob \Big( X'_{n,i} > c_\infty - c(\tfrac{i}n) x, X'_{n,i'} > c_\infty - c(\tfrac{i'}n) x', L(Y_{n,j}) >K \Big)  .
\end{align*}
By similar calculations as in \eqref{eq:rnk}, we have that
\begin{align} \label{eq:rnk2}
\lim_{K\to\infty} \limsup_{n\to\infty} \sup\{ R_{n,K}'(j) : j=1, \dots, m\} = 0.
\end{align}
Further, by Lemma~\ref{lem:probapp} and uniform continuity of~$c$, 
\begin{align*}
S_{n,K}'(j) 
&= 
\frac{1}{c_\infty r}  \sum_{i\in J_{n,j}}    c(\tfrac{i}n) \Big\{ d_0(x,x')  +  \sum_{h=1}^{K \wedge (r-i)} \{ d_h(x,x') + d_h(x',x) \} \Big\}  + o(1) \\
&= 
\frac{c(\tfrac{j-1}{m})}{c_\infty}  D_K(x,x')  + o(1) ,
\end{align*}
where the $o(1)$ is uniform in $x,x'\in[0,1]$ and $j=1, \dots, m$ and where
\[
D_K(x,x') =  d_0(x,x')  +  \sum_{h=1}^{K } \{ d_h(x,x') + d_h(x',x) \}.
\]
Assembling terms, we have 
\[
A_n(j) = \frac{c(\tfrac{j-1}{m})}{c_\infty}  D_K(x,x') + R_{n,K}(j) + R_{n,K}'(j) + o(1)
\]
where the $o(1)$ is uniform in $j=1, \dots, m$ and $x,x'\in[0,1]$.

As a consequence of the latter display and \eqref{eq:cnk}, we obtain that
\begin{align*}
\mathfrak c_n((s,x), (s',x')) 
&= 
\mathfrak c_{n,K}((s,x), (s',x'))  + \mathfrak r_{n,K}((s,x), (s',x'))  +o(1),
\end{align*}
where 
\begin{align*}
\mathfrak c_{n,K}((s,x), (s',x')) 
&=
c_\infty^{-1} \frac{r}{n} \sum_{j=1}^{\ip{(s \wedge s')m}}  c(\tfrac{j-1}{m}) D_K(x',x)  \\
\mathfrak r_{n,K}((s,x), (s',x')) 
&=
c_\infty^{-1} \frac{r}{n} \sum_{j=1}^{\ip{(s \wedge s')m}}   \{ R_{n,K}(j) + R_{n,K}'(j) \}
\end{align*}
By \eqref{eq:rnk} and \eqref{eq:rnk2}, we have
\[
\lim_{K\to\infty} \limsup_{n\to\infty} \mathfrak r_{n,K}((s,x), (s',x'))  = 0.
\]
Further,
\[
\lim_{n\to\infty} \mathfrak c_{n,K}((s,x), (s',x'))  = \frac{C(s\wedge s')}{c_\infty} D_K(x,x').
\]
We may finally apply Lemma~\ref{lem:seq} to conclude that \eqref{eq:cov1} is met.

The next step consists of getting rid of the assumption of independence of blocks. Recall that $1 < \ell_n < r$ denotes an integer sequence converging to infinity such that 
$\ell_n=o( r)$.
We may then write $\Zb_n(s,x) = \Zb_n^{+}(s,x) + \Zb_n^{-}(s,x)$, where
\begin{align*}
\Zb_n^{+}(s,x) 
&= 
\frac{1}{\sqrt{nv_n}} \sum_{j=1}^{\ip{sm}}  \sum_{i=(j-1)r+1}^{jr-\ell_n} \bm 1(X'_{n,i}> c_\infty-c(\tfrac{i}n)x) - \Pr(X'_{n,i}> c_\infty-c(\tfrac{i}n)x) \\
\Zb_n^{-}(s,x) 
&=
\frac{1}{\sqrt{nv_n}} \sum_{j=1}^{\ip{sm}}  \sum_{i=jr-\ell_n+1}^{jr} \bm 1(X'_{n,i}> c_\infty-c(\tfrac{i}n)x) - \Pr(X'_{n,i}> c_\infty-c(\tfrac{i}n)x).
\end{align*}
Further, for $n\in\N$, let $Y_{n,1}^*,\dots, Y_{n,m}^*$ denote an i.i.d.\ sequence, where $Y_{n,1}^*$ is equal in distribution to $Y_{n,1}$. Let $\Zb_n^*, \Zb_n^{\scs +,*}$ and $\Zb_n^{\scs -,*}$ be defined analogously to $\Zb_n, \Zb_n^{ +}$ and $\Zb_n^{-}$, but in terms of $Y_{n,1}^*,\dots, Y_{n,m}^*$. We will show that:
\begin{compactenum}
\item[(i)] For any $s,x\in[0,1]$, we have $\Zb_n^{-,*}(s,x) =o_\Prob(1)$ and $\Zb_n^{-}(s,x) = o_\Prob(1)$.
\item[(ii)] The fidis of $\Zb_n^{\scs +,*}$ converge weakly if and only if the fidis of $\Zb_n^+$ converge weakly. In that case, the weak limits coincide.
\end{compactenum}
As a consequence, the asymptotic distribution of the fidis of $\Zb_n$ coincides with the asymptotic distribution of the fidis of $\Zb_n^{*}$, and the latter has already been derived above.

\medskip
\noindent \textbf{Proof of (i).} 
Let us first show that $\Zb_n^{\scs -,*}(s,x) = o_\Prob(1)$, which follows if we show that $\Var(\Zb_n^{\scs -,*}(s,x)) = o(1)$. Now, by stationarity,
\begin{align}
\Exp\Big\{\sum_{i=1}^{r} \bm 1(X'_{n,i} \ne 0) \Big\}^2
&\ge \nonumber
\Exp\sum_{j=1}^{\ip{r/\ell_n}}\Big\{\sum_{i=(j-1)\ell_n+1}^{j\ell_n} \bm 1(X'_{n,i} \ne 0) \Big\}^2  \\
&= \Big\lfloor \frac{r}{\ell_n}\Big\rfloor \Exp\Big\{\sum_{i=1}^{\ell_n} \bm 1(X'_{n,i} \ne 0) \Big\}^2. \label{exp_neq_zero}
\end{align}
As a consequence, by independence of blocks, stationarity and \ref{cond:expnum},
\begin{align*}
\Var(\Zb_n^{-,*}(s,x)) 
\le 
\frac{m}{nv_n} \Var\Big(\sum_{i=1}^{\ell_n} \bm 1(X'_{n,i} \ne 0) \Big)  = O(\ell_n/r),
\end{align*}
which converges to $0$ by the assumption on $\ell_n$.

Now, consider $\Zb_n^{-}(s,x)$. 
Split the sum into two sums $\Zb_n^{\scs -,\rm even}(s,x)$ and $\Zb_n^{\scs -,\rm odd}(s,x)$, according to whether $j$ is even or odd. It suffices to show that each of these sums is $o_\Prob(1)$. We only consider the sum over the even blocks; the argumentation for the odd blocks is similar. Now, since the observations making up the even numbered blocks are separated by $r$ observations, we may follow the argumentation in \cite{Ebe84} to obtain that
\begin{align} \label{eq:dtv1}
d_{\rm TV}( P^{(Y_{n,2j})_{1\le j \le \ip{m/2}}} , P^{(Y_{n,2j}^*)_{1\le j \le \ip{m/2}}})
\le
 \ip{m/2} \beta(r),
\end{align}
where $d_{\rm TV}$ denotes the total variation distance between two probability laws.
Since $m\beta(\ell_n)=o(1)$ by \ref{cond:mixing}, the latter display is $o(1)$. As a consequence, $\Zb_n^{-,\rm even}(s,x) = \Zb_n^{\scs -,\rm even,*}(s,x)+o_\Prob(1)$. Finally, $\Zb_n^{\scs -,\rm even,*}(s,x)=o_\Prob(1)$ by the same reasoning as for $\Zb_n^{\scs -, *}$. 

\medskip
\noindent \textbf{Proof of (ii).} Note that $\Zb_n^{+}$ only depends on $(Y_{n,j}^{\scs (r-\ell_n)})_{1 \le j \le m}$, where $Y_{n,j}^{\scs (r-\ell_n)}$ consists of the first $r - \ell_n$ coordinates of   $Y_{n,j}$. A similar assertion holds for $\Zb_n^{\scs +, *}$, which is defined in terms of $((Y_{n,j}^*)^{\scs (r-\ell_n)})_{1 \le j \le m}$. The assertion in (ii) follows from the fact that
\[
d_{\rm TV}(  P^{(Y_{n,j}^{(r-\ell_n)})_{1\le j \le m}} , P^{((Y_{n,j}^*)^{(r-\ell_n)})_{1\le j \le m}})
\le
m\beta(\ell_n) \to 0
\]
by assumption and since the respective shortened blocks  are separated by $\ell_n$ observations.

\medskip

It remains to show asymptotic tightness. For that purpose, decompose $\Zb_n = \Zb_n^{\rm even} + \Zb_n^{\scs \rm odd}$ and likewise $\Zb_n^* = \Zb_n^{\scs \rm even,*} + \Zb_n^{\scs \rm odd,*}$ into sums over even and odd numbered blocks. Clearly, asymptotic tightness of $\{\Zb_n(s,x)\}_{(s,x)\in[0,1]^2}$ follows from asymptotic tightness of $\{\Zb_n^{\rm even}(s,x))\}_{(s,x)\in[0,1]^2}$ and $\{\Zb_n^{\scs \rm odd}(s,x)\}_{(s,x)\in[0,1]^2}$. We only consider the even numbered blocks. In view of \eqref{eq:dtv1}, 
it is further sufficient to show asymptotic tightness of $\{\Zb_n^{\scs \rm even,*}(s,x)\}_{(s,x)\in[0,1]^2}$. To reduce the notational complexity, we instead prove asymptotic tightness of $\{\Zb_n^{*}(s,x)\}_{(s,x)\in[0,1]^2}$. For that purpose, we apply Theorem 11.16 in \cite{Kos08}, with $t$ in that theorem replaced by $(s,x)$, and with 
\[
f_{n,j}(\omega;(s,x)) =  \bm 1( \tfrac{j}{m} \le s) \times \frac{1}{\sqrt{nv_n} }\sum_{i\in J_{n,j}} \bm 1\big(X_{n,i}^{'*}(\omega) > c_\infty - c(\tfrac{i}n) x \big),
\]
where $\omega$ is an element of the underlying probability space on which the $X_{n,i}^{'*}$ are defined.
We need to show that

\medskip
\begin{compactenum}
\item[(1)]  $\lbrace f_{n,j}: j=1, \dots,m \rbrace$ is almost measurable Suslin (AMS);
\item[(2)] the $\lbrace f_{n,j} \rbrace$ are manageable with envelopes $\{F_{n,j}\}$ given
      through 
      \[
      F_{n,j}(\omega) := \frac{1}{\sqrt{nv_n}} \sum_{i\in J_{n,j}} \bm1 (X_{n,i}^{'*}(\omega) \ne 0);
      \]
\item[(3)]  $\lim_{n \rightarrow \infty} \Exp \{\Zb_n^{*}(s,x) \Zb_n^{*}(s',x') \}$ exists for all $(s,x),(s',x') \in [0,1]^2$;
\item[(4)]  $\limsup_{n \rightarrow \infty} \sum_{j =1}^{m} \Exp F_{n,j}^2 < \infty$;
\item[(5)]  $\lim_{n \rightarrow \infty} \sum_{j =1}^{m} \Exp F_{n,j}^2 \bm1(F_{n,j}> \eps)=0$
      for all $\eps >0$;
\item[(6)]  $\rho(s,x;s',x') = \lim_{n \rightarrow \infty} \rho_n(s,x;s',x')$ exists for every
      $(s,x),(s',x') \in [0,1]^2$, where
	\begin{align} \label{eq:rhon}
	\rho_n(s,x;s',x') := \bigg\{ \sum_{j=1}^{m} \Exp \left| f_{n,j} (\cdot;s,x) -
			f_{nj}( \cdot; s',x') \right|^2 \bigg\}^{1/2}.
	\end{align}
	 [In that case, $\rho$ defines a semimetric on $[0,1]^2$.]
	Moreover, $\rho_n(s_n, x_n ; s'_n,x'_n) \rightarrow 0$ for all sequences $(s_n,x_n)_{n \in \mathbb N} $, $(s'_n,x'_n)_{n \in \mathbb N} \subset [0,1]^2$ such that $\rho(s_n,x_n;s'_n,x'_n) \rightarrow 0$.
\end{compactenum}
\medskip

\noindent \textbf{Proof of {\rm (1)}.} By Lemma 11.15 in \cite{Kos08}, the triangular array $\lbrace f_{n,j} \rbrace$ is AMS provided it is separable, that is, provided that,
for every $n \in \mathbb N$, there exists a countable subset $S_n \subset [0,1]^2$ such that
\[
\Prob^{*}\bigg(\sup_{(s,x) \in [0,1]} \inf_{(s',x') \in S_n} \sum_{j=1}^{m}  \{ f_{n,j}(\omega; s,x) -
f_{n,j}(\omega;s',x') \} ^2 >0\bigg)=0.
\]
Define $S_n := (\mathbb Q\cap[0,1])^2 $ for all $n \in \mathbb N$.
Then, for every element $\omega$ of the underlying probability space and for every $(s,x) \in[0,1]^2$, there exists 
$(s',x') \in S_n$ such that
\[
\sum_{j=1}^{m}  \{ f_{n,j}(\omega; s,x) -
f_{n,j}(\omega;s',x') \} ^2 =0.
\]

\medskip
\noindent \textbf{Proof of {\rm (2)}.}
By Theorem 11.17(iv) in \cite{Kos08}, it suffices to prove that the triangular arrays
$
\{ \tilde f_{n,j} (\omega;x) = \frac{1}{\sqrt{nv_n} }\sum_{i\in J_{n,j}} \bm 1\big(X_{n,i}^{'*}(\omega) > c_\infty - c(\tfrac{i}n) x \big):  x \in [0,1] \}_{j=1, \dots, m}
$
and
$
\{ \tilde g_{nj} (\omega;s) = \bm 1({j}/{m} \le s) : s \in [0,1]\}_{j=1, \ldots, m}
$
are manageable with respective envelopes
$\{ F_{n,j} (\omega)\}_{ j=1, \ldots,m }$ and
$\{\tilde G_{nj} (\omega) \equiv 1 \}_{ j=1, \ldots,m }$. Following the discussion on Page~221 in \cite{Kos08}, these two assertions are consequences of the fact that both $\tilde f_{n,j}$ and $\tilde g_{n,j}$ are increasing in $x$ and $s$, respectively.

\medskip
\noindent \textbf{Proof of {\rm (3), (4)} and {\rm (5)}.}
Condition (3) is simply the calculation of $\mathfrak c((s,x),(s',x'))$ above. 
Condition (4) is a consequence of \ref{cond:expnum}. Moreover, the  assumption $r=o(\sqrt{nv_n})$ in \ref{cond:mixing} implies (5).

\medskip
\noindent \textbf{Proof of {\rm (6)}.} Let
\[
\sigma^2(x,x')= d_0(x,x') + \sum_{h =1}^\infty  \big( d_h( x,  x') + d_h(x',x) \big).
\]
For $(s,x), (s',x')\in[0,1]^2$, let $\bar x = x$ if $s\ge s'$ and $\bar x=x'$ else. Then, by similar arguments that lead to \eqref{eq:cov1}, we have
\begin{align*}
& \mathrel{\phantom{=}}\rho_n^2(s,x;s',x')  \\
&=
\frac1{nv_n} \bigg\{ \sum_{j=1}^{\ip{(s\wedge s')m}} \Exp \big\{ g_{j,n,x}(Y_{n,j}) -  g_{j,n,x'}(Y_{n,j}) \big\}^2 
+   \sum_{j=\ip{(s\wedge s')m}+1}^{\ip{(s\vee s')m}} \Exp\big\{ g_{j,n, \bar x}(Y_{n,j}) \big\}^2 \bigg\} \\
&=
 c_\infty^{-1}\Big\{ C(s\wedge s')  \{ \sigma^2(x,x) - 2 \sigma^2(x,x') + \sigma^2(x',x') \}  \\
 & \hspace{4cm} + \{ C(s\vee s') -C(s\wedge s') \} \sigma^2(\bar x, \bar x) \Big\} 
+o(1) \\
&= \rho^2((s,x),(s',x')) +o(1),
\end{align*}
for any fixed $(s,x), (s',x')\in[0,1]^2$. In order to show the convergence along sequences as claimed in (6), it is sufficient to show that the convergence in the last display is in fact uniform. Note that the argumentation used for pointwise convergence does not imply uniform convergence, due to the pointwise nature of the main argument, Lemma~\ref{lem:seq}.

Let $t_j=(s_j,x_j,s'_j,x'_j) \in [0,1]^4, j=1,2$. Suppose we have shown that
\begin{align}\label{eq:rho}
|\rho_n^2(t_1) - \rho_n^2(t_2)| \lesssim H_n(t_1,t_2)  
\end{align}
with
\[
H_n(t_1, t_2)= h_0(|x_1-x_2|+q_n) +h_0( |x'_1-x'_2|+q_n) + |s_1-s_2| + |s_1'-s_2'| +q_n ,
\]
where $q_n$ denotes a sequence converging to zero (independent of $t_1, t_2$), where $h_0$ denotes a continuous, non-negative, increasing function on $[0,1]$ with $h_0(0)=0$ 
and where the symbol `$\lesssim$' means that the left-hand side is bounded by a constant multiple of the right-hand side, the constant being independent of $n, t_1, t_2$.
By pointwise convergence of $\rho_n^2$, we then also have 
\begin{align} \label{eq:rho2}
|\rho^2(t_1) - \rho^2(t_2)| \lesssim H(t_1, t_2),
\end{align} 
where
\[
H(t_1, t_2)= h_0( |x_1-x_2|) + h_0( |x'_1-x'_2|) + |s_1-s_2| + |s_1'-s_2'|.
\]

Now, let $\eps>0$ be given. Then, by uniform continuity of $h_0$, there exists $\delta>0$ such that $H_n(t_1,t_2) < \eps$ and $H(t_1,t_2)<\eps$ for all $\|t_1-t_2\|_2 < \delta$ and for all $n$ sufficiently large. Choose a finite grid of points $t^{(1)}, \dots, t^{(p)}$ such that each point $t\in[0,1]^4$ lies in the open ball of radius $\delta$ with center $t^{(j)}$, for some $j=1, \dots, p$. Then, 
\begin{align*}
|\rho_n^2(t) - \rho^2(t)|  
&\le 
|\rho_n^2(t) - \rho_n^2(t^{(j)})| + |\rho_n^2(t^{(j)}) - \rho^2(t^{(j)})|  + |\rho^2(t^{(j)}) - \rho^2(t)|   \\
&\lesssim
2 \eps +  \max_{j=1}^{p} |\rho_n^2(t^{(j)}) - \rho^2(t^{(j)})| .
\end{align*}
The upper bound does not depend on $t$, and converges to $2\eps$ for $n\to\infty$ by pointwise convergence. Since $\eps>0$ was arbitrary, we obtain that $\rho_n^2\to\rho^2$ uniformly.

It remains to show \eqref{eq:rho}. Let $s_{j}^\vee=s_j \vee s_j'$ and $s_j^\wedge=s_j \wedge s_j'$. Up to symmetry, we need to distinguish three cases:
\[
s_2^\vee \le s_1^\wedge, \quad 
s_2^\wedge \le s_1^\wedge \le s_2^\vee \le s_1^\vee, \quad
s_2^\wedge \le s_1^\wedge \le s_1^\vee \le s_2^\vee.
\]
For brevity, we only consider the first case, and make the further assumption that $s_2 < s_2' < s_1 < s_1'$. Introduce the notation $G_j(x) = g_{n,j,x}(Y_{n,j})$.
We may then write
\[
\rho_n^2(t_1) - \rho_n^2(t_2) = a_{n1}+a_{n2}+a_{n3} + a_{n4},
\]
where
\begin{align*}
a_{n1} 
&= 
\textstyle \frac1{nv_n} \sum_{j=1}^{\ip{ s_{2}m }} \Exp \Big[ \big\{G_j(x_1) - G_j(x_1') \big\}^2 - \big\{G_j(x_2) - G_j(x_2') \big\}^2  \Big], \\
a_{n2} 
&= 
\textstyle\frac1{nv_n} \sum_{\ip{ s_2m}+1}^{\ip{s_2'm} } \Exp \Big[ \big\{G_j(x_1) - G_j(x_1') \big\}^2 - \big\{G_j(\bar x_2)  \big\}^2  \Big], \\
a_{n3} 
&= 
\textstyle\frac1{nv_n} \sum_{\ip{ s_2'm}+1}^{\ip{s_1m} } \Exp \Big[ \big\{G_j(x_1) - G_j(x_1') \big\}^2 \Big], \\
a_{n4} 
&= 
\textstyle\frac1{nv_n} \sum_{\ip{ s_1m}+1}^{\ip{s_1'm} } \Exp \Big[ \big\{G_j(\bar x_1)\}^2 \Big].
\end{align*}
Note that $\Exp\{G_j(x)  \big\}^2 \le \Exp \{\sum_{i\in J_{n,1}} \bm 1(X'_{n,i} > 0 \}^2 = O(rv_n)$, uniformly in $x$ and $j=1, \dots, m$, by Condition~\ref{cond:expnum}. Hence, 
\begin{align*}
|a_{n2}| 
&\lesssim 
\frac{\ip{s_2' m}-\ip{s_2m}}{m} \le s_2'-s_2 + m^{-1} \le |s_1-s_2|+m^{-1}.
\end{align*}
Similarly, $|a_{n3}|$ and $|a_{n4}|$ are bounded by a constant multiple of $|s_1'-s_2'|+m^{-1}$. It remains to treat $|a_{n1}|$. The triangular inequality and the Cauchy-Schwarz-inequality imply that each summand of $|a_{n1}|$ can be bounded by
\begin{align*}
& \mathrel{\phantom{=}} \Exp \Big[ \big| G_j(x_1)  - G_j(x_1') + G_j(x_2) - G_j(x_2') \big| \cdot \big| G_j(x_1)  - G_j(x_1') - G_j(x_2) + G_j(x_2')\big|   \Big] \\
&\le
\Big\{\Exp\big| G_j(x_1)  - G_j(x_1') + G_j(x_2) - G_j(x_2') \big|^2\Big\}^{1/2}  \\
& \hspace{3cm} \times \bigg[ \Big\{ \Exp  \big| G_j(x_1)  - G_j(x_2) \big|^2 \Big\}^{1/2} + \Big\{ \Exp  \big| G_j(x_1')  - G_j(x_2') \big|^2 \Big\}^{1/2}  \bigg]
\end{align*}
The first factor is of the order $O((rv_n)^{1/2})$ by Condition~\ref{cond:expnum}, uniformly in $j=1, \dots, m$ and the $x$-arguments. Regarding the second factor note that, by Hölder-continuity of $c$ as assumed in Condition~\ref{cond:sbias},  we have
\[
0< c(\tfrac{j r} n) - K_c (\tfrac{r}n)^{1/2} \le c(\tfrac{i}n) \le c(\tfrac{j r } n) + K_c (\tfrac{r}n)^{1/2} \qquad \forall\ i \in J_{n,j},
\]
for sufficiently large $n$. Without loss of generality, let $x_1\le x_2$.
Then, by monotonicity and Condition~\ref{cond:hbound},
\begin{align} \label{eq:momb1}
\nonumber
&\mathrel{\phantom{=}} \Exp  \big| G_j(x_1)  - G_j(x_2) \big|^2  \\
\nonumber
&=
\Exp \Big\{ \sum_{i \in J_{n,j}}  \bm 1( c_\infty - c(\tfrac{i}n) x_1  \ge X'_{n,i} > c_\infty - c(\tfrac{i}n) x_2) \Big\}^2 \\
\nonumber
&\le 
\Exp \Big\{ \sum_{i \in J_{n,j}} \bm1( c_\infty - \{c(\tfrac{j r} n) - K_c(\tfrac{r}n)^{1/2}\} x_1 \ge X'_{n,i} > c_\infty - \{c(\tfrac{j r} n) +  K_c(\tfrac{r}n)^{1/2} \} x_2\Big\}^2 \\
\nonumber
&\le
h\big(c(\tfrac{j r}n)(x_2-x_1) + K_c(\tfrac{r}n)^{1/2}(x_1+x_2)\big)  r\tfrac{k}n \\
&\le 
h\big(c_\infty (x_2 - x_1) + 2 K_c m^{-1/2}\big)r v_n
\end{align}
As a consequence, 
\[
|a_{n1}| \lesssim h^{1/2}(c_\infty|x_1-x_2| + 2K_cm^{-1/2}) + h^{1/2}(c_\infty|x_1'-x_2'| + 2K_cm^{-1/2})
\]
which finally proves \eqref{eq:rho} with $h_0(x) =h^{1/2}(c_\infty x)$ and $q_n= 2K_cm^{-1/2}/c_\infty$.
\end{proof}

\begin{proof}[Proof of Proposition \ref{prop:step}]
Let $(s,x) \in [0,1] \times [0,L]$. Set $\varepsilon_n(x)=V(n/(kx)) = F^{-1}(1-kx/n)$ such that, almost surely,
\begin{align*} 				  
\mathbb{F}_n(s,x) 
&= 
\sqrt{k} \bigg\{ \frac{1}{k}\sum_{i=1}^{n} \bm{1}\Big\{ U_{i}^{(n)} > F_{n,i}\Big(V\Big(\frac{n}{kx}\Big)\Big)  \Big\} \I(i/n \leq s) - xC(s)  \bigg\} \\
&= 
\sqrt{k} \bigg\{ \frac{1}{k}\sum_{i=1}^{n} \bm{1}\Big\{ U_{i}^{(n)} > 1-\frac{kx}{n} \frac{1-F_{n,i}(\varepsilon_n(x))}{1-F(\varepsilon_n(x))}  \Big\} \I(i/n \leq s)  - xC(s)  \bigg\}.
\end{align*}
According to Condition \ref{cond:2nd order}, there exist $y_0 < x^{\ast}$  and $\tau > 0$ such that, for all $y > y_0, n \in \N, 1 \leq i \leq n$,
\[ 
c(i/n) \Big\{ 1- \frac{\tau}{c_{\min}} A \Big(\frac{1}{1-F(y)}\Big) \Big\} 
\le 
\frac{1-F_{n,i}(y)}{1-F(y)} 
\le 
c(i/n) \Big\{ 1 + \frac{\tau}{c_{\min}} A \Big(\frac{1}{1-F(y)}\Big) \Big\}. 
\]
Since $\varepsilon_n(x) \to x^{\ast}$, this implies, for $n$ large enough, 
\begin{align} 
\Big\{ U_i^{(n)} > 1- c(i/n) (1-\delta_n) \frac{kx}{n} \Big\}
& \subseteq \nonumber
\Big\{ U_i^{(n)} > 1- \frac{1-F_{n,i}(\varepsilon_n(x))}{1-F(\varepsilon_n(x))} \frac{kx}{n} \Big\} \\
& \subseteq \label{eq:uisubset}
\Big\{ U_i^{(n)} > 1- c(i/n) (1+\delta_n) \frac{kx}{n} \Big\},
\end{align}
where $\delta_n = \sup_{x \in (0,L]} \frac{\tau}{c_{\min}} A\big(\frac{n}{kx}\big) = \frac{\tau}{c_{\min}} A\big(\frac{n}{kL}\big)$.
As a consequence, by the definition of $\mathbb S_n$ in \eqref{eq:defsstep}, almost surely
\begin{align}
\mathbb{S}_n(s,x(1-\delta_n)) - \sqrt{k} \delta_n xC(s) \leq \mathbb{F}_n(s,x) \leq \mathbb{S}_n(s,x(1+\delta_n)) + \sqrt{k} \delta_n  xC(s). \nonumber %\label{ineq_Fn}
\end{align}
	Therefore, 
	\begin{align*}
		\sup_{(s,x)\in [0,1] \times [0,L]} |\mathbb{F}_n(s,x)-\mathbb{S}_n(s,x)| \leq 2 w_{\delta_n}(\mathbb{S}_n) + 2 \sqrt{k}\delta_n,
	\end{align*}
	where, for $\delta > 0$,
	\begin{align} 
	w_{\delta}(\mathbb{S}_n) = \sup_{(s,y),(s,z)\in [0,1]^2: |y-z| < \delta} |\mathbb{S}_n (s,y) -\mathbb{S}_n(s,z)|. \label{def:w_delta}
	\end{align}
	Now, since $\sqrt{k}\delta_n \leq \sqrt{k}\frac{\tau}{c_{\min}}A\big(\frac{n}{kL}\big) =o(1)$ by Condition \ref{cond:2nd order}, it suffices to show that, for any $\varepsilon>0$, 
	\begin{align}
	  \limsup_{n \to \infty} \Prob(w_{\delta_n}(\mathbb{S}_n)> \varepsilon) = 0. \label{w_delta_Sn}
	\end{align}
	For arbitrary $\delta > 0$, we have
	\begin{align} 
		\Prob(w_{\delta_n}(\mathbb{S}_n) > \varepsilon) 
		&= \Prob(w_{\delta_n}(\mathbb{S}_n) > \varepsilon, \delta_n < \delta) + \Prob(w_{\delta_n}(\mathbb{S}_n) > \varepsilon, \delta_n \geq \delta) \nonumber \\
		&\leq \Prob(w_{\delta}(\mathbb{S}_n) > \varepsilon)	+ o(1).	\nonumber %\label{w_delta}
	\end{align}
	In the following we set $L=1$ in order to be able to refer to the proof of Proposition~\ref{prop:sstep} in an easier manner; the general case $L \in \N$ can again be shown analogously.
	
Consider the semimetric $\rho$ on $[0,1]^2$ defined in the proof of Proposition \ref{prop:sstep}, see \eqref{eq:rhon}.	 By Theorem 11.16 in \cite{Kos08}, we know that $[0,1]^2$ is totally bounded under $\rho$. Further, by \eqref{eq:rho2}, 
$\rho((s,y),(s,z)) \lesssim h^{\scs 1/2}_0(|y-z|) \leq h^{\scs 1/2}_0(\delta)$ for all $s,y,z \in [0,1]$ with $|y-z|<\delta$, where $h_0$ is non-decreasing and continuous with $h_0(0)=0$.
Consequently, 
\begin{align*}
& \phantom{{}={}} \lim_{\delta \downarrow 0} \limsup_{n \to \infty} \Prob(w_{\delta}(\mathbb{S}_n) > \varepsilon) \\
& \leq   \lim_{\delta \downarrow 0} \limsup_{n \to \infty} \Prob \bigg( \sup_{(s,y),(s,z)\in [0,1]^2: \rho((s,y),(s,z)) < \delta} |\mathbb{S}_n(s,y)-\mathbb{S}_n(s,z)| > \varepsilon \bigg),
\end{align*}
which equals $0$ by Theorem 7.19  and Theorem 11.16 in \cite{Kos08}, the latter being applicable because of the proof of Proposition \ref{prop:sstep}.
\end{proof}

%%%%%%%%%%%%%%%%%%%%%%%%%%%%%%%%%%%%%%%%%%%%%%%%%%%%%%%%%%%%%%%%%%%%%%%%%%%%%%%%%%%

\section{Auxiliary Results} \label{sec:auxiliary}

\begin{lemma} \label{lem:probapp}
Under the assumptions of Proposition~\ref{prop:sstep}, for any fixed $h\ge 0$, we have that
\begin{align*} %\label{eq:probapp}
\sup_{x,x'\in[0,L]} \sup_{i=1, \dots n}  \bigg| \frac{1}{v_n} \Prob\Big(X'_{n,i} > c_\infty - c(\tfrac{i}n) x, X'_{n,i+h}>c_\infty - c(\tfrac{i+h}n) x' \Big)  
 - \frac{c(\tfrac{i}n)}{c_\infty} d_h(x,x') \bigg|  
\end{align*}
converges to $0$ as $n\to\infty$, where $v_n=v_n(L)=\tfrac{k}nc_\infty$ with $c_\infty=c_\infty(L)$.
\end{lemma}

\begin{proof}
First note that, as a consequence of \eqref{eq:tp} and the continuous mapping theorem, for any $\ell \in \N$ and with $c_n=c_\infty \tfrac{k}n$,
\begin{align*}
&\hspace{-1cm} \Prob\big( (X'_{n,1}  \dots, X'_{n,\ell} ) \in dx  \mid  X'_{n,1} > 0 \big) \\
=  &\,
\Prob\Big( \big(c_\infty\{1-(c_nZ_1)^{-1}\}_+, \dots, c_\infty\{1-(c_nZ_\ell)^{-1}\}_+\big) \in  dx \, \Big|\,  Z_1 > c_n^{-1}\Big) \\
\dto  &\,
\Prob( (W_1 , \dots, W_{\ell}) \in dx ),
\end{align*}
where $W_j=c_\infty(1-1/Y_{j-1})_+$. Note that $W_1$ is standard uniform on $(0,c_\infty)$ and that $W_j\ge0$ may have an atom at zero and is absolutely continuous on $(0,c_\infty)$, for $j\ge 2$. A simple extension of Lemma 2.11 in \cite{Van98} implies that
\[
\sup_{x_1, \dots, x_\ell >0 } \big| \Prob( X'_{n,1} > x_1  \dots, X'_{n,\ell}  >x_\ell  \mid  X'_{n,1} > 0 ) - \Prob(W_1 > x_1, \dots, W_\ell > x_\ell) \big| = o(1).
\]
Thus, for $h\ge 0$ fixed, by uniform continuity of $c$ and $r=o(n)$,
\begin{align*}
&\hspace{-1cm} \frac{1}{v_n} \Prob\Big(X'_{n,i} > c_\infty - c(\tfrac{i}n) x, X'_{n,i+h}>c_\infty - c(\tfrac{i+h}n) x' \Big) \\
&=
\Prob\Big(X'_{n,i} > c_\infty - c(\tfrac{i}n) x, X'_{n,i+h}>c_\infty - c(\tfrac{i+h}n) x'  \mid X'_{n,i} > 0 \Big) \\
&=
\Prob\Big(W_1 > c_\infty - c(\tfrac{i}n) x, W_{h+1} > c_\infty - c(\tfrac{i+h}n) x' \Big)   + o(1) \\
&=
\Prob\Big(W_1 > c_\infty - c(\tfrac{i}n) x, W_{h+1} > c_\infty - c(\tfrac{i}n) x' \Big)   + o(1) ,
\end{align*}
where  the $o(1)$ is uniform in $i=1, \dots, n$ and $x,x' \in[0,L]$.
Further, by the spectral decomposition of $(Y_t)_{t\in\N_0}$ (Theorem 3.1 in \citealp{BasSeg09}), that is $(Y_t)_{t\in\N_0}=(Y_0 \Theta_t)_{t\in\N_0}$ for some process $(\Theta_t)_{t\in\N_0}$ independent of $Y_0$ and with $\Theta_0=1$, we obtain, by a change of variable,
\begin{align*}
&\hspace{-1cm} \Prob\Big(W_1 > c_\infty - c(\tfrac{i}n) x, W_{h+1} > c_\infty - c(\tfrac{i}n) x' \Big)  \\
&= 
\Prob\Big(Y_0 > \frac{c_\infty }{ c(\tfrac{i}n) x} , Y_{h} > \frac{c_\infty }{ c(\tfrac{i}n) x'}\Big)   \\
&= 
\int_{1}^\infty \Prob\Big( Y_0  > \frac{c_\infty }{ c(\tfrac{i}n) x} , Y_0 \Theta_{h} > \frac{c_\infty }{ c(\tfrac{i}n) x'} \mid Y_0 = y\Big)  y^{-2} \, dy \\
&= 
\int_{\frac{c_\infty}{c(i/n)x}}^\infty \Prob\Big( \Theta_{h} > \frac{c_\infty }{ y c(\tfrac{i}n) x'} \Big) y^{-2} \, dy \\
&= 
\frac{c(\tfrac{i}n)}{c_\infty} \int_{1/x}^\infty \Prob\Big( \Theta_{h} > \frac{1}{ zx'} \Big) z^{-2} \, dz \\
&= 
\frac{c(\tfrac{i}n)}{c_\infty}  \Prob\Big(Y_0 > \frac{1}{ x} , Y_{h} > \frac{1}{x'}\Big) 
=
\frac{c(\tfrac{i}n)}{c_\infty} d_{h}(x,x'),
\end{align*} 
which implies the assertion.
\end{proof}

\begin{lemma}\label{lem:seq}
Let $(a_n)_{n\in\N}$, $(c_k)_{k\in\N}$ and $(r_{n,k})_{(n,k)\in\N^2}$ be sequences satisfying
\[
a_n=c_k+r_{n,k} \quad \text{ and } \quad \lim_{k \to \infty }\limsup_{n\to\infty} |r_{n,k}| = 0.
\] 
Then $(c_k)_{k\in\N}$ and $(a_n)_{n\in \N}$ are converging, and the respective limits are equal. 
\end{lemma}

\begin{proof}
Let $R_k= \limsup_{n\to\infty} r_{n,k}$. Along a subsequence, we have $\lim_{\ell \to \infty} r_{n_\ell,k} =R_k$. Hence, $a= \lim_{\ell \to \infty }a_{n_\ell} =c_k + R_k$ exists, and therefore $\lim_{k\to\infty} c_k = a$. Finally,
$\limsup_{n\to\infty} a_n \le c_k + R_k \to a$ and $\liminf_{n\to\infty} a_n \ge c_k - R_k \to a$ as $k\to\infty$.
\end{proof}

%%%%%%%%%%%%%%%%%
%%%%%%%%%%%%%%%%%

\section*{Acknowledgements}
This work has been supported by the Collaborative Research Center ``Statistical modeling of nonlinear dynamic processes'' (SFB 823) of the German Research Foundation, which is gratefully acknowledged.  Computational infrastructure and support were provided by the Centre for Information and Media Technology at Heinrich Heine University Düsseldorf.
The authors are grateful to Chen Zhou for helpful discussions and suggestions.

\bibliographystyle{chicago}
\bibliography{biblio}

\newpage

\thispagestyle{empty}

\begin{center}
{\bfseries SUPPLEMENTARY MATERIAL ON  \\ [0mm] ``STATISTICS FOR HETEROSCEDASTIC TIME SERIES EXTREMES''}
\vspace{.5cm}

{\small AXEL BÜCHER AND TOBIAS JENNESSEN}
\blfootnote{\today}

\end{center}

\begin{abstract}
This supplementary material contains the remaining proofs for the main paper. Proofs for Sections~\ref{sec:scedasis}-\ref{sec:EI} are presented in Sections~\ref{sec:psecdasis}-\ref{sec:pEI}, respectively.  Some auxiliary results are collected in Section~\ref{sec:auxiliary2}.
\end{abstract}

\appendix

%%%%%%%%%%%%%%%%%
%%%%%%%%%%%%%%%%%
%\input{text-supplement}

\section{Proofs for Section \ref{sec:scedasis}} \label{sec:psecdasis}

\begin{proof}[Proof of Theorem \ref{thm:scedasis}]
Fix $s \in [0,1]$. By definition, $K_b(\cdot, 0)$ and $K_b(\cdot, 1)$ do not depend on $n$, and the same is true for $K_b(\cdot, s)$ with $s\in(0,1)$ and sufficiently large $n$; we then have $K_b(\cdot, s)=K$.
Let 
\[ 
\Psi_n(x) = k^{-1} \sum_{i=1}^{n} \I \Big( X_i^{(n)} > V\Big(\frac{n}{kx}\Big) \Big)
\]
such that $\mathbb{F}_n(1,x) = \sqrt{k}\{\Psi_n(x) -x\}$. By Proposition \ref{prop:step}, $\{\mathbb{F}_n(1,x)\}_{x \in [0,1]} \wto \{ \mathbb{S}(1,x) \}_{x \in [0,1]}$ in $(\ell^\infty([0,1]), \| \cdot \|_\infty)$. Note that $\Psi_n^{-1}(x)= nk^{-1} \big\{1-F(X_{n,n-\lfloor kx \rfloor})\big\}$, such that
\[ 
\Big\{ \sqrt{k} \Big( nk^{-1} (1-F(X_{n,n-\lfloor kx \rfloor})) -x \Big) + \mathbb{F}_n(1,x) \Big\}_{x \in [0,1]}  = \op
\] 
by the functional delta-method applied to the inverse map (see Theorem 3.9.4 in \citealp{VanWel96}). 
In particular, for $y_n = nk^{-1} \{1-F(X_{n,n-k})\}$ we obtain 
\begin{align} \label{eq:ynconv}
\sqrt{k}(y_n-1) = -\mathbb{F}_n(1,1) + \op \wto -\mathbb{S}(1,1), 
\end{align}
yielding 
\begin{align}
\Prob\big( h^{1/4} k^{1/2} |y_n-1| \leq 1 \big) \to 1. \label{as_event}
\end{align} 
Let $K_b^+(\cdot,s)$ and $K_b^-(\cdot,s)$ denote the positive and negative part of $K_b(\cdot,s)$, respectively, and define, for $y>0$,
\[
c_n^{\pm}(y,s) = \frac{1}{kh} \sum_{i=1}^{n} \I \Big( X_i^{(n)} > V\Big(\frac{n}{ky}\Big) \Big) K_b^{\pm} \Big(\frac{s-i/n}{h},s\Big), 
\]
such that $\tilde{c}_n(s) = c_n^+(y_n,s) - c_n^-(y_n,s)$. Note that $c_n^{\pm}(\cdot,s)$ is monotonically increasing; therefore on the event $\{h^{1/4} k^{1/2} |y_n-1| \leq 1\}$ in (\ref{as_event}) we have 
	\[ 
	c_n^{+}(y^-,s) - c_n^-(y^+,s) \leq c_n^+(y_n,s) - c_n^-(y_n,s) \leq c_n^+(y^+,s) - c_n^-(y^-,s).
	\] 
where	 $y^{\pm} = 1 \pm (k^{1/2}h^{1/4})^{-1}$.
	 
The proof of the theorem is finished once we have shown 
\begin{align}
& \sqrt{kh} \big\{ c_n^+(y^+,s) - c_n^-(y^-,s) - c(s) \big\} \wto \Nor(\mu_s,\sigma_s^2), \label{eq:gsp} \\ 
& \sqrt{kh} \big\{ c_n^+(y^-,s) - c_n^-(y^+,s) - c(s) \big\} \wto \Nor(\mu_s,\sigma_s^2). \label{eq:gsm}
\end{align}
We restrict ourselves to proving \eqref{eq:gsp}, the assertion in \eqref{eq:gsm} can be treated analogously.
Set 
\[ d_n^{\pm}(y,s) = \frac{1}{kh} \sum_{i=1}^{n} \I \Big( U_i^{(n)} > 1- c(i/n) \frac{ky}{n} \Big) K_b^{\pm} \Big(\frac{s-i/n}{h},s\Big) 
\]
and let us first show that 
\begin{align}
\sqrt{kh} \big\{ c_n^+(y^+,s) - c_n^-(y^-,s) - d_n^+(y^+,s) + d_n^-(y^-,s) \big\} = \op. \label{star1}
\end{align}
which  is a consequence of
\begin{align}
\sqrt{kh} \big\{ c_n^+(y^+,s) - d_n^+(y^+,s) \big\} = \op, \quad \sqrt{kh} \big\{ c_n^-(y^-,s) - d_n^-(y^-,s) \big\} = \op. \label{star1n}
\end{align}
We only prove the first assertion in \eqref{star1n}, the second one follows by similar arguments.
By the same arguments that lead to \eqref{eq:uisubset}, defining $\varepsilon_n = F^{-1}(1-ky^+/n)$, we have
	\begin{align*}
	\Big\{ U_i^{(n)} > 1- c(i/n) (1-w_n) \frac{ky^+}{n} \Big\}
	& \subseteq 
	\Big\{ U_i^{(n)} > 1- \frac{1-F_{n,i}(\varepsilon_n)}{1-F(\varepsilon_n)} \frac{ky^+}{n} \Big\} \\
	& \subseteq \Big\{ U_i^{(n)} > 1- c(i/n) (1+w_n) \frac{ky^+}{n} \Big\},
	\end{align*}
	where $w_n = \frac{\tau}{c_{\min}} A\big(\frac{1}{1-F(\varepsilon_n)}\big) = \frac{\tau}{c_{\min}} A\big(\frac{n}{ky^+}\big)$. Consequently, rewriting
	\[
	c_n^{+}(y,s) = 
	\frac{1}{kh} \sum_{i=1}^{n} \I \Big( U_i^{(n)} > F(\eps_n) \Big)
	=
	\frac{1}{kh} \sum_{i=1}^{n} \I \Big\{ U_i^{(n)} > 1- \frac{1-F_{n,i}(\varepsilon_n)}{1-F(\varepsilon_n)} \frac{ky^+}{n} \Big\},
	\]
	(which is true a.s.), we have
	\begin{align}
	d_{n,-}^+(y^+,s) \leq c_n^+(y^+,s) \leq d_{n,+}^+(y^+,s), \label{star2}
	\end{align}
	where 
	\[ 
	d_{n,\pm}^+(x,s) = \frac{1}{kh} \sum_{i=1}^{n} K_b^+\Big(\frac{s-i/n}{h},s\Big) \I \Big(U_i^{(n)} > 1- c(i/n) (1\pm w_n) \frac{kx}{n} \Big). 
	\]
	As a consequence of \eqref{star2}, the proof of the first assertion in \eqref{star1n} is finished once we show that 
	\begin{align}
	\sqrt{kh} \big\{ d_{n,\pm}^+(y^+,s) - d_n^+(y^+,s) \big\} = \op. \label{star2b}
	\end{align}
For that purpose, note that
	\begin{align*}
	&\phantom{{}={}} \Exp\big[ \sqrt{kh} |d_{n,\pm}^+(y^+,s) - d_n^+(y^+,s) | \big] \\
	&\leq  \frac{1}{\sqrt{kh}} \sum_{i=1}^{n} \! K_b^{+} \Big( \frac{s-i/n}{h} ,s \Big) \! \Exp \bigg[ \Big| \I \Big( U_i^{(n)} > 1- c(i/n) (1 \pm w_n) \frac{ky^+}{n} \Big) \\
	& \hspace{8cm} - \I \Big( U_i^{(n)} > 1- c(i/n) \frac{ky^+}{n} \Big) \Big| \bigg] \\
	&\leq  2w_n y^+ \sqrt{kh} \frac{1}{nh} \sum_{i=1}^{n} K_b^+ \Big(\frac{s-i/n}{h},s \Big) c(i/n) \\
	&=  \frac{2 \tau}{c_{\min}} y^+ \sqrt{kh} A \Big(\frac{n}{ky^+}\Big) \frac{1}{nh} \sum_{i=1}^{n} K_b^+ \Big(\frac{s-i/n}{h}, s \Big) c(i/n) \\
	&=  \frac{2 \tau}{c_{\min}} y^+ \sqrt{kh} A \Big(\frac{n}{ky^+}\Big) \Big\{ c(s) \eta_1(s) -hc'(s) \eta_2(s) + \frac{h^2}{2} c''(s) \eta_3(s) + o(h^2) + O\Big(\frac{1}{nh}\Big) \Big\}
	\end{align*}
	by Lemma \ref{lem:kernel}, where $\eta_i$ is defined as in this lemma with $K$ replaced by $K_b^+(\cdot,s)$. The term in the last line of the above display converges to zero since  $\sqrt{kh} A ( n/(ky^+) ) \leq \sqrt{k} A\big(n/(2k)\big) \to 0$ by assumption.  This proves \eqref{star2b} and hence \eqref{star1} as argued above.  
	
	In the next step, we enforce a block structure, later allowing us to apply mixing conditions and show asymptotic independence of blocks. Let $r$ from Condition \ref{cond:mixing} denote the length of a block, and for simplicity we assume $m = n/r \in \N$ (otherwise, a potential remainder block of less than $r$ observations can be shown to be asymptotically negligible). Set 
	\[ e_n^{\pm}(y,s) = \frac{1}{kh} \sum_{j=1}^{m} K_b^{\pm} \Big( \frac{s-j/m}{h},s \Big) \sum_{t \in I_j} \I \Big( U_t^{(n)} > 1 - c(\tfrac{t}{n}) \frac{ky}{n} \Big). \]
	Subsequently, we show 
	\begin{align}
	\sqrt{kh} \big\{ d_n^+(y^+,s) - d_n^-(y^-,s) - e_n^+(y^+,s) + e_n^-(y^-,s) \big\} = \op. \label{star4}
	\end{align}
	Write 
	\begin{align*}
	&\phantom{{}={}} \Exp \big[ \sqrt{kh}  \big| d_n^+(y^+,s) - e_n^+(y^+,s) \big| \big] \\
	&\leq  \frac{1}{\sqrt{kh}} \sum_{j=1}^{m} \sum_{t \in I_j} \Prob \Big( U_t^{(n)}  > 1- c(\tfrac{t}{n}) \frac{ky^+}{n} \Big) \Big| K_b^+ \Big(\frac{s-t/n}{h}, s \Big) - K_b^+ \Big(\frac{s-j/m}{h}, s \Big) \Big| \\
	&=  \sqrt{\frac{k}{h}} \frac{y^+}{n} \sum_{j=1}^{m} \sum_{l=0}^{r-1} c \Big(\frac{jr-l}{n}\Big) \Big| K_b^+ \Big(\frac{s-\frac{jr-l}{n}}{h}, s \Big) - K_b^+ \Big(\frac{s-\frac{jr}{n}}{h}, s \Big) \Big|.
	\end{align*}
	Since $K_b^{\pm}(\cdot,s)$ does not depend on $n$ for sufficiently large $n$ and is Lipschitz-continuous, say with constant $L'$, the above can be bounded by 
	\begin{align*}
	L' y^+ \sqrt{\frac{k}{h}} \frac{r}{n^2 h} \sum_{j=1}^{m} \sum_{l=0}^{r-1} c \Big( \frac{jr-l}{n} \Big) 
	& = L'y^+ \frac{k^{1/2}r}{n^2 h^{3/2}} \sum_{j=1}^{n} c(\tfrac{j}{n}) 
	\end{align*}
	which converges to zero by Condition \ref{cond:bandwidth}. Analogously, $\Exp \big[ \sqrt{kh} \big| d_n^-(y^-,s) - e_n^-(y^-,s) \big| \big] = o(1)$, implying that (\ref{star4}) holds. Together with \eqref{star1}, we have shown that
	\[ 
	\sqrt{kh} \big\{ c_n^+(y^+,s) - c_n^-(y^-,s ) -c(s) \big\} = \sqrt{kh} \big\{ e_n^+(y^+,s) - e_n^-(y^-,s ) -c(s) \big\} + \op,
	\]
	whence the assertion in \eqref{eq:gsp} is shown once we prove that
\begin{align}
H_n =  \sqrt{kh} \big\{ e_n^+(y^+,s) - e_n^-(y^-,s) - c(s) \big\} \wto \Nor(\mu_s,\sigma_s^2). \label{eq:gsp2} 
\end{align}	
The assertion in \eqref{eq:gsp2} in turn is a consequence of
\begin{align}
\lim_{n\to\infty}\Exp[H_n] = \mu_s, \qquad H_n - \Exp[H_n] \wto \Nor(0,\sigma_s^2). \label{eq:gsp3} 
\end{align}
We start by proving the assertion regarding $\Exp[H_n]$ in \eqref{eq:gsp3}. For that purpose, write
\begin{align*}
	&\phantom{{}={}} \Exp \Big[ e_n^+(y^+,s) - e_n^-(y^-,s) \Big] \\
	&=  \frac{1}{kh} \sum_{j=1}^{m} K_b^+ \Big( \frac{s-j/m}{h},s \Big) \sum_{t \in I_j} \Prob\Big( U_t^{(n)} > 1 - c(\tfrac{t}{n}) \frac{ky^+}{n} \Big) \\
	& \hspace{5cm} - K_b^- \Big( \frac{s-j/m}{h},s \Big) \sum_{t \in I_j} \Prob\Big( U_t^{(n)} > 1 - c(\tfrac{t}{n}) \frac{ky^-}{n} \Big) \\
	&= \frac{1}{nh} \sum_{j=1}^{m} \sum_{t \in I_j} c(\tfrac{t}{n}) \Big\{ K_b^+ \Big(\frac{s-j/m}{h},s \Big) y^+ - K_b^- \Big(\frac{s-j/m}{h}, s\Big) y^- \Big\} \\
	&= \frac{1}{mh} \sum_{j=1}^{m} c(\tfrac{j}{m}) \Big\{ K_b^+ \Big(\frac{s-j/m}{h},s \Big) y^+ - K_b^- \Big(\frac{s-j/m}{h},s \Big) y^- \Big\} + o\big( (kh)^{-1/2} \big),
	\end{align*}
where the $o\big( (kh)^{-1/2} \big)$-term is due to $c$ being Lipschitz-continuous and $kr^2= o(n^2h)$, which holds by Condition \ref{cond:bandwidth} and $r=o(n/k)$ from \ref{cond:mixing}.

Hence, the above calculation and Lemma \ref{lem:kernel} imply that
\begin{align*}
 \frac{1}{\sqrt{kh}} \Exp[H_n] = \ & y^+ \Big( c(s)\eta_1^+(s)-hc'(s)\eta_2^+(s) + \frac{h^2}{2} c''(s) \eta_3^+(s) + o(h^2)+O(\tfrac{1}{mh}) \Big) \\
 - & y^- \Big( c(s)\eta_1^-(s)-hc'(s)\eta_2^-(s) + \frac{h^2}{2} c''(s) \eta_3^-(s) + o(h^2)+O(\tfrac{1}{mh}) \Big) \\
 - & c(s)+ o\big( (kh)^{-1/2} \big),
\end{align*}
where $\eta_i^+$ and $\eta_i^-$ are defined as $\eta_i$ in Lemma \ref{lem:kernel} but with $K$ replaced by $K_b^+(\cdot,s)$ and $K_b^-(\cdot,s)$, respectively (note that the latter two functions do not depend on $s$ or $n$ as argued at the beginning of this proof). 
Next, note that $|y^\pm -1 | = (k^{1/2}h^{1/4})^{-1} = o \big( {1}/{\sqrt{kh}} \big)$ and $o(h^2)+O\big( {1}/{(mh)} \big) = o\big({1}/{\sqrt{kh}}\big)$ due to $k^{1/5}h \to \lambda$ and $kr^2= o(n^2h)$, which follows from Conditions \ref{cond:bandwidth} and \ref{cond:mixing}.
As a consequence, 
\begin{align}
 \Exp[H_n] &= \sqrt{kh} \Big( c(s) \big( \eta_1^+(s)-\eta_1^-(s)-1 \big) -hc'(s)\big(\eta_2^+(s)-\eta_2^-(s) \big) \nonumber \\
 & \hspace{3cm} +\frac{h^2}{2} c''(s) \big( \eta_3^+(s)-\eta_3^-(s) \big) \Big) +o(1). \label{exp_Hn}
\end{align}
Note that $K_b(\cdot,s) = K_b^+(\cdot,s)-K_b^-(\cdot,s)$. First, let $s \in (0,1)$. For $n$ large enough such that $h<s<1-h$, we have $K_b(x,s)=K(x)$, $x \in [-1,1]$, and
\begin{align*}
 \eta_1^+(s)-\eta_1^-(s) &= \int_{-1}^1 K(x) \ \mathrm{d}x=1, \quad \eta_2^+(s)-\eta_2^-(s) = \int_{-1}^1 K(x)x \ \mathrm{d}x =0, \\
 & \quad \eta_3^+(s)-\eta_3^-(s) = \int_{-1}^1 K(x)x^2  \ \mathrm{d}x = a(s).
\end{align*}
Second, for $s=1$, the construction of the boundary kernel implies \citep{Jones93}
\begin{align*}
 \eta_1^+(s)-\eta_1^-(s) &= \int_{0}^1 K_b(x,1) \ \mathrm{d}x=1, \quad 
 \eta_2^+(s)-\eta_2^-(s) = \int_{0}^1 K_b(x,1)x \ \mathrm{d}x =0, \\
\eta_3^+(s)-\eta_3^-(s) &= \int_{0}^1 K_b(x,1)x^2  \ \mathrm{d}x = a(1).
\end{align*}
And for $s=0$, we have
\begin{align*}
 \eta_1^+(s)-\eta_1^-(s) &= \int_{-1}^0 K_b(x,0) \ \mathrm{d}x=1, \quad 
 \eta_2^+(s)-\eta_2^-(s) = \int_{-1}^0 K_b(x,0)x \ \mathrm{d}x =0, \\
 \eta_3^+(s)-\eta_3^-(s) &= \int_{-1}^0 K_b(x,0)x^2  \ \mathrm{d}x = a(0).
\end{align*}
Altogether, these equalities and equation (\ref{exp_Hn}) yield $\lim_{n\to\infty} \Exp[H_n] =  \tfrac{\lambda^{5/2}}{2} c''(s) a(s) = \mu_s$ for any $s \in [0,1]$, as asserted in \eqref{eq:gsp3}, where we again used $k^{1/5}h \to \lambda$ from \ref{cond:bandwidth}.

Next, consider the assertion on the right-hand side of \eqref{eq:gsp3}. For that purpose, recall $c_\infty=c_\infty(2) = 1+2\|c\|_\infty$
and $X'_{n,i}$ from \eqref{eq:xnprime} with $L=2$. We may then rewrite $e_n^\pm$ as
\[ 
e_n^\pm(y^\pm, s) = \frac{1}{kh} \sum_{j=1}^{m} K_b^\pm \Big(\frac{s-j/m}{h},s\Big) \sum_{t \in I_j} \I \big( X'_{n,t} > c_{\infty}- c(\tfrac{t}{n}) y^\pm \big). 
\]
We are going to apply a big-block-small-block technique. For that purpose, let
\[ 
I_j^B = \{ (j-1)r+1,\ldots,jr-\ell_n \}, \quad I_j^S = \{ jr-\ell_n+1,\ldots,jr \}, 
\]
	where the sequence $(\ell_n)_n$ is from Condition \ref{cond:mixing}. 
	Set 
	\begin{align*}
	e_{n,B}^\pm(y,s) = \frac{1}{kh} \! \sum_{j=1}^{m} \! K_b^\pm \Big( \frac{s-j/m}{h},s \Big) \! \sum_{t \in I_j^B} \! \I \big( X'_{n,t} > c_\infty -c(\tfrac{t}{n})y \big) \! - \! \Prob\big( X'_{n,t} > c_\infty -c(\tfrac{t}{n})y \big), \\
	e_{n,S}^\pm(y,s) = \frac{1}{kh} \! \sum_{j=1}^{m} \! K_b^\pm \Big( \frac{s-j/m}{h},s \Big) \! \sum_{t \in I_j^S} \! \I \big( X'_{n,t} > c_\infty -c(\tfrac{t}{n})y \big) \! - \! \Prob\big( X'_{n,t} > c_\infty -c(\tfrac{t}{n})y \big).
	\end{align*}
As a consequence, we may write
	\begin{align*} 
	H_n - \Exp[H_n] 
	=  \sqrt{kh} \big\{ e_{n,B}^+(y^+,s) - e_{n,B}^-(y^-,s ) \big\} + \sqrt{kh} \big\{ e_{n,S}^+(y^+,s) - e_{n,S}^-(y^-,s ) \big\},
	\end{align*}
whence the assertion on the right-hand side of \eqref{eq:gsp3} follows if we prove that
	\begin{align}
	H_{n1} & := \sqrt{kh}\big\{ e_{n,S}^+(y^+,s)-e_{n,S}^-(y^-,s)\big\} = \op, \label{eq:csb} \\
	H_{n2} & := \sqrt{kh} \big\{ e_{n,B}^+(y^+,s) - e_{n,B}^-(y^-,s ) \big\}   \wto \Nor(0,\sigma_s^2) \label{eq:csb2}
	\end{align}
	
	We start by proving \eqref{eq:csb}, for which it suffices to show that $\Var\big(\sqrt{kh}\big\{e_{n,S}^+(y^+,s)-e_{n,S}^-(y^-,s)\big\}\big) = o(1)$. For $n \in \N$ and $j\in\{1,\ldots,m\}$, let $V_{n,j} = (X'_{n,t})_{t \in I_j^S}$, and note that $e_{n,S}^{\pm}$ is a function of  $(V_{n,j})_{j=1,\ldots,m}$. Further, let $(V_{n,j}^{\ast})_{j=1,\ldots,m}$ denote an i.i.d.\ sequence, where $V_{n,j}^{\ast}$ is equal in distribution to $V_{n,j}$. Finally, let $e_{n,S}^{\pm, \ast}$ be defined as $e_{n,S}^{\pm}$, but in terms of $(V_{n,j}^{\ast})_{j=1,\ldots,m}$ instead of $(V_{n,j})_{j=1,\ldots,m}$ . First, we show the assertion in \eqref{eq:csb} with $e_{n,S}^{\pm}$ replaced by $e_{n,S}^{\pm,\ast}$.
By independence of blocks, we may write   
\begin{align}
&\phantom{{}={}} 
\Var\big(\sqrt{kh}\big\{ e_{n,S}^{+,\ast}(y^+,s)-e_{n,S}^{-,\ast}(y^-,s)\big\} \big) \nonumber \\
&\leq  
\frac{1}{kh} \sum_{j=1}^{m} \Exp \bigg[ \bigg\{ K_b^+\Big(\frac{s-j/m}{h},s\Big) \sum_{t \in I_j^S} \I \big(X'_{n,t} > c_\infty - c(\tfrac{t}{n})y^+\big) \nonumber \\
& \hspace{3cm} 
- K_b^-\Big(\frac{s-j/m}{h},s\Big) \sum_{t \in I_j^S} \I \big(X'_{n,t} > c_\infty - c(\tfrac{t}{n})y^-\big) \bigg\}^2 \bigg] \nonumber \\
&\leq  
\frac{1}{kh} \sum_{j=1}^{m} \bigg\{ K_b^+\Big(\frac{s-j/m}{h},s\Big)^2 \Exp \bigg[ \bigg\{ \sum_{t \in I_j^S} \I \big(X'_{n,t} > c_\infty - c(\tfrac{t}{n})y^+\big) \bigg\}^2 \bigg] \nonumber \\
& \hspace{3cm} 
+ K_b^-\Big(\frac{s-j/m}{h},s\Big)^2 \Exp \bigg[ \bigg\{ \sum_{t \in I_j^S} \I \big(X'_{n,t} > c_\infty - c(\tfrac{t}{n})y^-\big) \bigg\}^2 \bigg] \bigg\} \nonumber \\
&\leq  
\frac{1}{kh} \Exp \bigg[ \bigg\{ \sum_{t \in I_1^S} \I(X'_{n,t} \neq 0) \bigg\}^2 \bigg] \sum_{j=1}^{m} K_b^+\Big(\frac{s-j/m}{h},s\Big)^2 + K_b^-\Big(\frac{s-j/m}{h},s\Big)^2, \label{star6}
\end{align}
where the last step is due to stationarity. As in (\ref{exp_neq_zero}), we have 
\[ 
\Exp \bigg[ \bigg\{ \sum_{t \in I_1^S} \I(X'_{n,t} \neq 0) \bigg\}^2 \bigg] \leq \frac{\ell_n}{r} \Exp \bigg[ \bigg\{ \sum_{t = 1}^{r} \I(X'_{n,t} \neq 0) \bigg\}^2 \bigg] = O\Big(\frac{\ell_nk}{n}\Big),
\]
where the last bound follows from Condition \ref{cond:expnum}.
As a consequence, the expression in (\ref{star6}) can be bounded by 
	\begin{align} 
	&O\Big(\frac{\ell_n m}{n} \Big) \frac{1}{mh} \sum_{j=1}^{m} K_b^+\Big(\frac{s-j/m}{h},s\Big)^2 + K_b^-\Big(\frac{s-j/m}{h},s\Big)^2  \nonumber \\
	 & \hspace{3cm} = O \Big(\frac{\ell_n}{r}\Big) \bigg\{ \int_{\frac{s-1}{h}}^{s/h} K_b^+(x,s)^2 + K_b^-(x,s)^2 \ \mathrm{d}x + O\Big(\frac{1}{mh}\Big) \bigg\}, \label{sum_K+K-} 
	\end{align}
	which converges to zero due to $\ell_n=o(r)$ and $mh \to \infty$, since $kh \to \infty$ and $m \gg k$ by $r=o(n/k)$ in Condition \ref{cond:bandwidth} and \ref{cond:mixing}, respectively.
	Hence, $\sqrt{kh}\big\{e_{n,S}^{+,\ast}(y^+,s) - e_{n,S}^{-,\ast}(y^-,s)\big\}=\op$. The same argumentation that was used in the proof of Proposition~\ref{prop:sstep}  can be used to deduce \eqref{eq:csb}.
	
It remains to prove \eqref{eq:csb2}. For that purpose, write
\[ 
H_{n2} = \sqrt{kh}\big\{ e_{n,B}^{+}(y^+,s) - e_{n,B}^{-}(y^-,s) \big \} = \frac{1}{\sqrt{kh}} \sum_{j=1}^{m} f_{j,n}(s), 
\]
where, for $j\in\{1,\ldots,m\}$, 
\begin{align*}
f_{j,n}(s)  =  K_b^+ \Big(\frac{s-j/m}{h},s\Big) \sum_{t \in I_j^B} &\Big\{ \I \big( X'_{n,t} > c_\infty - c(\tfrac{t}{n}) y^+ \big) - \Prob\big(X'_{n,t} > c_\infty -c(\tfrac{t}{n}) y^+\big) \Big\} \\
 -  K_b^- \Big(\frac{s-j/m}{h},s\Big) \sum_{t \in I_j^B} &\Big\{ \I \big( X'_{n,t} > c_\infty - c(\tfrac{t}{n}) y^- \big) - \Prob\big(X'_{n,t} > c_\infty -c(\tfrac{t}{n}) y^-\big) \Big\}. 
\end{align*}
Note that $f_{j,n}(s)$ is centered and depends on the block $I_j^B$ only, such that the observations making up $f_{j,n}(s)$ and $f_{i,n}(s)$ are separated by at least $\ell_n$ observations for $j \neq i$. By the same arguments given in the proof of Proposition \ref{prop:sstep} we can assume that $f_{1,n}(s), \ldots, f_{m,n}(s)$ are independent. As a consequence, we may apply the classical Lindeberg Central Limit Theorem. The Lindeberg condition is satisfied, if for any $\varepsilon > 0$, 
	\[ \lim_{n \to \infty} \frac{1}{kh} \sum_{j=1}^{m} \Exp \Big[ f_{j,n}(s)^2 \I \big( |f_{j,n}(s)| > \varepsilon \sqrt{kh} \big) \Big] = 0. \]
	Since $|f_{j,n}(s)| \lesssim r-\ell_n \leq r$, the Lindeberg condition already follows from the assumption $r=o(\sqrt{kh})$ in \ref{cond:bandwidth}, see Corollary 3.6 in \cite{DreRoo10} for a similar argumentation.
	
It remains to prove that $\lim_{n\to\infty}\Var(H_{n2})=\sigma_s^2$. Let
\[ 
d_{n,j}(y) = \sum_{t \in I_j^B} \I \big( X'_{n,t} > c_\infty -c(\tfrac{t}{n}) y \big), 
\]
such that
\begin{align*}
\Var(H_{n2}) &=  
\frac{1}{kh} \sum_{j=1}^{m} \Var \Big( K_b^+ \Big(\frac{s-j/m}{h},s\Big) d_{n,j}(y^+) - K_b^- \Big(\frac{s-j/m}{h},s\Big) d_{n,j}(y^-) \Big) \\
&=  \frac{1}{kh} \sum_{j=1}^{m} \Exp \bigg[ \bigg\{ K_b^+ \Big(\frac{s-j/m}{h},s\Big) d_{n,j}(y^+) - K_b^- \Big(\frac{s-j/m}{h},s\Big) d_{n,j}(y^-) \bigg\}^2 \bigg] \\
& \hspace{1cm} - \Exp \bigg[ K_b^+ \Big(\frac{s-j/m}{h},s\Big) d_{n,j}(y^+) - K_b^- \Big(\frac{s-j/m}{h},s\Big) d_{n,j}(y^-) \bigg]^2 \\
&=:  A_n - B_n. 
\end{align*}
By stationarity, 
\[ 
| \Exp[ d_{n,j}(y^\pm) ] | \leq \sum_{t \in I_j^B} \Prob(X'_{n,t} \neq 0) = (r-\ell_n)\frac{k}{n} c_{\infty}(L), 
\]
implying 
\begin{align*}
|B_n| & \leq \frac{1}{kh} \sum_{j=1}^{m} \bigg\{ K_b^+ \Big(\frac{s-j/m}{h},s\Big)^2 \Exp[d_{n,j}(y^+)]^2 + K_b^- \Big(\frac{s-j/m}{h},s\Big)^2 \Exp[d_{n,j}(y^-)]^2 \bigg\} \\
& \leq c_\infty(L)^2 \frac{k}{m} \frac{1}{mh} \sum_{j=1}^{m} K_b^+ \Big(\frac{s-j/m}{h},s\Big)^2 + K_b^- \Big(\frac{s-j/m}{h},s\Big)^2, 
\end{align*}
which converges to zero by the previous calculation in (\ref{sum_K+K-})  and $k/m=o(1)$ by Condition \ref{cond:mixing}. As a consequence, 
\begin{align} \label{eq:varhn2}
\Var(H_{n2})=A_n + o(1).
\end{align}

Next, write $A_n = m^{-1} \sum_{j=1}^{m} A_n(j)$, where
\[
A_n(j) = \frac{m}{kh} \Exp \bigg[ \bigg\{ K_b^+ \Big(\frac{s-j/m}{h},s\Big) d_{n,j}(y^+) - K_b^- \Big(\frac{s-j/m}{h},s\Big) d_{n,j}(y^-) \bigg\}^2 \bigg]. 
\]
Recall the definition of the length of the core of a cluster $y$, denoted by $L(y)$, see Definition 2.1 in \cite{DreRoo10}. For some constant $K>0$, writing $Y_{n,j} = (X'_{n,t})_{t \in I_j^B}$ for $j\in\{1,\ldots,m\}$, we may then decompose $A_{n}(j) = S_{n,K}(j)+R_{n,K}(j)$, where 
\begin{align*}
& \hspace{-.2cm}S_{n,K}(j) \\
& = \frac{m}{kh} \Exp \bigg[ \bigg\{ K_b^+ \Big(\frac{s-j/m}{h},s\Big) d_{n,j}(y^+) - K_b^- \Big(\frac{s-j/m}{h},s\Big) d_{n,j}(y^-) \bigg\}^2 \I\big(L(Y_{n,j})\leq K \big) \bigg], \\
& \hspace{-.2cm}R_{n,K}(j) \\
& = \frac{m}{kh} \Exp \bigg[ \bigg\{ K_b^+ \Big(\frac{s-j/m}{h},s\Big) d_{n,j}(y^+) - K_b^- \Big(\frac{s-j/m}{h},s\Big) d_{n,j}(y^-) \bigg\}^2 \I\big(L(Y_{n,j}) > K \big) \bigg].
\end{align*}
We have 
\begin{align*}
R_{n,K}(j) & \leq \frac{m}{kh} K_b^+ \Big(\frac{s-j/m}{h},s\Big)^2 \Exp \Big[ d^2_{n,j}(y^+) \I \big(L(Y_{n,j}) > K\big) \Big] \\
& \ + \frac{m}{kh} K_b^- \Big(\frac{s-j/m}{h},s\Big)^2 \Exp \Big[ d^2_{n,j}(y^-) \I\big(L(Y_{n,j}) > K \big) \Big].
\end{align*}
The two summands summands on the right-hand side can be written as
\begin{align*}
&\phantom{{}={}} 
\frac{m}{kh} K_b^\pm \Big( \frac{s-j/m}{h},s \Big)^2 \sum_{q,t \in I_j^B} \Prob \Big( X'_{n,t} > c_\infty - c(\tfrac{t}{n}) y^\pm, \\
& \hspace{7cm}
X'_{n,q} > c_\infty - c(q/n) y^\pm, L(Y_{n,j}) > K \Big) \\
&\leq  
\frac{m}{kh} K_b^\pm \Big( \frac{s-j/m}{h},s \Big)^2 \sum_{q,t \in I_j^B} \Prob \Big( X'_{n,t} > 0, X'_{n,q} > 0, L(Y_{n,j}) > K \Big) \\
&= 
\frac{m}{kh} K_b^\pm \Big( \frac{s-j/m}{h},s \Big)^2 \Exp \Big[ \Big\{ \sum_{t=1}^{r} \I(X'_{n,t}>0) \Big\}^2 \I \big(L(Y_{n,1}) > K \big) \Big] \\
&\leq  
\frac{1}{h} K_b^\pm \Big( \frac{s-j/m}{h},s \Big)^2 \bigg\{ \frac{m}{k} \Exp \Big[ \Big\{ \sum_{t=1}^{r} \I(X'_{n,t}>0) \Big\}^{2+\delta} \Big] \bigg\}^{\frac{2}{2+\delta}} \Big\{ \frac{m}{k} \Prob \big( L(Y_{n,1}) > K \big)  \Big\}^{\frac{\delta}{2+\delta}}
\end{align*}
by Hölder's inequality. Consequently, we obtain 
\begin{align*}
\frac{1}{m} \sum_{j=1}^{m} R_{n,K}(j) & \leq \bigg\{ \frac{m}{k} \Exp \Big[ \Big\{ \sum_{t=1}^{r} \I(X'_{n,t}>0) \Big\}^{2+\delta} \Big] \bigg\}^{\frac{2}{2+\delta}} \Big\{ \frac{m}{k} \Prob \big( L(Y_{n,1}) > K \big)  \Big\}^{\frac{\delta}{2+\delta}}\\
& \ \ \times \frac{1}{mh} \sum_{j=1}^{m} K_b^+ \Big( \frac{s-j/m}{h},s \Big)^2 + K_b^- \Big( \frac{s-j/m}{h},s \Big)^2.
\end{align*}
By Condition \ref{cond:expnum} and Lemma 5.2 (vii) in \cite{DreRoo10}, which is applicable by Condition \ref{cond:mixing} (note that their $v_n$ is $v_n = \Pr(X'_{n,i} \ne 0) = \tfrac{k}{n} c_\infty(2)$ in our notation), we have 
\begin{align} \label{eq:crnk1}
\lim_{K \to \infty} \limsup_{n \to \infty} \frac{1}{m} \sum_{j=1}^{m} R_{n,K}(j) = 0. 
\end{align}
Next, consider the term $S_{n,K}(j)$ with $j\in\{1,\ldots,m\}$, which may be written as
	\begin{align*}
	&\phantom{{}={}} S_{n,K}(j)  \\
	&= \frac{m}{kh} \bigg\{ 
	K_b^+ \Big( \frac{s-j/m}{h},s \Big)^2 \Exp \Big[ d^2_{n,j}(y^+) \I \big( L(Y_{n,j}) \leq K \big) \Big] \\
	& \hspace{1cm} + K_b^- \Big( \frac{s-j/m}{h},s \Big)^2 \Exp \Big[ d^2_{n,j}(y^-) d_{n,j}(y^-) \I \big( L(Y_{n,j}) \leq K \big) \Big] \\
	& \hspace{1cm} -2 K_b^+ \Big( \frac{s-j/m}{h},s \Big) K_b^- \Big( \frac{s-j/m}{h},s \Big) \Exp \Big[ d_{n,j}(y^+) d_{n,j}(y^-) \I \big( L(Y_{n,j}) \leq K \big) \Big] \bigg\} \\
	& = \frac{m}{kh} \bigg\{ 
	K_b^+ \Big( \frac{s-j/m}{h},s \Big)^2 \! \! \sum_{q,t \in I_j^B, |q-t| \leq K} \! \!  \Prob \Big( X'_{n,t} > c_\infty -c(\tfrac{t}{n}) y^+, \\
	& \hspace{7.5cm} X'_{n,q} > c_\infty -c(\tfrac{q}{n}) y^+, L(Y_{n,j}) \leq K \Big) \\
	& \hspace{1cm} + K_b^- \Big( \frac{s-j/m}{h},s \Big)^2 \! \!  \sum_{q,t \in I_j^B, |q-t| \leq K} \! \!  \Prob \Big( X'_{n,t} > c_\infty -c(\tfrac{t}{n}) y^-, \\
	& \hspace{7.5cm} X'_{n,q} > c_\infty -c(\tfrac{q}{n}) y^-, L(Y_{n,j}) \leq K \Big) \\
	& \hspace{1cm} -2 K_b^+ \Big( \frac{s-j/m}{h},s \Big) K_b^- \Big( \frac{s-j/m}{h},s \Big) \! \!  \sum_{q,t \in I_j^B, |q-t| \leq K} \! \! \Prob \Big( X'_{n,t} > c_\infty -c(\tfrac{t}{n}) y^+, \\
	& \hspace{7cm} X'_{n,q} > c_\infty -c(\tfrac{q}{n}) y^-, L(Y_{n,j}) \leq K \Big) \bigg\}.
	\end{align*}
Let $S_{n,K}'(j)$ be defined exactly as the right-hand side of the previous display, but with the probability terms replaced by 
\begin{align*} 
& \Prob \Big( X'_{n,t} > c_\infty -c(\tfrac{t}{n}) y, X'_{n,q} > c_\infty -c(\tfrac{q}{n}) y' \Big).
\end{align*}
for $y,y' \in \{y^+,y^-\}$ (i.e., we omit the additional condition $L(Y_{n,j})\le K$ everywhere). Further, let $R_{n,K}'(j) = S_{n,K}'(j) - S_{n,K}(j)$. By the same arguments that were used for $R_{n,K}$ above, one can show that 
\begin{align} \label{eq:crnk2}
\lim_{K \to \infty} \limsup_{n \to \infty} \frac{1}{m} \sum_{j=1}^{m} R'_{n,K}(j) = 0. 
\end{align}
Regarding the remaining terms $S_{n,K}'(j)$ we obtain,  by uniform continuity of $c$ and Lem\-ma~\ref{lem:probapp},
\begin{align*}
S_{n,K}'(j) 
&= S''_{n,K}(j)+ R_n''(j) o(1), 
\end{align*}
where the $o(1)$ is uniform in $j=1,\ldots,m$, where
\[  
 R_n''(j) = \frac{1}{h} K_b^+ \Big( \frac{s-\tfrac{j}{m}}{h},s\Big)^2 \! \! + \frac{1}{h} K_b^-\Big(\frac{s-\tfrac{j}{m}}{h},s\Big)^2 \! - \! \frac{2}{h} K_b^+\Big(\frac{s-\tfrac{j}{m}}{h},s\Big) K_b^-\Big(\frac{s-\tfrac{j}{m}}{h},s\Big)
\]
and where
\begin{align*}
S''_{n,K}(j) &= 
\frac{1}{h} K_b^+\Big(\frac{s-\tfrac{j}{m}}{h},s\Big)^2  \frac{1}{r} \sum_{t \in I_j^B} c(\tfrac{t}{n}) \Big\{ d_0(y^+,y^+) + 2 \sum_{q=1}^{K \wedge (jr-\ell_n-t)} d_q(y^+,y^+) \Big\} \\
&\quad + 
\frac{1}{h} K_b^-\Big(\frac{s-\tfrac{j}{m}}{h},s\Big)^2 \frac{1}{r} \sum_{t \in I_j^B} c(\tfrac{t}{n}) \Big\{ d_0(y^-,y^-) + 2 \sum_{q=1}^{K \wedge (jr-\ell_n-t)}  d_q(y^-,y^-) \Big\} \\
&\quad - 
\frac{2}{h} K_b^+\Big(\frac{s-\tfrac{j}{m}}{h},s\Big) K_b^-\Big(\frac{s-\tfrac{j}{m}}{h},s\Big) \\
& \hspace{1cm} 
\cdot \frac{1}{r} \sum_{t \in I_j^B} c(\tfrac{t}{n}) \Big\{ d_0(y^+,y^-) + \sum_{q=1}^{K \wedge (jr-\ell_n-t)} d_q(y^+,y^-) + d_q(y^-,y^+)  \Big\}.
\end{align*}
By Lemma \ref{lem:kernel} (and a straightforward extension of this lemma to the case of a product of kernels) we have $m^{-1} \sum_{j=1}^m R_n''(j)=O(1)$.
Moreover,
\begin{align*}
S''_{n,K}(j) 
& = \frac{1}{h} K_b^+\Big(\frac{s-\tfrac{j}{m}}{h},s\Big)^2 c(\tfrac{j}{m}) D_K(y^+,y^+) + \frac{1}{h} K_b^-\Big(\frac{s-\tfrac{j}{m}}{h},s\Big)^2 c(\tfrac{j}{m}) D_K(y^-,y^-) \\
& \ \ -  \frac{2}{h} K_b^+\Big(\frac{s-\tfrac{j}{m}}{h},s\Big) K_b^-\Big(\frac{s-\tfrac{j}{m}}{h},s\Big) c(\tfrac{j}{m}) D_K(y^+,y^-) + o(1),
\end{align*}
where the $o(1)$ is uniform in $j=1,\ldots,m$ (and in $y^+,y^- \in [0,2]$ as arbitrary inputs), and 
\[ 
D_K(x,x') = d_0(x,x') + \sum_{q=1}^{K} d_q(x,x') + d_q(x',x) .
\]
Therefore, 
\begin{align*}
&\phantom{{}={}} 
\frac{1}{m} \sum_{j=1}^{m} S_{n,K}'(j) \\ 
&=  \frac{1}{mh} \sum_{j=1}^{m} c(\tfrac{j}{m}) \Big\{ K_b^+\Big(\frac{s-\tfrac{j}{m}}{h},s\Big)^2 D_K(y^+,y^+) + K_b^-\Big(\frac{s-\tfrac{j}{m}}{h},s\Big)^2 D_K(y^-,y^-) \\
& \hspace{2.7cm} - 2 K_b^+\Big(\frac{s-\tfrac{j}{m}}{h},s\Big) K_b^-\Big(\frac{s-\tfrac{j}{m}}{h},s\Big) D_K(y^+,y^-) \Big\} + o(1),
\end{align*}
which converges to $c(s)\eta(s)D_K(1,1)$ by a straightforward extension of Lemma \ref{lem:kernel} to the case of a product of kernels. Further, note that $D_K(y,y')$ converges to $D_K(1,1)$ for $y,y'\in \{y^+,y^-\}$, since $d_q$ is continuous in $(1,1)$ by Theorem 2 and its subsequent discussion in \cite{Seg03}.  Finally, since 
\begin{align*}
A_n &= \frac1{m} \sum_{j=1}^m A_{n}(j) 
=
\frac{1}{m} \sum_{j=1}^{m} S_{n,K}'(j) - \frac{1}{m} \sum_{j=1}^{m} R_{n,K}'(j) + \frac{1}{m} \sum_{j=1}^{m} R_{n,K}(j) 
\end{align*} 
Lemma \ref{lem:seq} and \eqref{eq:crnk1} and \eqref{eq:crnk2} imply 
$
\lim_{n\to\infty}\Var(H_{n2})= \lim_{n \to \infty} A_n 
= 
\sigma_s^2,
$
where we used \eqref{eq:varhn2}.
\end{proof}

\begin{lemma} \label{lem:kernel}
	Assume $c \in C^2([0,1])$. Let $K$ be a Lipschitz-continuous function on $[-1,1]$ with $K(x) = 0$ for $|x| > 1$. Further, let $h=h_n>0$ satisfy $h\to 0$ and $nh\to\infty$ for $n\to\infty$. Then, for any $s \in [0,1]$, as $n \to \infty$, 
	\begin{align*}
	& \frac{1}{nh} \sum_{i=1}^{n} K \Big(\frac{s-i/n}{h}\Big) c(i/n) = c(s) \eta_1(s) - hc'(s)\eta_2(s) + \frac{h^2}{2} c''(s) \eta_3(s) + o(h^2)+ O\Big(\frac{1}{nh}\Big), \\
	& \frac{1}{nh} \sum_{i=1}^{n} K^2\Big(\frac{s-i/n}{h}\Big) c(i/n) = c(s) \eta_4(s) + O(h) + O\Big(\frac{1}{nh}\Big).
	\end{align*}
	where
	{\small\begin{align*}
	\eta_1(s) & = \I(s \! \leq \! h) \int_{-1}^{s/h} \! K(x) \mathrm{d}x + \I(h \! <\!s\!<\!1-h) \int_{-1}^{1} \! K(x) \mathrm{d}x + \I(s \! \geq \! 1-h) \int_{\frac{s-1}{h}}^{1} \! K(x) \mathrm{d}x, \\
	\eta_2(s) & = \I(s \! \leq \! h) \int_{-1}^{s/h} \! K(x)x \mathrm{d}x + \I(h \! <\!s\!<\!1-h) \int_{-1}^{1} \! K(x)x \mathrm{d}x + \I(s \! \geq \! 1-h) \int_{\frac{s-1}{h}}^{1} \! K(x)x \mathrm{d}x, \\
	\eta_3(s) & = \I(s \! \leq \! h) \int_{-1}^{s/h} \! K(x)x^2 \mathrm{d}x + \I(h \! <\!s\!<\!1-h) \int_{-1}^{1} \! K(x)x^2 \mathrm{d}x + \I(s \! \geq \! 1-h) \int_{\frac{s-1}{h}}^{1} \! K(x)x^2 \mathrm{d}x, \\
	\eta_4(s) & = \I(s \! \leq \! h) \int_{-1}^{s/h} \! K^2(x) \mathrm{d}x + \I(h \! <\!s\!<\!1-h) \int_{-1}^{1} \! K^2(x) \mathrm{d}x + \I(s \! \geq \! 1-h) \int_{\frac{s-1}{h}}^{1} \! K^2(x) \mathrm{d}x.
	\end{align*}}
\end{lemma}

\begin{proof}
A Riemann sum approximation implies
\begin{align*}
& \frac{1}{nh} \sum_{i=1}^{n} K \Big(\frac{s-i/n}{h}\Big) c(i/n) = \int_{\frac{s-1}{h}}^{s/h} K(x) c(s-hx) \ \mathrm{d}x + O \Big(\frac{1}{nh}\Big). 
\end{align*}
Next, by Taylor's theorem, there exists some $\tau_x \in [0,1]$ such that
\begin{multline*}
\int_{\frac{s-1}{h}}^{s/h} K(x) c(s-hx) \ \mathrm{d}x   \\
=	 
c(s) \int_{\frac{s-1}{h}}^{s/h} K(x) \ \mathrm{d}x - hc'(s) \int_{\frac{s-1}{h}}^{s/h} K(x) x \ \mathrm{d}x + \frac{h^2}{2} c''(s) \int_{\frac{s-1}{h}}^{s/h} K(x) x^2 \ \mathrm{d}x \\
 + \frac{h^2}{2} \int_{\frac{s-1}{h}}^{s/h} K(x) x^2 \big\{ c''(s-xh\tau_x) - c''(s) \big\} \ \mathrm{d}x.
\end{multline*}
Since $K$ has compact support and $c''$ is continuous, the dominated convergence theorem implies that the last integral is of the order 	$o(h^2)$. 
If $s<1-h$, we have $\frac{s-1}{h} < -1$, and for $s>h$, we have $s/h>1$.  This allows to rewrite the boundaries of the integral accordingly in view of the fact that $K$ has support $[-1,1]$.
	
For the second assertion, write
\begin{align*}
&\phantom{{}={}} \frac{1}{nh} \sum_{i=1}^{n} K^2\Big(\frac{s-i/n}{h}\Big) c(i/n) \\
&= \int_{\frac{s-1}{h}}^{s/h} K^2(x) c(s-hx) \ \mathrm{d}x + O \Big(\frac{1}{nh}\Big) \\
&= c(s) \int_{\frac{s-1}{h}}^{s/h} K^2(x) \ \mathrm{d}x + \int_{\frac{s-1}{h}}^{s/h} K^2(x) \big\{ c(s-hx) -c(s) \big\} \ \mathrm{d}x + O \Big(\frac{1}{nh}\Big) \\
&= c(s) \int_{\frac{s-1}{h}}^{s/h} K^2(x) \ \mathrm{d}x + O(h) + O\Big(\frac{1}{nh}\Big),
	\end{align*}
where the last step is again due to the dominated convergence theorem.
\end{proof}

\begin{proof}[Proof of Theorem \ref{thm:integr_scedasis}]
As at the beginning of the proof of Theorem \ref{thm:scedasis}, let  $y_n= nk^{-1}\big\{1-F(X_{n,n-k})\big\}$. The definition of the STEP $\mathbb{F}_n$ in \eqref{eq:defstep} allows to write 
\[ 
\sqrt{k}\{ \hat{C}_n(s)-C(s)\} = \mathbb{F}_n(s,y_n) + C(s) \sqrt{k}(y_n-1)
\] 
By the proof of Theorem \ref{thm:scedasis}, see \eqref{eq:ynconv}, we know that 
\[
\sqrt{k}(y_n-1) = -\mathbb{F}_n(1,1) + \op. 
\]
Suppose we have shown that 
\begin{align} \label{eq:fneq}
\sup_{s \in [0,1]} |\mathbb{F}_n(s,y_n) - \mathbb{F}_n(s,1)| = \op.
\end{align}
Then, by the previous three displays, uniformly in $s$,
\begin{align} \label{eq:cndec}
\sqrt{k}\{ \hat{C}_n(s)-C(s) \} =\mathbb{F}_n(s,1) - C(s)\mathbb{F}_n(1,1) + \op, 
\end{align}
which implies the assertion since $\{ \mathbb{F}_n(s,1) \}_{s \in [0,1]} \wto \{\mathbb{S}(s,1)\}_{s\in [0,1]} $ in $(\ell^{\infty}([0,1]), \|\cdot\|_{\infty})$ by Proposition \ref{prop:step}. It remains to prove \eqref{eq:fneq}.
	Note that 
	\begin{align*}
	\sup_{s\in [0,1]} |\mathbb{F}_n(s,y_n) - \mathbb{F}_n(s,1)| \leq \sup_{(s,y),(s,z) \in [0,1]^2: |y-z|<|y_n-1|} |\mathbb{F}_n(s,y) - \mathbb{F}_n(s,z)|,
	\end{align*}
For any $\varepsilon > 0$ and $\mu \in (0,1/2)$, we have $\Prob(k^{\mu}|y_n-1| < \varepsilon) \to 1$. Thus, on this event the above supremum can be bounded by 
\begin{align}
\sup_{(s,y),(s,z) \in [0,1]^2: |y-z|< \delta_n} |\mathbb{F}_n(s,y) - \mathbb{F}_n(s,z)|. \nonumber %\label{sup_Fn_diff}
\end{align}
where $\delta_n := \varepsilon k^{-\mu} \downarrow 0$. Analogously to showing (\ref{w_delta_Sn}) in the proof of Proposition~\ref{prop:step}, we obtain that the last expression is asymptotically negligible (note that the same semimetric used in the proof of Proposition~\ref{prop:step} can be applied here by Theorem 7.19 in \cite{Kos08} and Proposition \ref{prop:step} and the proof of tightness in the proof of Proposition~\ref{prop:sstep}, which made Theorem 11.16 in \cite{Kos08} applicable for $\mathbb S_n$).
\end{proof}

\section{Proofs for Section \ref{sec:testing}}

For $b \in \N$  and $(s,x)\in [0,1]^2$ let
\begin{align}
\label{eq:sbnxi}
 \mathbb{S}^{(b)}_{n,\xi}(s,x) &= \frac{1}{\sqrt{k}} \sum_{j=1}^{m} \xi^{(b)}_j \sum_{i \in I_j} \I \Big( U_i^{(n)} > 1-\frac{kx}{n} c\big(\tfrac{i}{n}\big) \Big) \I(\tfrac{i}{n} \leq s), \\
 \nonumber
 \mathbb{F}^{(b)}_{n,\xi}(s,x) &= \frac{1}{\sqrt{k}} \sum_{j=1}^{m} \xi^{(b)}_j \sum_{i \in I_j} \I \Big( X_i^{(n)} > V \Big(\frac{n}{kx} \Big) \Big) \I(\tfrac{i}{n} \leq s)
\end{align}
denote \textit{bootstrap-versions} of the (simple) STEP defined in \eqref{eq:defsstep} and \eqref{eq:defstep}.

\begin{proposition} \label{prop:sstep_boot}
 Suppose that Conditions \ref{cond:basic}-\ref{cond:2nd order} hold for $L=1$. Then, for any $B\in\N$ and as $n \to \infty$,
 \[ \big( \Sb_n, \Sb_{n,\xi}^{(1)}, \ldots, \Sb_{n,\xi}^{(B)} \big) \wto \big( \Sb, \Sb^{(1)}, \ldots, \Sb^{(B)} \big) \quad \textrm{in} \quad \big( \ell^{\infty}([0,1]^2),\|\cdot\|_{\infty} \big)^{B+1}, \]
 where $\Sb^{(1)},\ldots,\Sb^{(B)}$ are independent copies of $\Sb$ from Proposition \ref{prop:sstep}.
\end{proposition}

\begin{proposition} \label{prop:step_boot}
 Suppose that Conditions \ref{cond:basic}-\ref{cond:2nd order} hold  for $L=1$.  Then, for any $b\in\N$ and as $n \to \infty$,
 \[ 
  \sup_{(s,x) \in [0,1]^2} |\mathbb{F}^{(b)}_{n,\xi}(s,x) - \mathbb{S}^{(b)}_{n,\xi}(s,x)| = \op. 
  \]
As a consequence, by Proposition~\ref{prop:step} and \ref{prop:sstep_boot}, for any $B\in\N$ and as $n \to \infty$,
 \[ \big( \Fb_n, \Fb_{n,\xi}^{(1)}, \ldots, \Fb_{n,\xi}^{(B)} \big) \wto \big( \Sb, \Sb^{(1)}, \ldots, \Sb^{(B)} \big) \quad \textrm{in} \quad \big( \ell^{\infty}([0,1]^2),\|\cdot\|_{\infty} \big)^{B+1}, \]
 where $\Sb^{(1)},\ldots,\Sb^{(B)}$ are independent copies of $\Sb$ from Proposition \ref{prop:sstep}.
\end{proposition}

\begin{proof}[Proof of Theorem \ref{thm:boot_jointconv}]
Define
\begin{align*} 
\mathbb{\tilde C}^{(b)}_{n,\xi} (s) &= \mathbb{\tilde D}^{(b)}_{n,\xi}(s) - \hat C_n(s) \mathbb{\tilde D}^{(b)}_{n,\xi}(1),
\end{align*} 
where
\begin{align*} 
\mathbb{\tilde D}^{(b)}_{n,\xi}(s) &= \frac{1}{\sqrt{k}} \sum_{j=1}^{m} \xi_j^{(b)} \sum_{i \in I_j} \I \Big( X_i^{(n)} > X_{n,n-k} \Big) \I(\tfrac{i}{n} \leq s), \quad s \in [0,1].
\end{align*}
Recall that $y_n = nk^{-1}\big( 1-F(X_{n,n-k}) \big)$ converges to 1 in probability by \eqref{eq:ynconv}. For $b\in\N$ and $s\in[0,1]$, we have 
\begin{align*}
\mathbb{\tilde C}_{n,\xi}^{(b)}(s) 
&= \Fb_{n,\xi}^{(b)}(s,y_n) - \hat C_n(s) \Fb_{n,\xi}^{(b)}(1,y_n)  \\
&= \Fb_{n,\xi}^{(b)}(s,1) - C(s) \Fb_{n,\xi}^{(b)}(1,1) \\
&= \Sb_{n,\xi}^{(b)}(s,1) - C(s) \Sb_{n,\xi}^{(b)}(1,1) 
\end{align*}
where the third equality is a consequence of Proposition~\ref{prop:step_boot} and
where the second equality is a consequence of $\sup_{s\in[0,1]}|\hat C_n(s) - C(s)| = \op$ by Theorem \ref{thm:integr_scedasis} and
\[ 
\sup_{s \in [0,1]} | \Fb_{n,\xi}^{(b)}(s,y_n)-\Fb_{n,\xi}^{(b)}(s,1)| = \op,
\]
which can be seen to hold by the same argumentation as in the proof of Theorem \ref{thm:integr_scedasis} for showing (\ref{eq:fneq}), since $\mathbb{F}_{n,\xi}^{(b)} \wto \mathbb{S}$ in $\big( \ell^{\infty}([0,1]^2),\|\cdot\|_{\infty} \big)$ by Proposition \ref{prop:step_boot}.
Hence, \eqref{eq:cndec} and Proposition~\ref{prop:step} imply the representation
 \begin{multline*}
\Big( \sqrt{k} \big( \hat C_n-C \big), \mathbb{\tilde C}^{(1)}_{n,\xi}, \ldots, \mathbb{\tilde C}^{(B)}_{n,\xi} \Big) 
=
 \Big( \mathbb{S}_n(s,1) - C(s)\mathbb{S}_n(1,1) , \\
 \Sb_{n,\xi}^{(1)}(s,1) - C(s) \Sb_{n,\xi}^{(1)}(1,1), \ldots,\Sb_{n,\xi}^{(B)}(s,1) - C(s) \Sb_{n,\xi}^{(B)}(1,1) \Big) + \op.
 \end{multline*}
By Proposition~\ref{prop:sstep_boot} and the continuous mapping theorem, the previous expression weakly converges to
\begin{multline*}
 \Big( \mathbb{S}(s,1) - C(s)\mathbb{S}(1,1) , \Sb^{(1)}(s,1) - C(s) \Sb^{(1)}(1,1), \ldots,\Sb^{(B)}(s,1) - C(s) \Sb^{(B)}(1,1) \Big) \\
 = \Big( \mathbb{C}, \mathbb{C}^{(1)}, \ldots, \mathbb{C}^{(B)} \Big)
 \end{multline*}
in $ (\ell^{\infty}([0,1]), \|\cdot\|_{\infty})^{B+1}$. Finally, since 
\[
  \mathbb{D}^{(b)}_{n,\xi}(s) = \mathbb{\tilde D}^{(b)}_{n,\xi}(s) - \sqrt{k}\bar{\xi}^{(b)}\hat C_n(s)
  \]
  and $\hat C_n(1)=1$, we have $\mathbb{C}_{n,\xi}^{(b)} = \mathbb{\tilde C}_{n,\xi}^{(b)}$, which proves the theorem.
  \end{proof}

\begin{proof}[Proof of Corollary \ref{cor:test_boot}]
 By Theorem \ref{thm:boot_jointconv} and the Continuous Mapping Theorem, we have that, under $H_0$, as $n \to \infty$, 
 \begin{align*}
  \big( S_{n,1}, S_{n,1}^{(1)},\ldots,S_{n,1}^{(B)} \big) & \wto \big( \|\mathbb C\|_{\infty}, \|\mathbb C^{(1)}\|_{\infty}, \ldots, \|\mathbb C^{(B)}\|_{\infty} \big), \\
  \big( T_{n,1}, T_{n,1}^{(1)},\ldots,T_{n,1}^{(B)} \big) & \wto \Big( \int_0^1 \mathbb C(s)^2 \ \mathrm{d}s, \int_0^1 \mathbb C^{(1)}(s)^2 \ \mathrm{d}s, \ldots,\int_0^1 \mathbb C^{(B)}(s)^2 \ \mathrm{d}s \Big).
 \end{align*}
 Note that $\mathbb C = \sigma \mathbb B$ in distribution by (\ref{brownianbridge}), where $\mathbb B$ is a Brownian Bridge on $[0,1]$, which implies that $\|\mathbb C\|_{\infty}$ and $\int_0^1 \mathbb C(s)^2 \ \mathrm{d}s$ are continuous random variables. 
  Further note that $(\xi_1^{\scs (b)},\ldots,\xi_m^{\scs (b)})_{b=1,\ldots,B}$ are i.i.d. Overall, Lemma 4.2 in \cite{BucKoj19} is applicable, which proves the assertion under $H_0$.
 For the assertion under $H_1$ let us consider $\varphi_{n,B,S}$; $\varphi_{n,B,T}$ can be treated analogously. Note that under $H_1$, \[k^{-1/2} S_{n,1} = \sup_{ s\in [0,1]} |\hat C_n(s)-s| \pto \sup_{s \in [0,1]} |C(s)-s| > 0\] and $r^{-1/2} S_{n,1}^{(b)} = r^{-1/2} \|\mathbb C_{n,\xi}^{(b)}\|_{\infty} = \op$, such that $S_{n,1}^{(b)} = O_{\Prob}(r^{1/2})$ for any $b \in \{1,\ldots,B\}$. The claim follows since $r = o(k)$ by Condition \ref{cond:mixing}.
\end{proof}

\begin{proof}[Proof of Corollary \ref{cor:test_norm}]
 The proof of the statement regarding the null hypothesis is an immediate consequence of Theorem \ref{thm:boot_jointconv}. Under $H_1$, one can easily show that $S_{n,2}, T_{n,2}$ converge to $\infty$ in probability, which implies the respective assertion under $H_1$.
\end{proof}

%%%%%%%%%%%%%%%%%%%%%%%%%%%%%%%%%%%%%%%%%%%%%%%%%%%%%%

\begin{proof}[Proof of Proposition \ref{prop:sstep_boot}]
Fix $b \in \{1,\ldots,B\}$. We only show weak convergence of $(\Sb_n, \Sb_{n,\xi}^{(b)})$; the joint weak convergence of all $B+1$ components can be shown analogously. In the following, we omit the upper index $(b)$ at all instances.
Recall $c_\infty(L)=1+ L\|c\|_\infty$ and $X_{n,i}'$ from \eqref{eq:xnprime}
and $v_n = \Pr(X'_{n,i} \ne 0) = \tfrac{k}{n} c_\infty(L)$. For $(s,x)\in [0,1]^2$, write 
\begin{align}
\mathbb{S}_{n,\xi}(s,x) = c_{\infty}(L)^{1/2} \tilde \Z_{n,\xi}(s,x) + R_{n,\xi}(s,x), \label{S_nxi}
\end{align}
where
\begin{align*}
\tilde \Z_{n,\xi}(s,x)
&=
\frac{1}{\sqrt{nv_n}} \sum_{j=1}^{m} \xi_j \sum_{i \in I_j} \I(\tfrac{i}{n} \leq s) \Big\{ \bm 1( X'_{n,i} > c_\infty(L)- c(\tfrac{i}{n}) x) - \Pr(X'_{n,i} > c_\infty(L)- c(\tfrac{i}{n}) x) \Big\}, \\
R_{n,\xi}(s,x) &= \frac{x\sqrt{k}}{n} \sum_{j=1}^{m} \xi_j \sum_{i \in I_j} \I(\tfrac{i}{n} \leq s) c( \tfrac{i}{n}).
\end{align*}
First, we show that $R_{n,\xi}=\op$. Due to Condition \ref{cond:mixing} one can easily show that it suffices to prove $\tilde R_{n,\xi} = \op$, where 
 \[ 
 \tilde R_{n,\xi}(s) = \sum_{j=1}^{m} f_{n,j}(s), \quad  f_{n,j}(s) = \frac{\sqrt{k}}{n} \xi_j \I(\tfrac{j}{m} \leq s) \sum_{i \in I_j} c(\tfrac{i}{n}), \quad s \in [0,1].
 \]
 First, for $s\in [0,1]$, we have $\Exp[\tilde R_{n,\xi}(s)] = 0$ and 
\[ 
\Var\big(  \tilde R_{n,\xi}(s) \big) 
= 
\frac{k}{n^2} \sum_{j=1}^{m} \Big\{ \I(\tfrac{j}{m} \leq s) \sum_{i\in I_j} c(\tfrac{i}{n}) \Big\}^2 
\leq 
\| c \|_{\infty}^2 \frac{kr^2m}{n^2} =  \| c \|_{\infty}^2 \frac{kr}{n} = o(1) 
\]
by Condition \ref{cond:mixing}, such that $\tilde R_{n,\xi}(s) = \op$ for any fixed $s\in[0,1]$. It remains to show tightness of $\tilde R_{n,\xi}$. 
To this, we will apply Lemma A.1 from \cite{KleVolDetHal16} with $\psi(x)=x^2, \ \bar \eta =2/m, \ T=[0,1]$ and $d(s,t)=|s-t|$. Note that the Orlicz-norm with $\psi(x)=x^2$ coincides with the $L_2$-norm $\|\cdot\|_2$. 
 First, for all $|s-t| \ge \bar \eta/2=1/m$, we have
 \[ \|\tilde R_{n,\xi}(s)-\tilde R_{n,\xi}(t)\|_2  \le 2\|c\|_\infty \frac{kr}n |s-t| \le |s-t| \]
 for sufficiently large $n$ by Condition \ref{cond:mixing}.
 By Lemma A.1 in \cite{KleVolDetHal16}, for any $\delta>0, \eta \geq \bar\eta$, there exists a random variable $S'$ and a constant $K'>\infty$, such that 
\begin{align}
 \Prob \Big( \sup_{d(s,t) < \delta} |\tilde R_{n,\xi}(s)-\tilde R_{n,\xi}(t)| > 2 \varepsilon \Big) \leq \Prob(|S'|> \varepsilon) + \Prob \Big( \sup_{d(s,t) \leq \bar \eta} |\tilde R_{n,\xi}(s)-\tilde R_{n,\xi}(t)| > \varepsilon/2 \Big), \label{bound}
\end{align}
for all $\varepsilon > 0$, where 
\begin{align*}
 \Prob (|S'|>\varepsilon) \leq \Big(\frac{8K'}{\varepsilon}\Big)^2 \Big( \int_{1/m}^{\eta} D(x,d)^{1/2} \mathrm{d}x + (\delta+4/m) D(\eta,d) \Big)^2.
\end{align*}
Here, $D(\cdot,d)$ denotes the \textit{packing number} on $([0,1],d)$ and satisfies $D(x,d) \leq 4x^{-1}+1$, $x >0$, see \cite{VanWel96}, page 98. Thus,
\begin{align} 
\lim_{\delta \downarrow 0} \limsup_{n \to \infty} \Prob(|S'|>\varepsilon) \leq \Big(\frac{8K'}{\varepsilon}\Big)^2 \Big( \int_{0}^{\eta} (4x^{-1}+1)^{1/2} \mathrm{d}x \Big)^2. \label{S}
\end{align}
Further, we have
\begin{align*}
 |\tilde R_{n,\xi}(s) - \tilde R_{n,\xi}(t)| \leq M\|c\|_\infty \frac{\sqrt{k}}{m} \sum_{j=1}^m |\I(j/m \leq s)-\I(j/m \leq t)|,
\end{align*}
where $|\I(j/m \leq s)-\I(j/m \leq t)| = \I(s \wedge t < j/m \leq s \vee t)$ does not equal zero for at most two different $j \in \{1,\ldots,m\}$, if $d(s,t)=|s-t| \leq \bar \eta =2/m$. Consequently,
\[  \Prob \Big( \sup_{d(s,t) \leq \bar \eta} |\tilde R_{n,\xi}(s)-\tilde R_{n,\xi}(t)| > \varepsilon/2 \Big) \leq \I\Big( 2M\|c\|_\infty \tfrac{\sqrt{k}}{m} > \varepsilon/2 \Big), \] which converges to zero as $n \to \infty$ since $rk=o(n)$ by Condition \ref{cond:mixing}. Altogether, by (\ref{bound}) and (\ref{S}), we have shown, for all $\varepsilon > 0$,
\[ \lim_{\delta \downarrow 0} \limsup_{n \to \infty} \Prob \Big( \sup_{d(s,t) < \delta} |\tilde R_{n,\xi}(s)-\tilde R_{n,\xi}(t)| > 2 \varepsilon \Big) \leq \Big(\frac{8K'}{\varepsilon}\Big)^2 \Big( \int_{0}^{\eta} (4x^{-1}+1)^{1/2} \mathrm{d}x \Big)^2, \]
which can be made arbitrarily small by choosing $\eta$ accordingly. This concludes the proof of $R_{n,\xi}=\op$.

By equation (\ref{S_nxi}) we obtain $\Sb_{n,\xi}= c_{\infty}(L)^{1/2} \tilde \Z_{n,\xi} + \op$. Since $|\xi_j|\leq M$ the same calculation as in (\ref{Zn_tilde}) in the proof of Proposition \ref{prop:sstep} (for treating $\Z_n$ and $\tilde \Z_n$) yields $\sup_{(s,x)\in [0,1]^2} |\tilde \Z_{n,\xi}(s,x) - \Z_{n,\xi}(s,x)| = \op$ with
\begin{align*} 
\Z_{n,\xi}(s,x) &=
\frac{1}{\sqrt{nv_n}} \sum_{j=1}^{m} \xi_j \I(\tfrac{j}{m} \leq s) \sum_{i \in I_j} \Big\{ \bm 1( X'_{n,i} > c_\infty(L)- c(\tfrac{i}{n}) x) \\
& \hspace{6cm} - \Pr(X'_{n,i} > c_\infty(L) - c(\tfrac{i}{n}) x) \Big\} \\
&= \frac{1}{\sqrt{nv_n}} \sum_{j=1}^{m} \xi_j \big\{ f_{j,n,s,x}(Y_{n,j}) - \Exp[f_{j,n,s,x}(Y_{n,j})] \big\},
\end{align*}
where $Y_{n,j} = (X'_{n,i})_{i \in I_j}$, $j=1,\ldots,m$, and $f_{j,n,s,x}$ is defined as in (\ref{f_jnsx}).
Further, by (\ref{Sn_Zn}) we know that $\Sb_n = c_{\infty}(L)^{1/2} \Z_n + \op$, where
\[
 \Z_n(s,x) = \frac{1}{\sqrt{nv_n}} \sum_{j=1}^{m}  \big\{ f_{j,n,s,x}(Y_{n,j}) -  \Exp[ f_{j,n,s,x}(Y_{n,j})] \big\},
\]
as defined after (\ref{Zn_tilde_def}), the only difference to $\Z_{n,\xi}$ being the multipliers $\xi_j$. It remains to show weak convergence of $(\Z_n,\Z_{n,\xi})$. We start with the corresponding weak convergence of the fidis. Since $\xi_1,\ldots,\xi_m$ are independent with $|\xi_j|\leq M$ and independent of $Y_{n,1},\ldots,Y_{n,m}$ the proof is analogous to the one of Proposition \ref{prop:sstep}. Let us just calculate the covariance function for independent blocks $Y_{n,1},\ldots,Y_{n,m}$. Note that $\Exp[\xi_j]=0$ and $\Exp[\xi_j^2]=1$. For $(s,x),(s',x') \in [0,1]^2$, we obtain
\begin{align*}
 & \frac{1}{nv_n} \sum_{j=1}^{m} \Cov \big( f_{j,n,s,x}(Y_{n,j}), \xi_j f_{j,n,s',x'}(Y_{n,j}) \big) = 0
\end{align*}
and 
\begin{align*}
 \frac{1}{nv_n} \sum_{j=1}^{m} \Cov \big( \xi_j f_{j,n,s,x}(Y_{n,j}), \xi_j f_{j,n,s',x'}(Y_{n,j}) \big) 
 &= \frac{1}{nv_n} \sum_{j=1}^{m} \Exp \big[ f_{j,n,s,x}(Y_{n,j}) f_{j,n,s',x'}(Y_{n,j}) \big],
\end{align*}
which equals $\mathfrak c_n((s,x),(s',x'))$ defined in (\ref{cn_cov}) and converges to $\mathfrak c((s,x),(s',x'))$ from Proposition \ref{prop:sstep} by the corresponding proof.

With regard to the asymptotic tightness, note that by Lemma 1.4.3 in \cite{VanWel96} it suffices to show asymptotic tightness of $\Z_n$ and $\Z_{n,\xi}$ separately. Asymptotic tightness of $\Z_n$ has been shown in  the proof of Proposition \ref{prop:sstep}. Concerning the asymptotic tightness of $\Z_{n\xi}$, the proof follows analogously. Here, the conditions (1)-(5) in the proof of Proposition \ref{prop:sstep} can immediately be seen to hold since $|\xi_j|\leq M$, and condition (6) follows since the function $\rho_n$ is the same as before due to $\Exp[\xi_j^2]=1$.
\end{proof}

\begin{proof}[Proof of Proposition \ref{prop:step_boot}]
 Let $s,x \in [0,1]$ and $b \in \{1,\ldots,B\}$. Set $\varepsilon_n(x)=V(n/(kx)) = F^{-1}(1-kx/n)$ such that, almost surely,
	\begin{align*} 				  
	\mathbb{F}^{(b)}_{n,\xi}(s,x) 
	&= \frac{1}{\sqrt{k}}\sum_{j=1}^{m} \xi_j^{(b)} \sum_{i \in I_j} \I \Big\{ U_{i}^{(n)} > 1-\frac{kx}{n} \frac{1-F_{n,i}(\varepsilon_n(x))}{1-F(\varepsilon_n(x))
	}  \Big\} \I(\tfrac{i}{n} \leq s).
	\end{align*}
	Note that $|\xi_j^{(b)}|\leq M$. Then, by relation (\ref{eq:uisubset}) and the definition of $\mathbb{S}^{(b)}_{n,\xi}$ and $\mathbb S_n$ in \eqref{eq:sbnxi} and \eqref{eq:defsstep}, respectively, we obtain
	\begin{align*}
	 &\phantom{{}={}} |\mathbb{F}_{n,\xi}^{(b)}(s,x)-\mathbb{S}^{(b)}_{n,\xi}(s,x)| \\
	 & \leq  \frac{M}{\sqrt{k}} \sum_{j=1}^m \sum_{i \in I_j} \I(\tfrac{i}{n} \leq s) \Big| \I\Big( U_i^{(n)} > 1- c(\tfrac{i}{n}) (1+\delta_n) \frac{kx}{n} \Big) - \I\Big( U_i^{(n)} > 1-c(\tfrac{i}{n}) \frac{kx}{n} \Big) \Big| \\
	 & \hspace{3cm} + \Big| \I\Big( U_i^{(n)} > 1- c(\tfrac{i}{n}) (1-\delta_n) \frac{kx}{n} \Big) - \I\Big( U_i^{(n)} > 1-c(\tfrac{i}{n}) \frac{kx}{n} \Big) \Big| \\
	 &= M \big\{ \mathbb S_n(s,x(1+\delta_n)) + \sqrt{k}C(s)x(1+\delta_n)- (\mathbb S_n(s,x)+\sqrt{k}C(s)x) \\
	 & \hspace{1cm} - (\mathbb S_n(s,x(1-\delta_n)) + \sqrt{k}C(s)x(1-\delta_n)) + \mathbb S_n(s,x)+\sqrt{k}C(s)x \big\} \\  
	 &= M \big\{ \mathbb S_n(s,x(1+\delta_n)) - \mathbb S_n(s,x(1-\delta_n)) + 2 C(s)x\sqrt{k}\delta_n \big\}, 
	\end{align*}
	where $\delta_n$ is defined after (\ref{eq:uisubset}).
	Consequently, 
	\begin{align*}
		\sup_{(s,x)\in [0,1]^2} |\mathbb{F}_{n,\xi}^{(b)}(s,x)-\mathbb{S}^{(b)}_{n,\xi}(s,x)| \leq M w_{2 \delta_n}(\mathbb{S}_n) +  2M \sqrt{k}\delta_n,
	\end{align*}
	where $w_\delta(\mathbb S_n)$ is defined in (\ref{def:w_delta}) in the proof of Proposition \ref{prop:step}.  There, it is further shown that $w_{2\delta_n}(\Sb_n) = \op$ and $\sqrt{k}\delta_n = o(1)$ by Condition \ref{cond:2nd order}, which implies the assertion. 
	\end{proof}

\section{Proofs for Section \ref{sec:EI}} \label{sec:pEI}

\begin{proof}[Proof of Lemma \ref{Zni_exp}]
	For $x > 0$ write 
	\begin{align*}
		\Prob (Z_{n,1+\lfloor \xi k' \rfloor} \geq x)
		&= \Prob \Big( \max_{i \in I'_{1+\lfloor \xi k'\rfloor}} F(X_i^{(n)}) \leq 1-x/\bs \Big) \\
		&= \Prob \Big( X_i^{(n)} \leq F^{-1}(1-x/\bs) \ \textrm{ for all } \ i \in I'_{1+\lfloor \xi k'\rfloor} \Big) \\
		&= \Prob \Big( Z_i \leq \frac{1}{1-F_{n,i}} \big(F^{-1}(1-x/\bs) \big) \ \textrm{ for all } \ i \in I'_{1+\lfloor \xi k'\rfloor} \Big).
	\end{align*}
	By Corollary \ref{cor1} the last term equals 
	\[ \Prob \Big( Z_i \leq \frac{\bs}{c(\xi) x} \ \textrm{ for all } \ i \in I'_{1+\lfloor \xi k'\rfloor} \Big) + o(1) = \Prob \Big( \max_{i \in I'_{1+\lfloor \xi k'\rfloor}} U_i \leq 1- \frac{c(\xi)x}{\bs} \Big) + o(1), \]
	which converges to $\exp(-\theta c(\xi) x)$ by (\ref{extremal_index}).
\end{proof}

\begin{proof}[Proof of Theorem \ref{thm:cons_EI}]
We start with part (a). By Theorem \ref{thm:scedasis} we know that $\tilde c_n(x) = c(x)+\op$ for any $x \in [0,1]$, and the continuous mapping theorem implies that $\hat c_n(x)^{-1} = \max(\tilde c_n(x),\kappa)^{-1} = c(x)^{-1}+\op$ for any $x \in [0,1]$. Since $\hat c_n \geq \kappa$, we obtain 
$\Exp \big[ |\hat c_n(x)^{-1}|^p \big] \leq \kappa^{-p} < \infty$
for any $p > 0$ and $x \in [0,1]$. By Example 2.21 in \cite{Van98}, this implies $\Exp \big[ \hat c_n(x)^{-1} \big] \to c(x)^{-1}$ for any $x \in [0,1]$, such that 
\[ 
\Exp[\hat{\tau}_n] = \int_{0}^{1} \Exp \big[ \hat c_n(x)^{-1} \big] \ \mathrm{d}x \to \int_{0}^{1} c(x)^{-1} \ \mathrm{d}x = \tau 
\] 
by the dominated convergence theorem. Next, we show that $\Var(\hat \tau_n)=o(1)$. Note that $\hat c_n(x)^{-1} \hat c_n(y)^{-1}= c(x)^{-1}c(y)^{-1}+\op$ for any $x,y \in [0,1]$ and 
$
\Exp \big[ | \hat{c}_n(x)^{-1}\hat{c}_n(y)^{-1}|^p \big] \leq \kappa^{-2p} 
$
for any $p>0$ and $x,y \in [0,1]$, such that as above $\Exp \big[ \hat{c}_n(x)^{-1}\hat{c}_n(y)^{-1} \big] = c(x)^{-1}c(y)^{-1}$ + o(1). Thus, by Fubini's theorem
\begin{align*} 
\Var(\hat \tau_n) 
&=  
\int_{0}^{1} \int_{0}^{1} \Exp \big[ \hat{c}_n(x)^{-1}\hat{c}_n(y)^{-1} \big] \ \mathrm{d}x \mathrm{d}y 
	- \Big( \int_{0}^{1} \Exp \big[ \hat{c}_n(x)^{-1} \big] \ \mathrm{d}x \Big)^2 \\
&\to  \int_{0}^{1} \int_{0}^{1} c(x)^{-1}c(y)^{-1} \ \mathrm{d}x \mathrm{d}y - \Big( \int_{0}^{1} c(x)^{-1} \ \mathrm{d}x \Big)^2 = 0.
\end{align*}
The assertion in (a) follows from Markov's inequality.
	
We continue with part (b). Write $\hat T_n = S_{n1}+S_{n2}+S_{n3}$, where 
\[ 
S_{n1} = \frac{1}{k'} \sum_{j=1}^{k'} \hat Z_{n,j}-Z_{n,j}, \quad 
S_{n2} = \frac{1}{k'} \sum_{j=1}^{k'} Z_{n,j}-\Exp[Z_{n,j}], \quad 
S_{n3} = \frac{1}{k'} \sum_{j=1}^{k'} \Exp[Z_{n,j}]. 
\]
First, we show $S_{n3} \to \tau/\theta$. Write $S_{n3}= \int_{0}^{1} \varphi_n(\xi) \ \mathrm{d}\xi$, where 
\[
\varphi_n(\xi) = \Exp[Z_{n,1+\lfloor \xi k' \rfloor}] \to (\theta c(\xi))^{-1}
\] 
by Lemma \ref{Zni_exp} and  uniform integrability, which follows from \ref{cond:uniform integr}. Hence, the dominated convergence theorem implies that $S_{n3} \to \int_{0}^{1} (\theta c(\xi))^{-1} \ \mathrm{d}\xi = \tau/\theta$; note $\sup_{n \in \N} \|\varphi_n\|_{\infty} < \infty$ by Condition \ref{cond:uniform integr}.
	
In the following, we prove $S_{n1}=\op$ and $S_{n2}=\op$, and start with $S_{n2}$. Split $S_{n2}$ into $S_{n2}^{\textrm{even}}$ and $S_{n2}^{\textrm{odd}}$, which are defined as $S_{n2}$ but with $j$ only ranging over the even or odd numbers in $\{1,\ldots,k'\}$, respectively. It suffices to show that $S_{n2}^{\textrm{even}}$ and $S_{n2}^{\textrm{odd}}$ are asymptotically negligible. We only treat $S_{n2}^{\textrm{even}}$; the proof for $S_{n2}^{\textrm{odd}}$ is similar. 

For $n \in \N$, let $(Z_{n,j}^{\ast})_{j=1,\ldots,k'}$ denote an independent sequence with $Z_{n,j}^{\ast}$  being equal in distribution to $Z_{n,j}$ for $j=1,\ldots,k'$. Since the observations making up the even numbered blocks are separated by at least $\bs$ observations, we may follow the argumentation in \cite{Ebe84} to obtain 
\[ 
d_{\rm TV}\Big( P^{(Z_{n,2j})_{1\leq j \leq \lfloor k'/2 \rfloor}}, \ P^{(Z_{n,2j}^{\ast})_{1\leq j \leq \lfloor k'/2 \rfloor}} \Big) \leq \lfloor k'/2 \rfloor \beta(\bs), 
\] where $d_{\rm TV}$ denotes the total variation distance between two probability laws. Since $k'\beta(\bs) = o(1)$ by \ref{cond:blocksize}, the above expression converges to zero as well, and $S_{n2}^{\textrm{even}} = S_{n2}^{\textrm{even},\ast} + \op$, where $S_{n2}^{\textrm{even},\ast}$ is defined as $S_{n2}^{\textrm{even}}$ but in terms of $(Z_{n,j}^{\ast})_j$.  Finally, $\Exp[S_{n2}^{\textrm{even},\ast}]=0$ and 
\[ 
\Var(S_{n2}^{\textrm{even},\ast}) 
= 
\frac{1}{(k')^2} \sum_{j=1, j \text{ even}}^{k'} \Var(Z_{n,j}^{\ast}) 
\le 
\frac{1}{(k')^2} \sum_{j=1}^{k'} \Var(Z_{n,j}^{\ast}) 
= \frac{1}{k'} \int_{0}^{1} g_n(\xi) \ \mathrm{d}\xi, 
\]
where $g_n(\xi) = \Var(Z_{n,1+\lfloor \xi k' \rfloor}) \to (\theta c(\xi))^{-2}$ by Lemma \ref{Zni_exp} and uniform integrability from \ref{cond:uniform integr}, which implies $\Var(S_{n2}^{\textrm{even},\ast}) =o(1)$ and $S_{n2}^{\textrm{even},\ast}=\op$.  
	
It remains to show $S_{n1}=\op$. Note that the STEP $\mathbb{F}_n$ from \eqref{eq:defsstep} with $k=k'$  satisfies
\[ 
\mathbb{F}_n(1,\bs(1-F(x))) = \bs \sqrt{k'} %\frac{n}{\sqrt{m}} 
\big\{ F(x)-\hat F_n(x) \big\},
\] 
which yields $\hat{Z}_{n,j}-Z_{n,j}= \frac{1}{\sqrt{k'}} \mathbb{F}_n(1,Z_{n,j})$ for $j=1,\ldots,k'$ by the definition of $Z_{n,j}$ and $\hat Z_{n,j}$ in \eqref{eq:znj}. Therefore, 
\[ 
S_{n1} 
= \frac{1}{k'} \sum_{j=1}^{k'} \frac{1}{\sqrt{k'}} \mathbb{F}_n(1,Z_{n,j}) = \int_{0}^{\infty} \mathbb{F}_n(1,x) \ \mathrm{d}H_n(x), 
\]
where $H_n(x)=(k')^{-3/2} \sum_{j=1}^{k'} \I(Z_{n,j}\leq x)$. Note that $\sup_{x \in [0,T]} |H_n(x)| \leq (k')^{-1/2} =o(1)$ for any $T \in \N$. Under the imposed conditions, Proposition \ref{prop:step} is applicable for $k=k'$, yielding $\{\mathbb{F}_n(1,x)\}_{x \in [0,T]} \wto \{\mathbb{S}(1,x)\}_{x \in [0,T]}$ in $(\ell^\infty([0,T]),\|\cdot\|_{\infty})$, such that 
\[ 
S_{n1}(T) 
:= 
\int_{0}^{T} \mathbb{F}_n(1,x) \ \mathrm{d}H_n(x) = \op, \quad T \in \N, 
\] 
by Lemma C.8 in \cite{BerBuc17}. By Theorem 4.2 in \cite{Bil68}, the proof of  $S_{n1}=\op$ is finished once we show that, for any $\delta > 0$,
\[ 
\lim\limits_{T \to \infty} \limsup_{n \to \infty}  P\big( |S_{n1}-S_{n1}(T)| > \delta \big) = 0.
\]
Set $f_n(x,z)= \I(x > V(\bs/z)) - z/\bs$, such that $S_{n1}= \frac1{(k')^2} \sum_{i=1}^{k'} \sum_{j=1}^n  f_n(X_j^{(n)}, Z_{n,i})$.
Write $S_{n1}-S_{n1}(T) = A_{n,T}+B_{n,T}+C_{n,T}$, where 
\begin{align*}
	A_{n,T} &= \frac{1}{(k')^2} \sum_{i=1}^{k'} \sum_{j \in \{i-1,i,i+1\}} \sum_{s \in I'_j} f_n(X_s^{(n)},Z_{n,i}) \I(Z_{n,i}\geq T), \\
	B_{n,T} &= \frac{1}{(k')^2} \sum_{i=1}^{k'-2} \sum_{j=i+2}^{k'} \sum_{s \in I'_j} f_n(X_s^{(n)},Z_{n,i}) \I(Z_{n,i}\geq T), \\
	C_{n,T} &= \frac{1}{(k')^2} \sum_{i=3}^{k'} \sum_{j=1}^{i-2} \sum_{s \in I'_j} f_n(X_s^{(n)},Z_{n,i}) \I(Z_{n,i}\geq T).
\end{align*}
First, $|A_{n,T}| \leq 3\bs/k' = o(1)$ by Condition \ref{cond:blocksize}. It remains to show, for any $\delta > 0$, 
\begin{align*}
	& \lim_{T \to \infty} \limsup_{n \to \infty} P(|B_{n,T}|> \delta) = 0, \quad \lim_{T \to \infty} \limsup_{n \to \infty} P(|C_{n,T}|> \delta) = 0.
\end{align*}
We only consider $C_{n,T}$; $B_{n,T}$ can be treated similarly. Write 
\[ 
C_{n,T} = \frac{1}{k'} \sum_{i=3}^{k'} \varphi_{n,i-2}(Z_{n,i}) \I(Z_{n,i} \geq T) \] with \[ \varphi_{n,i-2}(z) = \frac{1}{k'} \sum_{j=1}^{i-2} \sum_{s \in I'_j} f_n(X_s^{(n)},z). 
\]
For fixed $i \in \{3, \dots, k'\}$, consider the expectation $\Exp\big[ |\varphi_{n,i-2}(Z_{n,i})| \I(Z_{n,i}\geq T) \big]$.
By Berbee's coupling lemma \citep{Ber79}, we may construct a random variable $Z_{n,i}^{\ast}$ independent of $((X_s^{\scs (n)})_{s \in I'_j} )_{j=1,\ldots,i-2}$ and equal in distribution to $Z_{n,i}$ with  
\[ 
\Prob(Z_{n,i}^{\ast} \neq Z_{n,i}) = \beta\big( \sigma(Z_{n,i}), \sigma\big( \big( (X_s^{(n)})_{s \in I'_j} \big)_{j=1,\ldots,i-2} \big) \big) \leq \beta(\bs). 
\]
Hence,
\begin{multline*}
\Exp\big[ |\varphi_{n,i-2}(Z_{n,i})| \I(Z_{n,i}\geq T) \big] 
=
\Exp\big[ |\varphi_{n,i-2}(Z_{n,i}^*)| \I(Z_{n,i}^*\geq T) \big]  \\
+
\Exp\big[ \big\{ |\varphi_{n,i-2}(Z_{n,i})| \I(Z_{n,i}\geq T) - |\varphi_{n,i-2}(Z_{n,i}^{\ast})| \I(Z_{n,i}^{\ast}\geq T) \big\} \I(Z_{n,i}\neq Z_{n,i}^{\ast}) \big].
\end{multline*}
Since $|\varphi_{n,i-2}| \leq \bs$ the second summand can be bounded by $2\bs \Prob(Z_{n,i}\neq Z_{n,i}^{\ast}) \leq 2\bs \beta(\bs) \leq 2k'\beta(\bs)=o(1)$ by Condition \ref{cond:blocksize}, uniformly in $i$. Now, consider the first summand in the above display, for which we first treat $\Exp[|\varphi_{n,i-2}(z)|]$ for $T \leq z \leq \bs$ (note that $Z_{n,i}^{\ast} \leq \bs$ a.s.). We have 
\begin{align*}
	\Exp[|\varphi_{n,i-2}(z)|] & \leq \frac{1}{k'} \sum_{j=1}^{i-2} \sum_{s \in I'_j} \Exp \big[ |f_n(X_s^{(n)},z)| \big]
	\leq z + \frac{1}{k'} \sum_{j=1}^{i-2} \sum_{s \in I'_j} \Prob \big( X_s^{(n)} > V(\bs/z) \big).
\end{align*}
Since $\Prob \big( X_s^{(n)} > V(\bs/z) \big) = 1-F_{n,s}\big( (1-F)^{-1}(z/\bs) \big)$ and by Condition \ref{cond:2nd order}, there exists some $\tau > 0$  such that, for all $s \leq n$ and $n$ large enough, 
\[
\Prob \big( X_s^{(n)} > V(\bs/z) \big) < \frac{z}{\bs} c\Big( \frac{s}{n}\Big) \bigg\{ 1+\tfrac{\tau}{c_{\min}} A \Big(\frac{\bs}{z}\Big) \bigg\}. 
\]
As a consequence, uniformly in $i$,
\[ 
\Exp[|\varphi_{n,i-2}(z)|] \leq z + z \|c\|_{\infty} \bigg\{ 1+\tfrac{\tau}{c_{\min}} A \Big(\frac{\bs}{z}\Big) \bigg\}. 
\]
Since $A$ is eventually decreasing, the last expression can be bounded by \[z\big[ 1+ \|c\|_{\infty} \big\{ 1+\tfrac{\tau}{c_{min}} A (1) \big\} \big]\] for $T \leq z \leq \bs$. After conditioning on $Z_{n,i}^{\ast}$ we thus obtain with the Cauchy-Schwarz-inequality
\begin{align*}
\Exp \big[ |\varphi_{n,i-2}(Z_{n,i}^{\ast})| \I(Z_{n,i}^{\ast}\geq T) \big] 
& \leq  
\Big[ 1+ \|c\|_{\infty} \Big\{ 1+\tfrac{\tau}{c_{min}} A (1) \Big\} \Big]  \Exp \big[ Z_{n,i}^{\ast} \I(Z_{n,i}^{\ast}\geq T) \Big].
\end{align*}
Since $Z_{n,i}^*$ has the same distribution as $Z_{n,i}$ and by the Cauchy Schwartz inequality, we have thus found the bound
\begin{align*} 
 \Exp[|C_{n,T}|] 
 &\le o(1) + \frac1{k'} \sum_{i=1}^{k'} \Big[ 1+ \|c\|_{\infty} \Big\{ 1+\tfrac{\tau}{c_{min}} A (1) \Big\} \Big] \Exp \big[ Z_{n,i} \I(Z_{n,i}\geq T) \Big] \\
&\lesssim  o(1) + \int_{0}^{1} g_n(\xi) \ \mathrm{d}\xi,
\end{align*}
where $g_n(\xi) = \Exp[Z_{n,1+\lfloor \xi k' \rfloor}^2]^{1/2} \Prob(Z_{n,1+\lfloor \xi k' \rfloor} \geq T)^{1/2}$ converges to $\Exp[V_{\xi}^2]^{1/2} \Prob(V_{\xi} \geq T)^{1/2}$ as $n \to \infty$ for $V_{\xi} \sim \textrm{Exp}(\theta c(\xi))$ by Lemma \ref{Zni_exp} and Condition \ref{cond:uniform integr}. Altogether,
\begin{align*} 
& \lim_{T \to \infty} \limsup_{n \to \infty}  
\le 
\lim_{T \to \infty} \int_{0}^{1} \Exp[V_{\xi}^2]^{1/2} P(V_{\xi}\geq T)^{1/2} \ \mathrm{d}\xi = 0,
\end{align*}
which implies (b). 
\end{proof}

%%%%%%%%%%%%%%%%%%%%%%%%%%%%%%%%%%%%%%%%%%%%%%%%%%%%%%%%%%%%%%%%%%%%%%%%%

\section{Auxiliary Results} \label{sec:auxiliary2}

\begin{lemma} \label{Bn-Cn}
Fix $\xi \in [0,1)$ and $x>0$. Under Conditions \ref{cond:basic}-\ref{cond:sbias}, \ref{cond:2nd order} and \ref{cond:blocksize}, $A_n = B_n + \op$ and $B_n=C_n +\op$ as $n \to \infty$, where  
\begin{align*} 
A_n &= \sum_{i \in I'_{1+\lfloor \xi k'\rfloor}} \I \Big( Z_i > \frac{1}{1-F_{n,i}} \big( F^{-1}(1-x/\bs) \big) \Big), \\
B_n &= \sum_{i \in I'_{1+\lfloor \xi k'\rfloor}} \I \Big( Z_i > \frac{\bs}{c(i/n) x} \Big), \qquad C_n = \sum_{i \in I'_{1+\lfloor \xi k'\rfloor}} \I \Big( Z_i > \frac{\bs}{c(\xi)x} \Big). 
\end{align*}
\end{lemma}

\begin{proof}
	For the first part of the lemma, first, note that since $c$ is a positive and continuous function on $[0,1]$, there exist $v,w>0$ such that $v < c(s) < w$ for all $s \in [0,1]$. By Condition \ref{cond:2nd order} there are real numbers $y_0 < x^{\ast}$ and $\tau > 0$ such that for all $y > y_0, \ n \in \N$ and $1 \leq i \leq n$, 
	\[ c(i/n) \Big\{  1-\frac{\tau}{v} A \Big( \frac{1}{1-F(y)} \Big) \Big\} < \frac{1-F_{ni}(y)}{1-F(y)} < c(i/n) \Big\{  1 + \frac{\tau}{v} A \Big( \frac{1}{1-F(y)} \Big) \Big\}. \] Set $y_n = F^{-1}(1-x/\bs)$ and $w_n = \frac{\tau}{v}A\big(\frac{1}{1-F(y_n)}\big) = \frac{\tau}{v} A \big(\frac{\bs}{x}\big)$. Thus, for $n$ large enough (such that $y_n > y_0$) we have for all $1 \leq i \leq n$,
	\[ \Big\{ Z_i \geq \frac{\bs}{c(i/n)x} (1-w_n)^{-1} \Big\} \subseteq \Big\{ Z_i \geq \frac{\bs}{x} \frac{1-F(y_n)}{1-F_{n,i}(y_n)} \Big\} \subseteq \Big\{ Z_i \geq \frac{\bs}{c(i/n)x} (1+w_n)^{-1} \Big\}. \]
	Since 
	\[ A_n = \sum_{i \in I'_{1+\lfloor \xi k'\rfloor}} \I \bigg( Z_i > \frac{\bs}{x} \frac{1-F(y_n)}{1-F_{n,i}(y_n)} \bigg), \] 
	this implies $B_n^- \leq A_n \leq B_n^+$, where
	\[ B_n^{\pm} = \sum_{i \in I'_{1+\lfloor \xi k'\rfloor}} \I \Big( Z_i > \frac{\bs}{c(i/n) x} (1 \pm w_n)^{-1} \Big). \]
	Next, we have
	\begin{align*}
	\Exp[|B_n^{\pm}-B_n|] 
	& \leq \sum_{i \in I'_{1+\lfloor \xi k'\rfloor}} \Exp \Big[ \Big| \I\Big( Z_i > \frac{\bs}{c(i/n)x} (1\pm w_n)^{-1} \Big) - \I\Big( Z_i > \frac{\bs}{c(i/n)x} \Big) \Big| \Big] \\
	& \leq \sum_{i \in I'_{1+\lfloor \xi k'\rfloor}} \Big\{ \Prob \Big( \frac{\bs}{c(i/n)x} (1\pm w_n)^{-1} < Z_i \leq \frac{\bs}{c(i/n)x} \Big) \\
	& \hspace{2.5cm} + \Prob \Big( \frac{\bs}{c(i/n)x} < Z_i \leq \frac{\bs}{c(i/n)x} (1\pm w_n)^{-1} \Big) \Big\}.
	\end{align*} 
	Let us consider the case with the plus-sign. Note that $w_n>0$. Recalling that $Z_i$ is Pareto-distributed the above expression reduces to 
	\begin{align*}
	w_n \frac{x}{\bs} \sum_{i \in I'_{1+\lfloor \xi k'\rfloor}} c(i/n) \le w_n x \|c\|_\infty,
	\end{align*}
	which converges to 0 since $w_n \to 0$ by Condition \ref{cond:2nd order}.
	The case with the minus-sign can be treated analogously. Hence, we have shown $B_n^{\pm} -B_n \to 0$ in $L_1(\Prob)$ as $ n \to \infty$. The assertion follows from $B_n^- \leq A_n \leq B_n^+$.

	For the second part of the lemma write 
	\begin{align*}
	\Exp [|B_n - C_n|] & \leq \sum_{i \in I'_{1+\lfloor \xi k'\rfloor}} \Exp \Big[ \Big| \I\Big( Z_i > \frac{\bs}{c(i/n)x} \Big) - \I\Big( Z_i > \frac{\bs}{c(\xi)x} \Big) \Big| \Big] \\
	& \leq \sum_{i \in I'_{1+\lfloor \xi k'\rfloor}} \Big\{ \Prob \Big( \frac{\bs}{c(i/n)x} < Z_i \leq \frac{\bs}{c(\xi)x} \Big) + \Prob \Big( \frac{\bs}{c(\xi)x} < Z_i \leq \frac{\bs}{c(i/n)x} \Big) \Big\} \\
	& = \frac{x}{\bs} \sum_{i \in I'_{1+\lfloor \xi k'\rfloor}} |c(i/n) - c(\xi)|,
	\end{align*}
	where the last equation is due to the fact that $Z_i$ is Pareto-distributed. Further, by Condition \ref{cond:sbias}, we have 
\begin{align*}
\frac{1}{\bs} \sum_{i \in I'_{1+\lfloor \xi k'\rfloor}} |c(i/n)-c(\xi)| 
&\le
\frac{K_c}{\bs} \sum_{i \in I'_{1+\lfloor \xi k'\rfloor}} |i/n-\xi|^{1/2} \\ %\label{c_mean} \\
&= 
\frac{K_c}{\bs} \sum_{j=1}^{\bs} \Big|\frac{\lfloor\xi k' \rfloor r + j}{n} - \xi \Big|^{1/2} \nonumber \\
&=
\frac{K_c}{\bs n^{1/2}} \sum_{j=1}^{\bs} \Big| (\lfloor\xi k'\rfloor - \xi k') \bs + j  \Big|^{1/2} 
\le \nonumber 
\frac{\sqrt 2K_c(\bs)^{1/2}}{n^{1/2}} = o(1)
\end{align*}
by Condition \ref{cond:blocksize}. Therefore, $B_n -C_n \to 0$ in $L_1(\Prob)$ as $n \to \infty$, which implies the second assertion.
\end{proof}

\begin{corollary} \label{cor1}
	Fix $\xi \in [0,1)$ and $x>0$. Under Condition \ref{cond:basic}-\ref{cond:sbias}, \ref{cond:2nd order} and \ref{cond:blocksize},
	\begin{align*}
	\Prob \Big( Z_i \leq \frac{\bs}{c(i/n) x} \ \textrm{ for all } \ i \in I'_{1+\lfloor \xi k'\rfloor} \Big) - \Prob \Big( Z_i \leq \frac{\bs}{c(\xi)x} \ \textrm{ for all } \ i \in I'_{1+\lfloor \xi k'\rfloor} \Big) = o(1),
	\end{align*}
	and
	\begin{multline*}
	 \Prob \Big( Z_i \leq \frac{1}{1-F_{n,i}} \big( F^{-1}(1-x/\bs) \big) \ \textrm{ for all } \ i \in I'_{1+\lfloor \xi k'\rfloor} \Big) \\
	  - \Prob \Big( Z_i \leq \frac{\bs}{c(i/n)x} \ \textrm{ for all } \ i \in I'_{1+\lfloor \xi k'\rfloor} \Big) = o(1), \quad n \to \infty.
	\end{multline*}
\end{corollary}

\begin{proof}
Note that
\begin{align*}
\Prob \Big( Z_i \leq \frac{1}{1-F_{n,i}} \big( F^{-1}(1-x/\bs) \big) \ \textrm{ for all } \ i \in I'_{1+\lfloor \xi k'\rfloor} \Big) &= \Prob(A_n=0), \\
\Prob \Big( Z_i \leq \frac{\bs}{c(i/n) x} \ \textrm{ for all } \ i \in I'_{1+\lfloor \xi k'\rfloor} \Big) &=\Prob(B_n=0), \\
\Prob \Big( Z_i \leq \frac{\bs}{c(\xi)x} \ \textrm{ for all } \ i \in I'_{1+\lfloor \xi k'\rfloor} \Big)  &= \Prob(C_n=0).
\end{align*}
Therefore,
\begin{align*}
|P(A_n=0)-P(B_n=0)| 
&=
P(A_n=0, B_n>0) + P(A_n>0, B_n=0)  \\
& \le 2 P(|A_n - B_n|>1/2) = o(1)
\end{align*}
by Lemma \ref{Bn-Cn}. And $|P(B_n=0)-P(C_n=0)|=o(1)$ can be shown analogously.
\end{proof}

%%%%%%%%%%%%%%%%%
%%%%%%%%%%%%%%%%%

%\bibliographystyle{chicago}
%\bibliography{biblio}

\end{document}